\documentclass[11pt]{article}
\usepackage[a4paper, left=2.5cm, right=2.5cm, top=3cm, bottom=2cm]{geometry}

\usepackage[utf8]{inputenc}  
\usepackage[english]{babel}
\usepackage{amsmath,amssymb,amsfonts,amsthm}
\usepackage{xcolor}
\usepackage{enumerate}
\usepackage{tikz,tkz-tab}
\usepackage{array}
\usepackage{hyperref}

\usepackage{graphicx}

\theoremstyle{plain}
\newtheorem{theorem}{Theorem}
\newtheorem{lemma}[theorem]{Lemma}
\newtheorem{proposition}[theorem]{Proposition}
\newtheorem{corollary}[theorem]{Corollary}

\theoremstyle{definition}
\newtheorem{definition}[theorem]{Definition}
\newtheorem{example}[theorem]{Example}

\theoremstyle{remark}
\newtheorem{remark}[theorem]{Remark}

\newcounter{a}
\newcommand{\C}{Cond.}
\newenvironment{cond}{\par \refstepcounter{a}  
 {\upshape \textbf{Cond~\thea}:}}{\par}

\newenvironment{condd}{\par
 {\upshape \textbf{Cond~\thea*}:}}{\par}

\newcommand{\encadre}[1]{\begin{tabular}{|c|} \hline #1 \\ \hline \end{tabular}}

\newcommand{\ber}{\mathcal{B}}
\newcommand{\app}[5]{\begin{array}{rccl} #1:&#2&\longrightarrow&#3\\ &#4&\longmapsto&#5\end{array}}  

\newcommand{\prob}[1]{\mathbb{P}\left(#1\right)}
\newcommand{\pw}[1]{\mathbb{P}_W\left(#1\right)}
\newcommand{\eqd}{\stackrel{(d)}{=}}
\newcommand{\ind}[1]{1_{#1}}
\newcommand{\diff}[1]{\mathrm{d} #1}
\newcommand{\borel}[1]{\mathcal{B}\left( #1 \right)}
\newcommand{\triang}[2]{\mathcal{T}_{#1} \left( #2 \right)}
\newcommand{\triangE}[1]{\mathcal{T}_{#1}}

\newcommand{\NN}{\mathbb{N}}
\newcommand{\ZZ}{\mathbb{Z}}
\newcommand{\Z}{\mathbb{Z}}
\newcommand{\Ze}{\mathbb{Z}_{\mathrm e}^2}
\newcommand{\RR}{\mathbb{R}}

\newcommand{\vu}{{\bf \textup{u}}}
\newcommand{\vv}{{\bf \textup{v}}}

\newcommand{\sym}{\text{sym}}

\newcommand{\haut}[1]{$\begin{array}{c}(#1)\\ \; \end{array}$}
\newcommand{\espace}{\rule{0pt}{6ex}}
\newcommand{\test}{\vspace{0.15cm}\rule{0pt}{5ex}\vspace{0.15cm}}

\usetikzlibrary{arrows}
\usetikzlibrary{decorations.markings}
\tikzset{->-/.style={decoration={
  markings,
  mark=at position .7 with {\arrow[scale=2]{>}}},postaction={decorate}}}
\tikzset{>>/.style={decoration={
  markings,
  mark=at position .7 with {\arrow[scale=5]{>}}},postaction={decorate}}}
\tikzset{>>>/.style={decoration={
  markings,
  mark=at position 1 with {\arrow[scale=5]{>}}},postaction={decorate}}}

\title{\vspace{-3cm} Probabilistic cellular automata with memory two:\\ invariant laws and multidirectional reversibility}
 
\author{J\'er\^ome Casse\footnote{NYU Shanghai, 1555 Century Avenue, Pudong, Shanghai, China 200112, \texttt{jerome.casse@nyu.edu}}, Ir\`ene Marcovici\footnote{Institut Elie Cartan de Lorraine, Universit\'e de Lorraine, Campus Scientifique, BP 239, 54506 Vandoeuvre-l\`es-Nancy Cedex, France, \texttt{irene.marcovici@univ-lorraine.fr}}}

\begin{document}

\maketitle

\begin{abstract} We focus on a family of one-dimensional probabilistic cellular automata with memory two: the dynamics is such that the value of a given cell at time $t+1$ is drawn according to a distribution which is a function of the states of its two nearest neighbours at time $t$, and of its own state at time $t-1$. Such PCA naturally arise in the study of some models coming from statistical physics ($8$-vertex model, directed animals and gaz models, TASEP, etc.). We give conditions for which the invariant measure has a product form or a Markovian form, and we prove an ergodicity result holding in that context. The stationary space-time diagrams of these PCA present different forms of reversibility. We describe and study extensively this phenomenon, which provides  families of Gibbs random fields on the square lattice having nice geometric and combinatorial properties. \end{abstract}

\smallskip

{\footnotesize
\tableofcontents
}

\newpage

Probabilistic cellular automata (PCA) are a class of random discrete dynamical sytems. They can be seen both as the synchronous counterparts of finite-range interacting particle systems, and as a generalization of deterministic cellular automata: time is discrete and at each time step, all the cells are updated independently in a random fashion, according to a distribution depending only on the states of a finite number of their neighbours.

In this article, we focus on a family of one-dimensional probabilistic cellular automata with \emph{memory two} (or \emph{order two}): the value of a given cell at time $t+1$ is drawn according to a distribution which is a function of the states of its two nearest neighbours at time $t$, and of its own state at time $t-1$. The space-time diagrams describing the evolution of the states can thus be represented on a two-dimensional grid.

We study the invariant measures of these PCA with memory two. In particular, we give necessary and sufficient conditions for which the invariant measure has a product form or a Markovian form, and we prove an ergodicity result holding in that context. We also show that when the parameters of the PCA satisfy some conditions, the stationary space-time diagram presents some multidirectional (quasi)-reversibility property: the random field has the same distribution as if we had iterated a PCA with memory two in another direction (the same PCA in the reversible case, or another PCA in the quasi-reversible case). This can be seen has a probabilistic extension of the notion of \emph{expansivity} for deterministic CA. For expansive CA, one can indeed reconstruct the whole space-time diagram from the knowledge of only one column. In the context of PCA with memory two, the criteria of quasi-reversibility that we obtain are reminiscent of the notion of \emph{permutivity} for determistic CA. Stationary space-time diagrams of PCA are known to be Gibbs random fields~\cite{GKLM, LMS}. The family of PCA that we will describe thus provide examples of Gibbs fields with i.i.d.\ lines on many directions and nice combinatorial and geometric properties.

\medskip

The first theoretical results on PCA and their invariant measures go back to the seventies~\cite{belyaev, kozlov, vasilyev}, and were then gathered in a survey which is still today a reference book~\cite{dobrushin}. In particular, it contains a detailed study of binary PCA with memory one with only two neighbours, including a presentation of the necessary and sufficient conditions that the four parameters defining the PCA must satisfy for having an invariant measure with a product form or a Markovian form. Some extensions and alternative proofs were proposed by Mairesse and Marcovici in a later article~\cite{mm_ihp}, together with a study of some properties of the random fields given by stationary space-time diagrams of PCA having a product form invariant measure (see also the survey on PCA of the same authors~\cite{mm_tcs}). The novelty was to highlight that these space-time diagrams are i.i.d.\ along many directions, and present a directional reversibility: they can also be seen as being obtained by iterating some PCA in another direction. Soon after, Casse and Marckert have proposed an indepth study of the Markovian case~\cite{CM15, Casse16}. Motivated by the study of the $8$-vertex model, Casse was then led to introduce a class of one-dimensional PCA with memory two, called \emph{triangular PCA}~\cite{Casse17}. 

\medskip

In the present article, we propose a comprehensive study of PCA with memory two having and invariant measure with a product form, and we show that their stationary space-time diagrams share some specificities. We first extend the notion of reversibility and quasi-reversibility to take into account other symmetries than the time reversal and, in a second time, we characterize PCA with an invariant product measure that are reversible or quasi-reversible.
Even if most one-dimensional positive-rates PCA are usually expected to be ergodic, the ergodicity of PCA is known to be a difficult problem, algorithmically undecidable~\cite{dobrushin, bmm_aap}. In Section~\ref{sec:cond_ergo}, after characterizing positive-rates PCA having a product invariant measure, we prove that these PCA are ergodic (Theorem~\ref{theo:ergo}). A novelty of our work is also to display some PCA for which the invariant measure has neither a product form nor a Markovian one, but for which the finite-dimensional marginals can be exactly computed (Theorems \ref{theo:cool} and \ref{theo:cool_hzmc}). In Section~\ref{sec:markov}, we study PCA having Markov invariant measures. Section~\ref{sec:sp} is then devoted to the presentation of some applications of our models and results to statistical physics ($8$-vertex model, directed animals and gaz models, TASEP, etc.). In particular, we introduce an extension of the TASEP model, in which the probability for a particle to move depends on the distance of the previous particle and of its speed. It can also be seen as a traffic flow model, more realistic than the classical TASEP model. 
Finally, we give on one side a more explicit description of (quasi-)reversible binary PCA (Section~\ref{sec:binary}), and on the other side, we provide some extensions to general sets of symbols (Section~\ref{sec:general}).

When describing the family of PCA presenting some given directional reversibility or quasi-reversibility property, for each family of PCA involved, we give the conditions that the parameters of the PCA must satisfy in order to present that behaviour, and we provide the dimension of the corresponding submanifold of the parameter space, see Table~\ref{table:pres}. Our purpose is to show that despite their specificity, these PCA build up rich classes, and we set out the detail of the computations in the last section.

\section{Definitions and presentation of the results}\label{sec:introdef}

\subsection{Introductory example}

In this paragraph, we give a first introduction to PCA with memory two, using an example motivated by the study of the $8$-vertex model~\cite{Casse17}. We present some properties of the stationary space-time diagram of this PCA: although it is a non-trivial random field, it is made of lines of i.i.d.\ random variables, and it is reversible. In the rest of the article, we will study exhaustively the families of PCA having an analogous behaviour.

Let us set $\Ze = \{(i,t) \in \ZZ^2 : i+t \equiv 0 \mod 2\}$, and introduce the notations: $\Z_t=2\Z$ if $t\in 2\Z$, and $\Z_t=2\Z+1$ if $t\in 2\Z+1$, so that the grid $\Ze$ can be seen as the union on $t\in\ZZ$ of the points $\{(i,t) : i\in\Z_t\}$, that will contain the information on the state of the system at time~$t$. Note that one can scroll the positions corresponding to two consecutive steps of time along an horizontal zigzag line: $\ldots(i,t), (i+1,t+1), (i+2,t), (i+3,t+1) \ldots $ This will explain the terminology introduced later.

We now define a PCA dynamics on the \emph{alphabet} $S=\{0,1\}$, which, through a recoding, can be shown to be closely related to the $8$-vertex model (see Section \ref{sec:sp} for details). The configuration $\eta_t$ at a given time $t\in\Z$ is an element of $S^{\Z_t}$, and the evolution is as follows. Let us denote by $\ber(q)$ the Bernoulli measure $q\delta_1+(1-q)\delta_0$. Given the configurations $\eta_t$ and $\eta_{t-1}$ at times $t$ and $t-1$, the configuration $\eta_{t+1}$ at time $t+1$ is obtained by updating each site $i \in\Z_{t+1}$ simultaneously and independently, according to the distribution $T(\eta_t(i-1),\eta_{t-1}(i),\eta_{t}(i+1);\cdot)$, where
\begin{align*}
&T(0,0,1; \cdot)=T(1,0,0; \cdot)=\ber(q),\\
&T(0,1,1; \cdot)=T(1,1,0; \cdot)=\ber(1-q)\\
&T(0,1,0; \cdot)=T(1,1,1; \cdot)=\ber(r),\\
&T(1,0,1; \cdot)=T(0,0,0; \cdot)=\ber(1-r).
\end{align*}

As a special case, for $q=r$, we have: $T(a,b,c; \cdot)=q\,\delta_{a+b+c \mod 2} + (1-q)\, \delta_{a+b+c+1 \mod 2}$, so that the new state is equal to $a+b+c \mod 2$ with probability $q$, and to $a+b+c+1 \mod 2$ with probability $1-q$. Fig.~\ref{fig:def} shows how $\eta_{t+1}$ is computed from $\eta_{t}$ and $\eta_{t-1}$, illustrating the progress of the Markov chain.

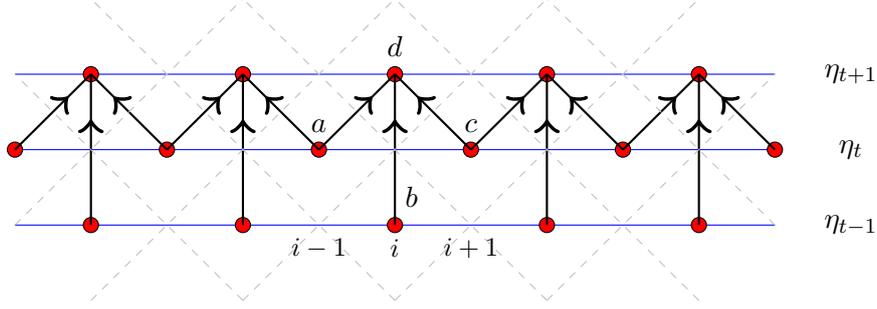
\begin{figure}
\begin{center}
\begin{tikzpicture}
\foreach \y in {0,1,2} \draw[blue] (-5,\y) -- (5,\y) ;
\foreach \x in {-4,-2,...,4} \foreach \y in {0,2} {\draw [fill=red] (\x,\y) circle [radius=0.1] ;}
\foreach \x in {-5,-3,...,5} \foreach \y in {1} {\draw [fill=red] (\x,\y) circle [radius=0.1] ;}
\foreach \x in {-4,-2,...,4} \foreach \y in {2} {\draw[->-,thick] (\x,\y-2) -- (\x,\y) ;}
\foreach \x in {-4,-2,...,4} \foreach \y in {2} {\draw[->-,thick] (\x+1,\y-1) -- (\x,\y) ;}
\foreach \x in {-4,-2,...,4} \foreach \y in {2} {\draw[->-,thick] (\x-1,\y-1) -- (\x,\y) ;}
\node at (6,2) {$\eta_{t+1}$};
\node at (6,1) {$\eta_{t}$};
\node at (6,0) {$\eta_{t-1}$};
\node[above] at (-1,1.1) {$a$};
\node[above right] at (0,0.1) {$b$};
\node[above] at (1,1.1) {$c$};
\node[above] at (0,2.1) {$d$};
\node at (-1,-0.3) {\small{$i-1$}};
\node at (0,-0.3) {\small{$i$}};
\node at (1,-0.3) {\small{$i+1$}};
\foreach \x in {-4,-2,...,0} {\draw [color=gray!50,dashed] (\x,-1) -- (\x+4,3) ;}
\draw [color=gray!50,dashed] (-5,0) -- (-5+3,3) ;
\draw [color=gray!50,dashed] (-2,-1) -- (-2-3,2) ;
\foreach \x in{0,2,...,4} {\draw [color=gray!50,dashed] (\x,-1) -- (\x-4,3) ;}
\draw [color=gray!50,dashed] (5,0) -- (5-3,3) ;
\draw [color=gray!50,dashed] (2,-1) -- (2+3,2) ;
\end{tikzpicture}
\end{center}
\caption{Illustration of the way $\eta_{t+1}$ is obtained from $\eta_{t}$ and $\eta_{t-1}$, using the transition kernel~$T$. The value $\eta_{t+1}(i)$ is equal to $d$ with probability $T(a,b,c;d)$, and conditionnally to $\eta_{t}$ and $\eta_{t-1}$, the values $(\eta_{t+1}(i))_{i \in \Z_{t+1}}$ are independent.}\label{fig:def}
\end{figure}

Let us assume that initially, $(\eta_0,\eta_1)$ is distributed according to the uniform product measure $\lambda=\ber(1/2)^{\otimes \Z_0}\otimes\ber(1/2)^{\otimes \Z_1}$. Then, we can show that for any $t\in \NN$, $(\eta_t,\eta_{t+1})$ is also distributed according to $\lambda$. We will say that the PCA has an invariant \emph{Horizontal Zigzag Product Measure}. By stationarity, we can then extend the space-time diagram to a random field with values in $S^{\Ze}$. The study of the space-time diagram shows that it has some peculiar properties, which we will precise in the next sections. In particular, it is quasi-reversible: if we reverse the direction of time, the random field corresponds to the stationary space-time diagram of another PCA. Furthermore, the PCA is ergodic: whatever the distribution of $(\eta_0,\eta_1)$, the distribution of $(\eta_t,\eta_{t+1})$ converges weakly to $\lambda$ (meaning that for any $n\in\NN$, the restriction of $(\eta_t,\eta_{t+1})$ to the cells of abscissa ranging between $-n$ and $n$ converges to a uniform product measure). For $q=r$, the stationary space-time diagram presents even more symmetries and directional reversibilities: it has the same distribution as if we had iterated the PCA in any other of the four cardinal directions. In addition, any straight line drawn along the space-time diagram is made of i.i.d.\ random variables, see Fig.~\ref{fig:std} for an illustration.

In the following of the article, we will show that this PCA belongs to a more general class of PCA that are all ergodic and for which the stationary space-time diagram share specific properties (independence, directional reversibility).

\begin{figure}
\begin{center}
\begin{tabular}{cc}
\includegraphics[scale=0.27]{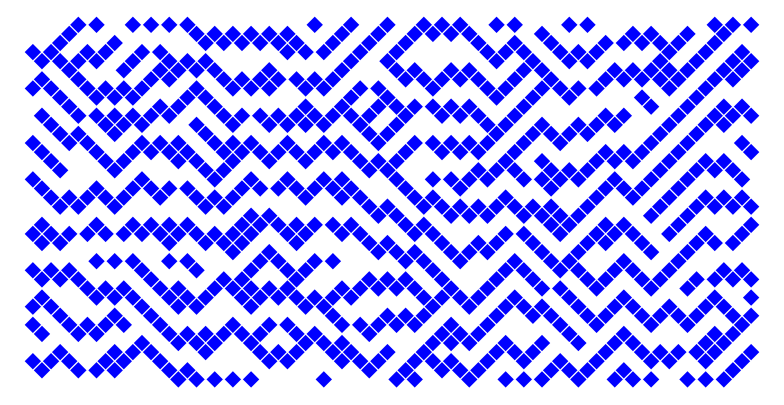}&
\includegraphics[scale=0.27]{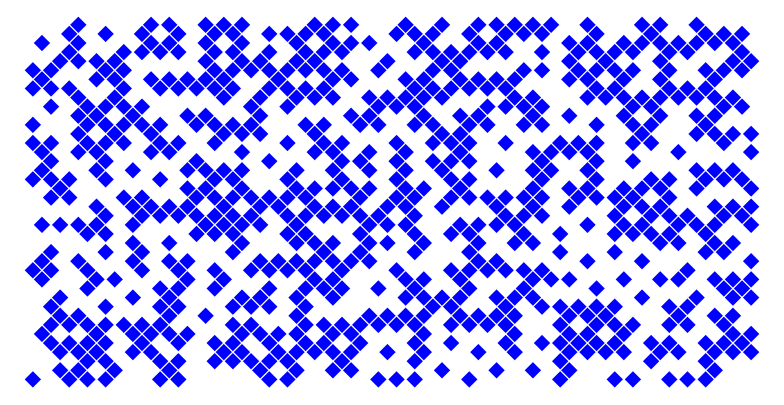}\\
$q=0.9$ and $r=0.2$ & $q=r=0.2$
\end{tabular}
\caption{Examples of portions of stationary space-time diagrams of the $8$-vertex PCA, for different values of the parameters. Cells in state $1$ are represented in blue, and cells in state $0$ are white.}\label{fig:std}
\end{center}
\end{figure}

\subsection{PCA with memory two and their invariant measures}

In this article, we will only consider PCA with memory two for which the value of a given cell at time $t+1$ is drawn according to a distribution which is a function of the states of its two nearest neighbours at time $t$, and of its own state at time $t-1$. We thus introduce the following definition of transition kernel and of PCA with memory two.

\begin{definition} Let $S$ be a finite set, called the \emph{alphabet}. A \emph{transition kernel} is a function $T$ that maps any $(a,b,c)\in S^3$ to a probability distribution on $S$. We denote by $T(a,b,c;\cdot)$ the distribution on $S$ which is the image of the triplet $(a,b,c)\in S^3$, so that: $\forall d\in S$, $T(a,b,c;d)\in [0,1]$ and $\sum_{s\in S} T(a,b,c;d)$=1.

A \emph{probabilistic cellular automaton (PCA) with memory two} of transition kernel $T$ is a Markov chain of order two $(\eta_t)_{t\geq 0}$ such that $\eta_t$ has values in $S^{\Z_t}$, and conditionnally to $\eta_t$ and $\eta_{t-1}$, for any $i \in \Z_{t+1}$, $\eta_{t+1}(i)$ is distributed according to $T(\eta_t(i-1),\eta_{t-1}(i),\eta_{t}(i+1);\cdot)$, independently for different $i \in\Z_{t+1}$. 

We say that a PCA has \emph{positive rates} if its transition kernel $T$ is such that $\forall a,b,c,d\in S$, $T(a,b,c;d)>0$.
\end{definition}

By definition, if $(\eta_t(i-1),\eta_{t-1}(i),\eta_{t}(i+1))=(a,b,c)$, then $\eta_{t+1}(i)$ is equal to $d\in S$ with probability $T(a,b,c;d)$, see Fig.~\ref{fig:def} for an illustration.

\medskip

Let us introduce the two vectors $\vu=(-1,1)$ and $\vv=(1,1)$ of $\Ze$.

Let $\mu$ be a distribution on $S^{\Z_{t-1}}\times S^{\Z_{t}}$. We denote by $\sigma_v(\mu)$ the distribution on $S^{\Z_{t}}\times S^{\Z_{t+1}}$ which is the image of $\mu$ by the application:
$$\sigma_{\vv}:((x_k)_{k\in \Z_{t-1}},(y_l)_{l\in \Z_{t}}) \rightarrow ((x_{k-1})_{k\in \Z_{t}},(y_{l-1})_{l\in \Z_{t+1}}).$$

When considering the distribution $\mu$ as living on the two consecutive horizontal lines of the lattice $\Ze$, corresponding to times $t-1$ and $t$, the distribution $\sigma_{\vv}(\mu)$ thus corresponds to shifting $\mu$ by a vector $\vv=(1,1)$.  Similarly, we denote by $\sigma_{\vv-\vu}(\mu)$ the distribution on $S^{\Z_{t-1}}\times S^{\Z_{t}}$ which is the image of $\mu$ by the application:
$\sigma_{\vv-\vu}:((x_k)_{k\in \Z_{t-1}},(y_l)_{l\in \Z_{t}}) \rightarrow ((x_{k-2})_{k\in \Z_{t-1}},(y_{l})_{l\in \Z_{t}}).$

In that context of PCA with memory two, we introduce the following definitions.

\begin{definition} Let $\mu$ be a probability distribution on $S^{\Z_{0}}\times S^{\Z_{1}}$.

The distribution $\mu$ is said to be \emph{shift-invariant} if $\sigma_{\vv-\vu}(\mu)=\mu$.

The distribution $\mu$ on $S^{\Z_{0}}\times S^{\Z_{1}}$ is an \emph{invariant distribution} of a PCA with memory two if the PCA dynamics is such that:
$(\eta_{0},\eta_{1}) \sim \mu \implies (\eta_{1},\eta_{2}) \sim \sigma_{\vv}(\mu).$
\end{definition}

By a standard compactness argument, one can prove that any PCA has at least one invariant distribution which is shift-invariant. In this article, we will focus on such invariant distributions. Note that if $\mu$ is both a shift-invariant measure and an invariant distribution of a PCA, then we also have $(\eta_{0},\eta_{1}) \sim \mu \implies (\eta_{1},\eta_{2}) \sim \sigma_{\vu}(\mu).$

\begin{definition} Let $p$ be a distribution on $S$. The $p$-HZPM (for \emph{Horizontal Zigzag Product Measure}) on $S^{\Z_{t-1}}\times S^{\Z_{t}}$ is the distribution $\pi_p=\ber(p)^{\otimes\Z_{t-1}}\otimes \ber(p)^{\otimes\Z_t}$. \end{definition}

Observe that we do not specify $t$ in the notation, since there will be no possible confusion. By definition, $\pi_p$ is invariant for a PCA if:
$$(\eta_{t-1},\eta_t) \sim \pi_p \implies (\eta_{t},\eta_{t+1}) \sim \pi_p.$$

\subsection{Stationary space-time diagrams and directional (quasi-)reversibility}\label{ssec:sstd}

Let $A$ be a PCA and $\mu$ one of its invariant measures. Let $G_n = (\eta_t(i) : t \in \{-n,\dots,n\}, i \in \ZZ_t)$ be a space-time diagram of $A$ under its invariant measure $\mu$, from time $t=-n$ to $t=n$. Then $(G_n)_{n\geq 0}$ induces a sequence of compatible measures on $\Ze$ and, by Kolmogorov extension theorem, defines a unique measure on $\Ze$, that we denote by $G(A,\mu)$.

\begin{definition}\label{def:std} Let $A$ be a PCA and $\mu$ one of its invariant distributions which is shift-invariant. A random field $(\eta_t(i) : t \in \ZZ, i \in \ZZ_t)$ which is distributed according to $G(A,\mu)$ is called a \emph{stationary space-time diagram} of $A$ taken under $\mu$. 
\end{definition}

We denote by $D_4$ the dihedral group of order $8$, that is, the group of symmetries of the square. We denote by $r$ the rotation of angle $\pi/2$ and by $h$ the horizontal reflection. We denote the vertical reflection by $v=r^2\circ h$, and the identity by $id$. For a subset $E$ of $D_4$, we denote by $<E>$ the subgroup of $D_4$ generated by the elements of $E$.

\begin{definition}\label{def:rev} Let $A$ be a positive-rates PCA, and let $\mu$ be an invariant measure of $A$ which is shift-invariant. For $g \in D_4$, we say that $(A,\mu)$ is \emph{$g$-quasi-reversible}, if there exists a PCA $A_g$ and a measure $\mu_g$ such that $G(A,\mu) \eqd g^{-1} \circ G(A_g,\mu_g)$. In this case, the pair $(A_g,\mu_g)$ is the $g$-reverse of $(A,\mu)$. If, moreover, $(A_g,\mu_g)=(A,\mu)$, then $(A,\mu)$ is said to be \emph{$g$-reversible}. 

For a subset $E$ of $D_4$, we say that $A$ is $E$-quasi-reversible (resp. $E$-reversible) if it is $g$-quasi-reversible (resp. $g$-reversible) for any $g\in E$. 
\end{definition}

Classical definitions of quasi-reversibility and reversibility of PCA correspond to time-reversal, that are, $h$-quasi-reversibility and $h$-reversibility. Geometrically, the stationary space-time diagram $(A,\mu)$ is $g$-quasi-reversible if after the action of the isometry $g$, the random field has the same distribution as if we had iterated another PCA $A_g$ (or the same PCA $A$, in the reversible case). In particular, if $(A,\mu)$ is $r$-quasi-reversible (resp. $r^2$, $r^3$), it means that even if the space-time diagram is originally defined by an iteration of the PCA $A$ towards the North, it can also be described as the stationary space-time diagram of another PCA directed to the East (resp. to the South, to the West).

Table~\ref{table:pres} presents a summary of the results that will be proven in the next sections, concerning the stationary space-time diagrams of PCA having an invariant HZPM. For each possible \mbox{(quasi-)}reversibility behaviour, we give the conditions that the parameters of the PCA must satisfy (see Section~\ref{sec:rqrPCA} for details), and provide the number of degrees of freedom left by these equations, that is, the dimension of the corresponding submanifold of the parameter space (see Section~\ref{sec:dim}).

\begin{table}
\hspace{-0.5cm}\begin{tabular}{|c|c|c|}
\hline
Conditions & Property & Dimension of the submanifold \\
on the parameters & of the PCA & (number of degrees of freedom)\\
  \hline \hline
  \begin{tabular}{c}
  \test
	\encadre{
      \C~\ref{cond-pm}: $\forall a,c,d\in S$,\\ $p(d) = \sum_{b \in S} p(b) T(a,b,c;d)$
    }
  \end{tabular}
           & \begin{tabular}{c}
               HZPM invariant\\
               $\{r^2,h\}$-quasi-reversible
             \end{tabular}
           & $n^2(n-1)^2$  \\ 
  \hline 
  \begin{tabular}{c}
    \C~\ref{cond-pm} + \\
    \test
    \encadre{\C~\ref{cond:r1rev}: $\forall a,b,d \in S$,\\ $p(d) = \sum_{c \in S} p(c) T(a,b,c;d)$}
  \end{tabular}
           & $r$-quasi-reversible & $n(n-1)^3$ \\
\hline 
  \begin{tabular}{c}
    \C~\ref{cond-pm} + \\
      \test\encadre{\C~\ref{cond:r3rev}: $\forall b,c,d \in S$,\\ $p(d) = \sum_{a \in S} p(a) T(a,b,c;d)$}
  \end{tabular}
           & $r^{-1}$-quasi-reversible & $n(n-1)^3$ \\
  \hline
  \begin{tabular}{c}
   \test \C~\ref{cond-pm} + \C~\ref{cond:r1rev} + \C~\ref{cond:r3rev}
  \end{tabular}
           & $D_4$-quasi-reversible & $(n-1)^4$ \\
  \hline \hline
  \begin{tabular}{c}
    \C~\ref{cond-pm} + \\
      \test\encadre{$\forall a,b,c,d \in S$,\\
    $T(a,b,c;d) = T(c,b,a;d)$}
  \end{tabular}
           & $v$-reversible & $\displaystyle \frac{(n-1)^2n(n+1)}{2}$ \\
  \hline
  \begin{tabular}{c}
    \C~\ref{cond-pm} + \\
      \test\encadre{$\forall a,b,c,d \in S$,\\
    $p(b) T(a,b,c;d) = p(d) T(c,d,a;b)$}
  \end{tabular}
           & $r^2$-reversible & $\displaystyle \frac{(n-1)^2n(n+1)}{2}$ \\
  \hline
  \begin{tabular}{c}
    \C~\ref{cond-pm} + \\
      \test\encadre{$\forall a,b,c,d \in S$,\\
    $p(b) T(a,b,c;d) = p(d) T(a,d,c;b)$}
  \end{tabular}
           & $h$-reversible & $\displaystyle \frac{n^3(n-1)}{2}$ \\
  \hline

  \begin{tabular}{c}
    \C~\ref{cond-pm} + \\
      \test\encadre{$\forall a,b,c,d \in S$,\\
    $T(a,b,c;d) = T(c,b,a;d)$ and \\
    $p(b) T(a,b,c;d) = p(d) T(c,d,a;b)$}
  \end{tabular}
           & $<r^2,v>$-reversible & $\displaystyle \frac{(n-1)n^2(n+1)}{4}$ \\
  \hline

  \begin{tabular}{c}
    \C~\ref{cond-pm} + \\
      \test\encadre{$\forall a,b,c,d \in S$,\\
    $p(a) T(a,b,c;d) = p(d) T(b,c,d;a)$}
  \end{tabular}
           & $<r>$-reversible & $\displaystyle \frac{n(n-1)(n^2-3n+4)}{4}$ \\
  \hline

  \begin{tabular}{c}
    \C~\ref{cond-pm} + \\
    \test \encadre{$\forall a,b,c,d \in S$,\\
    $p(a) T(a,b,c;d) = p(d) T(d,c,b;a)$}
  \end{tabular}
           & $<r \circ v>$-reversible & $\displaystyle \frac{(n-1)^2 (n^2-2n+2)}{2}$ \\
  \hline
  \begin{tabular}{c}
    \C~\ref{cond-pm} + \\
      \test\encadre{$\forall a,b,c,d \in S$,\\
    $p(a) T(a,b,c;d) = p(d) T(b,c,d;a)$ and\\
    $T(a,b,c;d) = T(c,b,a;d)$}
  \end{tabular}
           & $D_4$-reversible & $\displaystyle \frac{n(n-1)(n^2-n+2)}{8}$ \\
  \hline  
\end{tabular}
\caption{Summary of the characterization of (quasi-)reversible PCA. We denote by $n$ the cardinal of the alphabet $S$.}\label{table:pres}
\end{table}

\section{Invariant product measures and ergodicity}\label{sec:cond_ergo}

To start with, next theorem gives a characterization of PCA with memory two having an HZPM invariant measure.

\begin{theorem} \label{theo:gen0}
Let $A$ be a positive-rates PCA with transition kernel $T$, and let $p$ be a probability
vector on $S$. The HZPM $\pi_p$ is invariant for $A$ if and only if
\begin{cond}\label{cond-pm}
for any $a,c,d\in S$, $p(d) = \sum_{b \in S} p(b) T(a,b,c;d)$.
\end{cond}
\end{theorem}

Note that since $A$ has positive rates, if $\pi_p$ is invariant for $A$, then the vector $p$ has to be positive.

\begin{corollary} \label{theo:gen}
Let $A$ be a positive-rates PCA with transition kernel $T$. The PCA $A$ has an invariant HZPM if and only if for any $a,c \in S$, the left eigenspace $E_{a,c}$ of matrices $\left(T(a,b,c;d)\right)_{b,d \in S}$ related to the eigenvalue $1$ is the same. In that case, the invariant HZPM is unique: it is the measure $\pi_p$ defined by the unique vector $p$ such that $E_{a,c} = \text{Vect}(p)$ for all $a,c\in S$ and $\sum_{b \in S} p(b) =1$.
\end{corollary}

\begin{proof} Let $p$ be a positive vector such that $\pi_p$ is invariant by $A$ and assume that $(\eta_{t-1},\eta_t) \sim \pi_p$. Then, on the one hand, since $\pi_p$ is invariant by $A$, we have:
\begin{equation*}\prob{\eta_{t}(i-1)=a , \eta_{t+1}(i) = d, \eta_{t}(i+1) = c} = p(a) p(c) p(d).
\end{equation*}
And on the other hand, by definition of the PCA,
\begin{equation*}\prob{\eta_{t}(i-1)=a , \eta_{t+1}(i) = d, \eta_{t}(i+1) = c} = \sum_{b\in S} p(a) p(b) p(c) T(a,b,c;d).
\end{equation*}
Cond.~\ref{cond-pm} follows.

Conversely, assume that \C~\ref{cond-pm} is satisfied, and that $(\eta_{t-1},\eta_t) \sim \pi_p$. 
For some given choice of $n\in\Z_{t}$, let us denote: $X_i=\eta_{t-1}(n+1+2i), Y_i=\eta_{t}(n+2i), Z_i=\eta_{t+1}(n+1+2i),$ for $i\in \Z$, see Fig.~\ref{fig:hzpm_proof} for an illustration. Then, for any $k\geq 1$, we have:
\begin{align*}
&\prob{(Y_i)_{0\leq i\leq k}= (y_i)_{0\leq i\leq k}, (Z_i)_{0\leq i\leq k-1}= (z_i)_{0\leq i\leq k-1}}\\
& \quad = \sum_{(x_i : 0\leq i\leq k-1)} \prob{(X_i)_{0\leq i\leq k}=(x_i)_{0\leq i\leq k-1}, (Y_i)_{0\leq i\leq k}= (y_i)_{0\leq i\leq k}} \prod_{i=0}^{k-1} T(y_i,x_{i},y_{i+1};z_i)\\
& \quad = \sum_{(x_i : 0\leq i\leq k-1)} \prod_{i=0}^{k-1} p(x_i) \prod_{i=0}^{k} p(y_i) \prod_{i=0}^{k-1} T(y_i,x_{i},y_{i+1};z_i) \\
& \quad = \prod_{i=0}^{k} p(y_i) \prod_{i=0}^{k-1} \sum_{x_i\in S} p(x_i)T(y_i,x_{i},y_{i+1};z_i) \\
& \quad = \prod_{i=0}^{k} p(y_i) \prod_{i=0}^{k-1} p(z_i) \mbox{ by Cond.~\ref{cond-pm},}
\end{align*}
thus, $\pi_p$ is invariant by $A$.

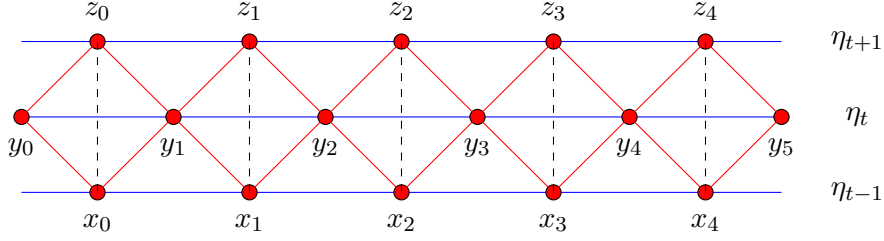
\begin{figure}
\begin{center}
\begin{tikzpicture}
\foreach \y in {0,1,2} \draw[blue] (-5,\y) -- (5,\y) ;
\foreach \x in {-4,-2,...,4} \foreach \y in {0,2} {\draw [fill=red] (\x,\y) circle [radius=0.1] ;}
\foreach \x in {-5,-3,...,5} \foreach \y in {1} {\draw [fill=red] (\x,\y) circle [radius=0.1] ;}
\foreach \x in {-4,-2,...,4} \foreach \y in {2} {\draw[dashed] (\x,\y-2) -- (\x,\y) ;}
\foreach \x in {-3,-1,...,3} \foreach \y in {1} {\draw[red] (\x-1,\y-1) -- (\x,\y) -- (\x+1,\y-1); \draw[red] (\x-1,\y+1) -- (\x,\y) -- (\x+1,\y+1);}
\draw[red] (-5,1) -- (-4,2) ; \draw[red] (-5,1) -- (-4,0) ; \draw[red] (5,1) -- (4,0) ; \draw[red] (5,1) -- (4,2) ; 
\foreach \x in {0,1,...,4} \foreach \y in {0} {\node at (2*\x-4,\y-0.4) {$x_{\x}$};}
\foreach \x in {0,1,...,4} \foreach \y in {2} {\node at (2*\x-4,\y+0.4) {$z_{\x}$};}
\foreach \x in {0,1,...,5} \foreach \y in {1} {\node at (2*\x-4-1,\y-0.4) {$y_{\x}$};}
\node at (6,2) {$\eta_{t+1}$};
\node at (6,1) {$\eta_{t}$};
\node at (6,0) {$\eta_{t-1}$};
\end{tikzpicture}
\end{center}\caption{Illustration of the proof of Theorem~\ref{theo:gen0}.}\label{fig:hzpm_proof}
\end{figure}
\end{proof}

\begin{theorem} \label{theo:ergo}
Let $A$ be a PCA with transition kernel $T$ and positive rates, satisfying \C~\ref{cond-pm}. Then, $A$ is ergodic. Precisely, whatever the distribution of $(\eta_0, \eta_1)$ is, the distribution of $(\eta_t,\eta_{t+1})$ converges (weakly) to $\pi_p$.
\end{theorem}

\begin{proof} The proof we propose is inspired from~\cite{vasilyev}, see also~\cite{dobrushin} and~\cite{marco_cie}. Let us fix some boundary conditions $(\ell, r)\in S^2$. Then, for any $k\geq 0$, the transition kernel $T$ induces a Markov chain on $S^{2k+1}$, such that the probability of a transition from the sequence $(a_0,b_0,a_1,b_1,\ldots, b_{k-1},a_{k})\in S^{2k+1}$ to a sequence $(a'_0,b'_0,a'_1,b'_1,\ldots, b'_{k-1},a'_{k})\in S^{2k+1}$ is given by:
\begin{align*}
&P_k^{(\ell,r)}((a_0,b_0,a_1,b_1,\ldots, b_{k-1},a_{k}), (a'_0,b'_0,a'_1,b'_1,\ldots, b'_{k-1},a'_{k}))\\
&\quad =T(\ell, a_0, b_0;a'_0)T(a'_0,b_0,a_1;b'_0)T(b_0, a_1, b_1;a'_1)\cdots T(b_{k-1},a_{k},r;a'_k)\\
&\quad = T(\ell, a_0, b_0;a'_0)T(b_{k-1},a_{k},r;a'_k)\prod_{i=1}^{k-1}T(b_{i-1}, a_i, b_{i+1};a'_i)\prod_{i=0}^{k-1}T(a'_{i}, b_i, a'_{i+1};b'_i).
\end{align*}
We refer to Fig.~\ref{fig:ergo_proof} for an illustration. Let us observe that the $p$-HZPM $\pi^k_p=\ber(p)^{\otimes 2k+1}$ is left invariant by this Markov chain. This is an easy consequence from \C~\ref{cond-pm}. For any $(\ell,r)\in S^2$, the transition kernel $P^{(\ell,r)}$ is positive. Therefore, there exists $\theta_{(\ell,r)}<1$ such that for any probability distributions $\nu, \nu'$ on $S^{2k+1}$, we have 
$$||P_k^{(\ell,r)}\nu-P_k^{(\ell,r)}\nu'||_1\leq \theta_k^{(\ell,r)} ||\nu-\nu'||_1,$$
the above inequality being true in particular for $\theta_k^{(\ell,r)}=1-\varepsilon_k^{(\ell,r)}$, where $$\varepsilon_k^{(\ell,r)}=\min\{P_k^{(\ell,r)}(x,y) \, : \; x,y\in S^{2k+1}\}.$$
Let us set $\theta_k=\max\{\theta_k^{(\ell,r)} \, : \; (\ell,r)\in S^2\}$. It follows that for any sequence $(\ell_t,r_t)_{t\geq 0}$ of elements of $S^2$, we have:
$$||P_k^{(\ell_{t-1},r_{t-1})}\cdots P_k^{(\ell_{1},r_{1})}P_k^{(\ell_{0},r_{0})}\nu-P_k^{(\ell_{t-1},r_{t-1})}\cdots P_k^{(\ell_{1},r_{1})}P_k^{(\ell_{0},r_{0})}\nu'||_1\leq \theta^t ||\nu-\nu'||_1.$$
In particular, for $\nu'=\pi^k_p$, we obtain that for any distribution $\nu$ on $S^{2k+1}$ and any sequence $(\ell_t,r_t)_{t\geq 0}$ of elements of $S^2$, we have:
$$||P_k^{(\ell_{t-1},r_{t-1})}\cdots P_k^{(\ell_{1},r_{1})}P_k^{(\ell_{0},r_{0})}\nu-\pi^k_p||_1\leq 2\theta^t.$$
Let now $\mu$ be a distribution on $S^{\Z_0\cup\Z_1}$, and let $k\geq 0$. When iterating $A$, the distribution $\mu$ induces a random sequence of symbols $\ell_t=\eta_{2t+1}(-(2k+1))$ and $r_t=\eta_{2t+1}(2k+1)$. Let us denote by $\nu_t$ the distribution of the sequence $(\eta_{2t}(-2k),\eta_{2t+1}(-2k+1),\eta_{2t}(-2k+2),\ldots,\eta_{2t}(2k-2),\eta_{2t+1}(2k-1),\eta_{2t}(2k))$, and let $\pi^{2k}_p=\ber(p)^{\otimes 4k+1}$. We have:
$$\forall t\geq 0, \; ||\nu_t-\pi^{2k}_p||_1\leq  \max_{(\ell_0,r_0)\ldots (\ell_{t-1},r_{t-1})\in S^2} ||P_{2k}^{(\ell_{t-1},r_{t-1})}\ldots P_{2k}^{(\ell_{1},r_{1})}P_{2k}^{(\ell_{0},r_{0})}\nu_0-\pi^{2k}_p||_1 \leq 2\theta^t.$$ This concludes the proof.

\begin{figure}
\begin{center}
\begin{tikzpicture}
\foreach \y in {0,1,2, 3} \draw[blue] (-5,\y) -- (5,\y) ;
\foreach \x in {-4,-2,...,4} \foreach \y in {0,2} {\draw [fill=red] (\x,\y) circle [radius=0.1] ;}
\foreach \x in {-5,-3,...,5} \foreach \y in {1,3} {\draw [fill=red] (\x,\y) circle [radius=0.1] ;}
\foreach \x in {-4,-2,...,4} \foreach \y in {2} {\draw[dashed] (\x,\y-2) -- (\x,\y) ;}
\foreach \x in {-3,-1,...,3} \foreach \y in {3} {\draw[dashed] (\x,\y-2) -- (\x,\y) ;}
\foreach \x in {-4,-2,...,4} \foreach \y in {2} {\draw[dashed] (\x+1,\y-1) -- (\x,\y) ;}
\foreach \x in {-4,-2,...,4} \foreach \y in {2} {\draw[dashed] (\x-1,\y-1) -- (\x,\y) ;}
\foreach \x in {-3,-1,...,3} \foreach \y in {1,3} {\draw[red] (\x-1,\y-1) -- (\x,\y) ;}
\foreach \x in {-3,-1,...,3} \foreach \y in {1,3} {\draw[red] (\x+1,\y-1) -- (\x,\y) ;}
\node at (5.5,1) {$r$};
\node at (-5.5,1) {$\ell$};
\foreach \x in {0,1,...,4} \foreach \y in {0} {\node at (2*\x-4,\y-0.4) {$a_{\x}$};}
\foreach \x in {0,1,...,4} \foreach \y in {2} {\node at (2*\x-4,\y+0.4) {$a'_{\x}$};}
\foreach \x in {0,1,...,3} \foreach \y in {1} {\node at (2*\x-3,\y-0.4) {$b_{\x}$};}
\foreach \x in {0,1,...,3} \foreach \y in {3} {\node at (2*\x-3,\y+0.4) {$b'_{\x}$};}
\end{tikzpicture}
\end{center}\caption{Illustration of the proof of Theorem~\ref{theo:ergo}.}\label{fig:ergo_proof}
\end{figure}
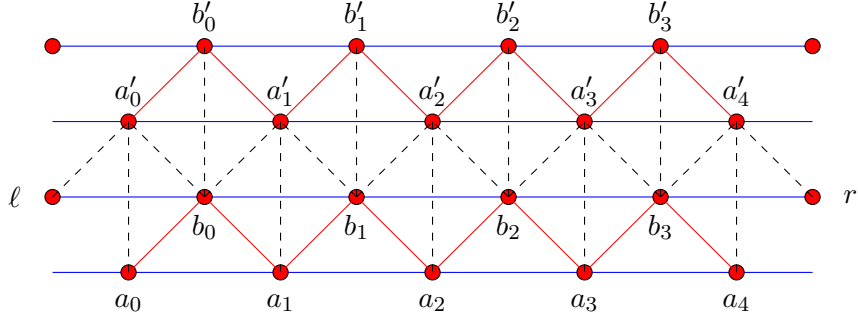
\end{proof}

\section{Directional (quasi-)reversibility of PCA having an invariant product measure} \label{sec:rqrPCA}
The stationary space-time diagram of a PCA (see Def.~\ref{def:std}) is a random field indexed by $\Ze$. For a point $x=(i,t)\in\Ze$, we will also use the notation $\eta(x)=\eta(i,t)=\eta_t(i)$, and for a family $L\subset\Ze$, we define $\eta(L)=(\eta(x))_{x\in L}$.

The following lemma proves that the space-time diagram of a positive-rate PCA characterizes its dynamics. Precisely, if two positive-rates PCA $A$ and $A'$ have the same space-time diagram taken under their respective invariant measures $\mu$ and $\mu'$, then $A=A'$ and $\mu=\mu'$.

\begin{lemma}\label{lem-std}
Let $(A,\mu)$ and $(A',\mu')$ two positive-rates PCA with one of their invariant measure. Then, $\quad G(A,\mu) \eqd G(A',\mu') \quad \implies \quad (A,\mu) = (A',\mu').$
\end{lemma}

\begin{proof} Let us set $G=G(A,\mu)=(\eta_t(i) : t \in \ZZ,  i \in \ZZ_t)$ and $G'=G(A',\mu')=(\eta'_t(i) : t \in \ZZ, i \in \ZZ_t)$. By definition, $G|_{t=0,1} \sim \mu$ and $G'|_{t=0,1} \sim \mu'$. Since $G \eqd G'$, we obtain $\mu = \mu'$. 

Let us denote $\mu(a,b,c)=\prob{\eta_1(-1) = a, \eta_0(0) = b, \eta_1(1) = c}$. As a consequence from the fact that $A$ has positive rates, we have: $\forall (a,b,c) \in S^3, \mu(a,b,c)> 0$. 
Thus, for any $a,b,c,d\in S,$ we have: $\prob{\eta_1(-1) = a, \eta_0(0) = b, \eta_1(1) = c, \eta_2(0)) = d} = \mu(a,b,c) T(a,b,c;d) > 0.$
The same relation holds for $A'$ and as $G \eqd G'$, we obtain: 
$\mu(a,b,c) T(a,b,c;d) = \mu'(a,b,c) T'(a,b,c;d).$ 
Since $\mu=\mu'$, we deduce that $T(a,b,c;d) = T'(a,b,c;d)$ for any $a,b,c,d\in S$. Hence, $A=A'$.
\end{proof}

By Lemma~\ref{lem-std}, if a PCA is $g$-quasi-reversible (see Def~\ref{def:rev}), its $g$-reverse is thus unique. Let's enumerate some easy results on quasi-reversible PCA and reversible PCA.

\begin{proposition} \label{prop:obvious}
Let $A$ be a positive-rates PCA and let $\mu$ be one of its invariant measures.
\begin{enumerate}
\item $(A,\mu)$ is $id$-reversible.
\item $(A,\mu)$ is $v$-quasi-reversible and the $v$-reverse PCA is defined by the transition kernel $T_v(c,b,a;d) = T(a,b,c;d)$. 
\item For any $g \in D_4$, if $(A,\mu)$ is $g$-quasi-reversible, then its $g$-reverse $(A_g,\mu_g)$ is $g^{-1}$-quasi-reversible and $(A,\mu)$ is the $g^{-1}$-reverse of $(A_g,\mu_g)$. 
\item If $(A,\mu)$ is $g$-quasi-reversible and $(A_g,\mu_g)$ is its $g$-reverse and if $(A_g,\mu_g)$ is $g'$-quasi-reversible and $(A_{g'g},\mu_{g'g})$ is its $g'$-reverse, then $(A,\mu)$ is $g' g$-quasi-reversible and $(A_{g'g},\mu_{g'g})$ is its $g' g$-reverse. 
\item For any subset $E$ of $D_4$, if $(A,\mu)$ is $E$-reversible, then $(A,\mu)$ is $<E>$-reversible.
\end{enumerate}
\end{proposition} 

\begin{remark}
Since $<r,v>=D_4$, a consequence of the last point of Prop.~\ref{prop:obvious} is that if $(A,\mu)$ is $r$ and $v$-reversible, then it is $D_4$-reversible. 
\end{remark}

\subsection{Quasi-reversible PCA with $p$-HZPM invariant}\label{sec:qrPCAinv}

Let us denote by $\triangE{S}$ the subset of positive-rates PCA with set of symbols $S$ having an invariant HZPM. In addition, for a positive probability vector $p$ on $S$, we define $\triang{S}{p}$ as the subset of $\triangE{S}$ made of PCA for which the measure $\pi_p$ is invariant. By Theorem~\ref{theo:gen0}, $\triang{S}{p}$ is thus the set of PCA satisfying \C~\ref{cond-pm}. 

In this section, we characterize PCA of $\triang{S}{p}$ that are $g$-quasi-reversible, for each possible $g\in D_4$. First of all, let us focus on the $r^2$-quasi-reversibility, or equivalently on the $h$-quasi-reversibility, which corresponds to time-reversal. For any stationary Markov chain, we can define a time-reversed chain, which has still the Markov property. But in general, the time-reversed chain of a PCA is no more a PCA. Next theorem shows that any PCA in $\triang{S}{p}$ is $r^2$-quasi-reversible, which means that the time-reversed chain of a PCA of $\triang{S}{p}$ is still a PCA with memory two, which furthermore belongs to $\triang{S}{p}$, since it preserves the measure $\pi_p$.

\begin{theorem} \label{theo:r2rev}
Any PCA $A \in \triang{S}{p}$ is $r^2$-quasi-reversible, and the transition kernel $T_{r^2}$ of its $r^2$-reverse ${A}_{r^2}$ is given by:
\begin{displaymath}
\forall a,b,c,d\in S, \quad T_{r^2}(c,d,a;b) = \frac{p(b)}{p(d)} T(a,b,c;d).
\end{displaymath}
\end{theorem}

\begin{proof} For some given choice of $n\in\Z_{t}$, let us denote again: $X_i=\eta_{t-1}(n+1+2i), Y_i=\eta_{t}(n+2i), Z_i=\eta_{t+1}(n+1+2i),$ for $i\in \Z$, see Fig.~\ref{fig:hzpm_proof}. The following computation proves the result wanted.
\begin{align*}
&\prob{(X_i)_{0\leq i\leq k} = (x_i)_{0\leq i\leq k} |(Y_i)_{0\leq i\leq k+1} = (y_i)_{0\leq i\leq k+1}, (Z_i)_{0\leq i\leq k} = (z_i)_{0\leq i\leq k}}\\
&\quad={\prob{(X_i)_{0\leq i\leq k} = (x_i)_{0\leq i\leq k},(Y_i)_{0\leq i\leq k+1} = (y_i)_{0\leq i\leq k+1}, (Z_i)_{0\leq i\leq k} = (z_i)_{0\leq i\leq k}}\over \prob{(Y_i)_{0\leq i\leq k+1} = (y_i)_{0\leq i\leq k+1}, (Z_i)_{0\leq i\leq k} = (z_i)_{0\leq i\leq k}\}}}\\
&\quad={p(y_0) \prod_{i=0}^k p(x_i)p(y_{i+1}) T(y_i,x_i,y_{i+1},z_i) \over p(y_{0}) \prod_{i=0}^k p(z_i) p(y_{i+1})}\\
&\quad=\prod_{i=0}^k \frac{p(x_i)}{p(z_i)} T(y_i,x_i,y_{i+1},z_i).
\end{align*}
\end{proof}

With~$(2)$ and~$(4)$ of Prop.~\ref{prop:obvious}, we instantly obtain the following corollary.

\begin{corollary} \label{cor:allquasirev}
Any PCA $A \in \triangE{S}$ is $\{h,r^2,v\}$-quasi-reversible.
\end{corollary}

Let us now focus on the space-time diagram $G(A,\pi_p)$ of a PCA $A\in\triang{S}{p}$, taken under its unique invariant measure $\pi_p$. By definition, any horizontal line of that space-time diagram is i.i.d. The following proposition extends that result to other types of lines.

\begin{definition} A \emph{zigzag polyline} is a sequence $(i,t_i)_{m\leq i\leq n}\in \Ze$ such that for any $i\in\{m,\ldots, n\}$, $(t_{i+1}-t_i)\in\{-1,1\}$.
\end{definition}

\begin{proposition} \label{prop:ind}
Let $A \in \triang{S}{p}$ be a PCA of stationary space-time diagram $G(A,\pi_p)=(\eta_t(i) : t \in \ZZ, i \in \ZZ_t)$. For any zigzag polyline $(i,t_i)_{m\leq i\leq n}$, we have: $(\eta_{t_i}(i) : i \in \{m,\ldots,n\}) \sim \mathcal{B}(p)^{\otimes (n-m+1)}$. 
\end{proposition}
 
Observe that Prop.~\ref{prop:ind} implies that (bi-)infinite zigzag polylines are also made of i.i.d.\ $\ber(p)$ random variables.
 
\begin{proof}
The proof is done by induction on $T = \max(t_i) - \min(t_i)$.
If $T = 1$, then the zigzag polyline is an horizontal zigzag, and since $A \in \triang{S}{p}$, the result is true.

Now, suppose that the result is true for any zigzag polyline such that $\max(t_i)-\min(t_i) = T$, and consider a zigzag polyline $(i,t_i)_{m\leq i\leq n}$ such that $\max(t_i) - \min(t_i) = T+1$. Then, there exists $t$ such that $\min(t_i) = t$ and $\max(t_i) =t+T+1$. Let $M = \{i \in \{m,\ldots,n\} : t_i = t+T+1\}$. For any $i \in M$, we have $t_{i\pm 1} = t+T$ (we assume that $0,n\notin M$, even if it means extending the line). So, by induction, we have that $(\eta(i,t_i - 2\ \ind{i \in M}) : i \in \{m,\ldots,n\}) \sim \mathcal{B}(p)^{\otimes (m-n+1)}$.
For any $(a_i )_{m \leq i \leq n}\in S^{m-n+1}$, we have:
\begin{align*}
& \prob{\eta(x_i,t_i) = a_i : m \leq i \leq n} \\
\quad & = \sum_{(b_i : i \in M) \in S^{M}} \prob{\{\eta(i,t_i) = a_i : i \notin M\},\{\eta(i,t_i-2) = b_i : i \in M\}} \prod_{i \in M} T(a_{i-1},b_i,a_{i+1} ; a_i) \\
\quad & = \sum_{(b_i : i \in M) \in S^{M}} \prod_{i \notin M} p(a_i) \prod_{i \in M} p(b_i) T(a_{i-1},b_i,a_{i+1} ; a_i) \\
\quad & = \prod_{i \notin M} p(a_i) \prod_{i \in M} \sum_{b_i \in S} p(b_i) T(a_{i-1},b_i,a_{i+1} ; a_i) \\
\quad & = \prod_{i =m}^{n} p(a_i).
\end{align*}
\end{proof}

Now, we will characterize PCA in $\triangE{S}$ that are $r$-quasi-reversible. 
\begin{proposition} \label{prop:r1rev}
Let $A \in \triang{S}{p}$. $A$ is $r$-quasi-reversible if and only if:
\begin{cond} \label{cond:r1rev}
for any $a,b,d \in S$, $\sum_{c \in S} p(c) T(a,b,c;d) = p(d)$.
\end{cond}
In that case, the transition kernel $T_{r}$ of its $r$-reverse ${A}_r$ is given by:
\begin{equation} \label{eq:r1rev}
\forall a,b,c,d\in S, \quad T_r(d,a,b;c) = \frac{p(c)}{p(d)} T(a,b,c;d).
\end{equation}
\end{proposition}

\begin{proposition} \label{prop:r3rev}
Let $A \in \triang{S}{p}$. $A$ is $r^{-1}$-quasi-reversible if and only if:
\begin{cond} \label{cond:r3rev}
for any $b,c,d \in S$, $\sum_{a \in S} p(a) T(a,b,c;d) = p(d)$.
\end{cond}
In that case, the transition kernel $T_{r^{-1}}$ of its $r^{-1}$-reverse ${A}_{r^{-1}}$ is given by:
\begin{equation}
\forall a,b,c,d\in S, \quad T_{r^{-1}}(b,c,d;a) = \frac{p(a)}{p(d)} T(a,b,c;d).
\end{equation}
\end{proposition}

We prove only Prop.~\ref{prop:r1rev}, the proof of Prop.~\ref{prop:r3rev} being similar.

\begin{proof} $\bullet$ Let us first prove that if $A$ is $r$-quasi-reversible, then \C~\ref{cond:r1rev} holds, and that the $r$-reverse satisfies: $T_r(d,a,b;c) = \frac{p(c)}{p(d)} T(a,b,c;d)$. Let us recall the notations $\vu=(-1,1)$ and $\vv=(1,1)$. Since $A \in \triang{S}{p}$, for any $x \in \Ze$ and $a,b,c,d \in S$, we have:
\begin{align*}
\prob{\eta(x+\vu) = a, \eta(x) = b, \eta(x + \vv) = c, \eta(x+\vu+\vv) = d} = p(a) p(b) p(c) T(a,b,c;d).
\end{align*}
Hence, \begin{align}
T_r(d,a,b;c) & = \prob{\eta(x + \vv ) = c | \eta(x + \vu ) = a, \eta(x) = b,  \eta(x + \vu + \vv ) = d} \nonumber \\
                     & = \frac{p(a) p(b) p(c) T(a,b,c;d)}{\sum_{c'\in S} p(a) p(b) p(c') T(a,b,c';d)} \nonumber \\
                     & = \frac{p(c) T(a,b,c;d)}{\sum_{c'\in S} p(c') T(a,b,c';d)} \label{eq:uniquerev}
\end{align}

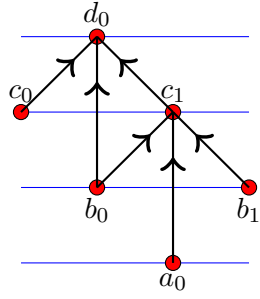
\begin{figure}
\begin{center}
\begin{tikzpicture}[scale=1]
\foreach \y in {0,1,2,3} \draw[blue] (-2,\y) -- (1,\y) ;
\foreach \x in {0} \foreach \y in {0} {\draw [fill=red] (\x,\y) circle [radius=0.1] ;}
\foreach \x in {-1,1} \foreach \y in {1} {\draw [fill=red] (\x,\y) circle [radius=0.1] ;}
\foreach \x in {-2,0} \foreach \y in {2} {\draw [fill=red] (\x,\y) circle [radius=0.1] ;}
\foreach \x in {-1} \foreach \y in {3} {\draw [fill=red] (\x,\y) circle [radius=0.1] ;}
\foreach \x in {0} \foreach \y in {2} {\draw[->-,thick] (\x,\y-2) -- (\x,\y) ;}
\foreach \x in {-1} \foreach \y in {3} {\draw[->-,thick] (\x,\y-2) -- (\x,\y) ;}
\foreach \x in {0} \foreach \y in {2} {\draw[->-,thick] (\x+1,\y-1) -- (\x,\y) ;}
\foreach \x in {-1} \foreach \y in {3} {\draw[->-,thick] (\x+1,\y-1) -- (\x,\y) ;}
\foreach \x in {0} \foreach \y in {2} {\draw[->-,thick] (\x-1,\y-1) -- (\x,\y) ;}
\foreach \x in {-1} \foreach \y in {3} {\draw[->-,thick] (\x-1,\y-1) -- (\x,\y) ;}
\node[below] at (-1,1) {$b_0$};
\node[below] at (0,0) {$a_0$};
\node[below] at (1,1) {$b_1$};
\node[above] at (0,2) {$c_1$};
\node[above] at (-2,2) {$c_0$};
\node[above] at (-1,3) {$d_0$};
\end{tikzpicture}
\end{center}
\caption{The pattern $L$.}  \label{fig:pattern}
\end{figure}

For some $x \in \Ze$, let us introduce the pattern $L=(x,x+\vu,x+\vv,x+2\vu,x+\vu+\vv,x+2\vu+\vv)$, see Fig.~\ref{fig:pattern}. For $a_0,b_0,b_1,c_0,c_1,d_0 \in S$, we are interested in the quantity:
$$Q(a_0,b_0,b_1,c_0,c_1,d_0)=\prob{\eta(L)=(a_0,b_0,b_1,c_0,c_1,d_0)}.$$

On the one hand, using the fact that we have a portion of the space-time diagram $G(A,\pi_p)$, Prop.~\ref{prop:ind} implies that:
$$\prob{\eta(x + 2 \vu) = c_0, \eta(x + \vu ) = b_0, \eta(x) = a_0,  \eta(x + \vv ) = b_1}=p(c_0) p(b_0) p(a_0) p(b_1).$$
We thus obtain:
$Q(a_0,b_0,b_1,c_0,c_1,d_0)= p(c_0) p(b_0) p(a_0) p(b_1) T(b_0,a_0,b_1;c_1) T(c_0,b_0,c_1;d_0)$.
On the other hand, using the fact that $A$ is $r$-quasi-reversible, we have:
\begin{align*}
Q(a_0,b_0,b_1,c_0,c_1,d_0)=&\sum_{b'_1,c'_1\in S}Q(a_0,b_0,b'_1,c_0,c'_1,d_0)T_r(d_0,c_0,b_0;c_1)T_r(c_1,b_0,a_0;b_1).
\end{align*}
It follows that:
\begin{align}
1=&\sum_{b'_1,c'_1\in S}{Q(a_0,b_0,b'_1,c_0,c'_1,d_0)\over Q(a_0,b_0,b_1,c_0,c_1,d_0)}T_r(d_0,c_0,b_0;c_1)T_r(c_1,b_0,a_0;b_1) \nonumber \\
=&\sum_{b'_1,c'_1\in S} {p(b'_1) T(b_0,a_0,b'_1;c'_1) T(c_0,b_0,c'_1;d_0)\over p(b_1) T(b_0,a_0,b_1;c_1) T(c_0,b_0,c_1;d_0)}T_r(d_0,c_0,b_0;c_1)T_r(c_1,b_0,a_0;b_1) \label{eq-remp}
\end{align}
By \eqref{eq:uniquerev}, we have: 
\begin{align*}
T_r(d_0,c_0,b_0;c_1)={p(c_1)T(c_0,b_0,c_1;d_0)\over \sum_{c\in S} p(c) T(c_0,b_0,c;d_0)},\\
T_r(c_1,b_0,a_0;b_1)={p(b_1)T(b_0,a_0,b_1;c_1)\over \sum_{b\in S} p(b)T(b_0,a_0,b;c_1)}. 
\end{align*}
After replacing in~\eqref{eq-remp}, we obtain: 
\begin{equation*}
\left(\sum_{b\in S} p(b)T(b_0,a_0,b;c_1)\right) \left( \sum_{c\in S}p(c) T(c_0,b_0,c;d_0)\right)=p(c_1)\sum_{b'_1,c'_1\in S} p(b'_1)T(b_0,a_0,b'_1;c'_1)T(c_0,b_0,c'_1;d_0).
\end{equation*}
Summing over $d_0\in S$ on both sides and simplifying gives:
$\sum_{b\in S} p(b)T(b_0,a_0,b;c_1)=p(c_1)$.
Hence, \C~\ref{cond:r1rev} is necessary. Together with \eqref{eq:uniquerev}, we deduce~\eqref{eq:r1rev}.

$\bullet$ Let us now assume that \C~\ref{cond:r1rev} holds, and let $T_r$ be defined by~\eqref{eq:r1rev}. For any $d,a,b \in S$, we have:
\begin{displaymath}
\sum_{c\in S} T_r(d,a,b;c) = \frac{\sum_{c\in S} p(c) T(a,b,c;d)}{p(d)} = 1.
\end{displaymath}
Hence, $T_r$ is a transition kernel.

For some $x\in \Ze$, and $m \in \NN$ let us define the pattern $M=(x+i \vu+ j \vv)_{0\leq i,j\leq m}$. Using Prop.~\ref{prop:ind}, for any $(a_{i,j})_{0 \leq i,j \leq m} \in S^{\{0,1,\dots,m\}^2}$, we have:
$$
\prob{\eta(M) = (a_{i,j})_{0 \leq i,j \leq m}} = \prod_{i=0}^m p(a_{i,0}) \prod_{j=1}^m p(a_{0,j}) \prod_{i=1}^m\prod_{j=1}^m T(a_{i,j-1},a_{i-1,j-1},a_{i-1,j};a_{i,j}).
$$
This computation is represented on Fig.~\ref{fig:flips2} (a). The points for which $p(a_{i,j})$ appears in the product are marked by black dots, while the black vertical arrows represent the values that are computed through the transition kernel $T$. Now, by \eqref{eq:r1rev}, we know that:
\begin{equation}\label{eq:flip}
\forall a,b,c,d\in S, \quad p(c)T(a,b,c;d) = p(d)T_r(d,a,b;c).
\end{equation}
It means that in the product above, we can perform flips as represented in Fig.~\ref{fig:flips1}, where an arrow to the right now represents a computation made with the transition kernel ${ T}_r$. We say that such a use of \eqref{eq:flip} is a flip of $(c,d)$. By flipping successively the cells from right to left and bottom to top: first $(a_{0,m},a_{1,m})$, then $(a_{0,m-1},a_{1,m-1}),(a_{1,m},a_{2,m}),$ and $(a_{0,m-2},a_{1,m-2}),$ $(a_{1,m-1},a_{2,m-1}),$ $(a_{2,m},a_{3,m})$ etc., we finally obtain (see Fig.~\ref{fig:flips2} for an illustration):
\begin{equation}\label{eq:tourne}
\prob{\eta(M) = (a_{i,j})_{0 \leq i,j \leq m}}= \prod_{i=0}^m p(a_{i,0}) \prod_{j=1}^m p(a_{m,j}) \prod_{i=0}^{m-1} \prod_{j=1}^m T_r(a_{i+1,j},a_{i+1,j-1},a_{i,j-1};a_{i,j}).
\end{equation}

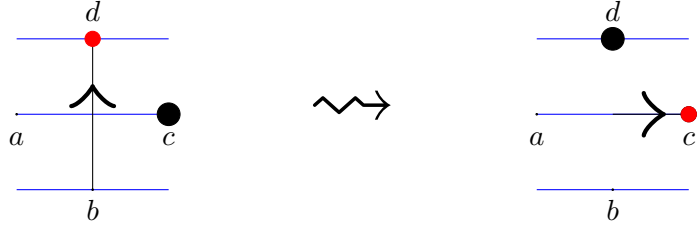
\begin{figure}
\begin{center}
\begin{tabular}{>{\centering\arraybackslash}m{4cm}>{\centering\arraybackslash}m{2cm}>{\centering\arraybackslash}m{4cm}}
\begin{tikzpicture}[scale=1, transform shape]
\foreach \y in {0,1,2} \draw[blue] (-1,\y) -- (1,\y) ;
\foreach \x in {0} \foreach \y in {2} {\draw[>>] (\x,\y-2) -- (\x,\y) ;}
\draw [fill=black] (1,1) circle [radius=0.1] ;
\node[below] at (-1,0.9) {$a$};
{\draw (-1,1) circle [radius=0.01] ;}
\node[below] at (0,0) {$b$};
{\draw (0,0) circle [radius=0.01] ;}
\node[below] at (1,0.9) {$c$};
{\draw [fill=black] (1,1) circle [radius=0.15] ;}
\node[above] at (0,2.1) {$d$};
{\draw [red, fill=red] (0,2) circle [radius=0.1] ;}
\end{tikzpicture}\
&\scalebox{3}{$\rightsquigarrow$}&
\begin{tikzpicture}[scale=1, transform shape]
\foreach \y in {0,1,2} \draw[blue] (-1,\y) -- (1,\y) ;
\foreach \x in {1} \foreach \y in {1} {\draw[>>] (\x-1,\y) -- (\x,\y) ;}
\draw [fill=black] (1,1) circle [radius=0.1] ;
\node[below] at (-1,0.9) {$a$};
{\draw (-1,1) circle [radius=0.01] ;}
\node[below] at (0,0) {$b$};
{\draw (0,0) circle [radius=0.01] ;}
\node[below] at (1,0.9) {$c$};
{\draw [red, fill=red] (1,1) circle [radius=0.1] ;}
\node[above] at (0,2.1) {$d$};
{\draw [fill=black] (0,2) circle [radius=0.15] ;}
\end{tikzpicture}
\end{tabular}
\end{center}
\caption{Elementary flip illustrating the relation $p(c)T(a,b,c;d) = p(d)T_r(d,a,b;c)$.} \label{fig:flips1}
\end{figure}

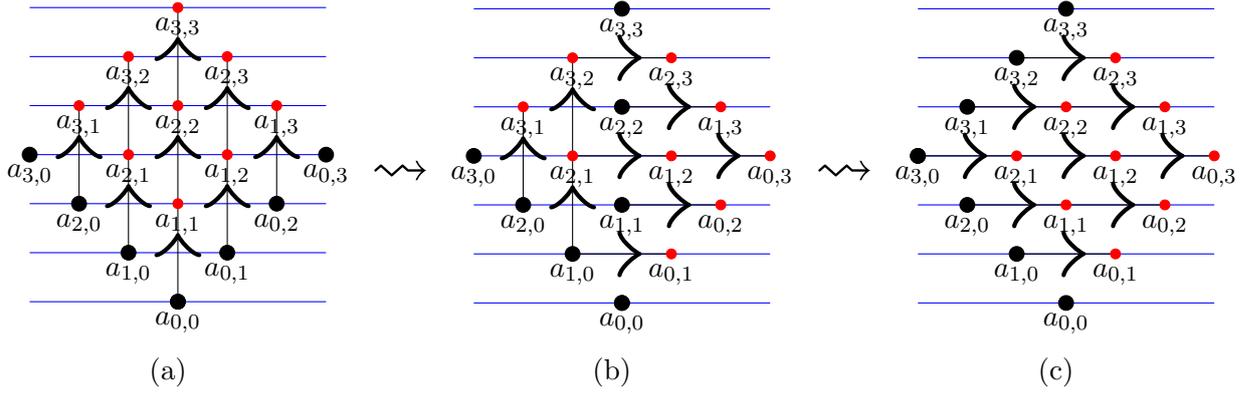
\begin{figure}
\begin{center}
\begin{tabular}{m{4.5cm}m{0.5cm}m{4.5cm}m{0.5cm}m{4.5cm}}
\begin{tikzpicture}[scale=0.65]
\foreach \y in {0,1,2,3,4,5,6} \draw[blue] (-3,\y) -- (3,\y) ;

\foreach \x in {0} \foreach \y in {2,4,6} {\draw[>>] (\x,\y-2) -- (\x,\y) ;}
\foreach \x in {-1,1} \foreach \y in {3,5} {\draw[>>] (\x,\y-2) -- (\x,\y) ;}
\foreach \x in {-2,2} \foreach \y in {4} {\draw[>>] (\x,\y-2) -- (\x,\y) ;}

\node[below] at (0,0) {$a_{0,0}$};
\node[below] at (-1,1) {$a_{1,0}$};
\node[below] at (1,1) {$a_{0,1}$};
\node[below] at (-2,2) {$a_{2,0}$};
\node[below] at (0,2) {$a_{1,1}$};
\node[below] at (2,2) {$a_{0,2}$};
\node[below] at (-3,3) {$a_{3,0}$};
\node[below] at (-1,3) {$a_{2,1}$};
\node[below] at (1,3) {$a_{1,2}$};
\node[below] at (3,3) {$a_{0,3}$};
\node[below] at (-2,4) {$a_{3,1}$};
\node[below] at (0,4) {$a_{2,2}$};
\node[below] at (2,4) {$a_{1,3}$};
\node[below] at (-1,5) {$a_{3,2}$};
\node[below] at (1,5) {$a_{2,3}$};
\node[below] at (0,6) {$a_{3,3}$};

\foreach \x in {0} \foreach \y in {0} {\draw [red, fill=red] (\x,\y) circle [radius=0.1] ;}
\foreach \x in {-1,1} \foreach \y in {1} {\draw [red, fill=red] (\x,\y) circle [radius=0.1] ;}
\foreach \x in {-2,0,2} \foreach \y in {2} {\draw [red, fill=red] (\x,\y) circle [radius=0.1] ;}
\foreach \x in {-3,-1,1,3} \foreach \y in {3} {\draw [red, fill=red] (\x,\y) circle [radius=0.1] ;}
\foreach \x in {-2,0,2} \foreach \y in {4} {\draw [red, fill=red] (\x,\y) circle [radius=0.1] ;}
\foreach \x in {-1,1} \foreach \y in {5} {\draw [red, fill=red] (\x,\y) circle [radius=0.1] ;}
\foreach \x in {0} \foreach \y in {6} {\draw [red, fill=red] (\x,\y) circle [radius=0.1] ;}

\foreach \x in {0,1,2,3} {\draw [fill=black] (-\x,\x) circle [radius=0.15] ; \draw [fill=black] (\x,\x) circle [radius=0.15] ;}
\end{tikzpicture}
&\scalebox{2}{$\rightsquigarrow$}&

\begin{tikzpicture}[scale=0.65]
\foreach \y in {0,1,2,3,4,5,6} \draw[blue] (-3,\y) -- (3,\y) ;

\foreach \x in {1} \foreach \y in {1,3,5} {\draw[>>] (\x-2,\y) -- (\x,\y) ;}
\foreach \x in {2} \foreach \y in {2,4} {\draw[>>] (\x-2,\y) -- (\x,\y) ;}
\foreach \x in {3} \foreach \y in {3} {\draw[>>] (\x-2,\y) -- (\x,\y) ;}

\foreach \x in {-1} \foreach \y in {3,5} {\draw[>>] (\x,\y-2) -- (\x,\y) ;}
\foreach \x in {-2} \foreach \y in {4} {\draw[>>] (\x,\y-2) -- (\x,\y) ;}

\node[below] at (0,0) {$a_{0,0}$};
\node[below] at (-1,1) {$a_{1,0}$};
\node[below] at (1,1) {$a_{0,1}$};
\node[below] at (-2,2) {$a_{2,0}$};
\node[below] at (0,2) {$a_{1,1}$};
\node[below] at (2,2) {$a_{0,2}$};
\node[below] at (-3,3) {$a_{3,0}$};
\node[below] at (-1,3) {$a_{2,1}$};
\node[below] at (1,3) {$a_{1,2}$};
\node[below] at (3,3) {$a_{0,3}$};
\node[below] at (-2,4) {$a_{3,1}$};
\node[below] at (0,4) {$a_{2,2}$};
\node[below] at (2,4) {$a_{1,3}$};
\node[below] at (-1,5) {$a_{3,2}$};
\node[below] at (1,5) {$a_{2,3}$};
\node[below] at (0,6) {$a_{3,3}$};

\foreach \x in {0} \foreach \y in {0} {\draw [red, fill=red] (\x,\y) circle [radius=0.1] ;}
\foreach \x in {-1,1} \foreach \y in {1} {\draw [red, fill=red] (\x,\y) circle [radius=0.1] ;}
\foreach \x in {-2,0,2} \foreach \y in {2} {\draw [red, fill=red] (\x,\y) circle [radius=0.1] ;}
\foreach \x in {-3,-1,1,3} \foreach \y in {3} {\draw [red, fill=red] (\x,\y) circle [radius=0.1] ;}
\foreach \x in {-2,0,2} \foreach \y in {4} {\draw [red, fill=red] (\x,\y) circle [radius=0.1] ;}
\foreach \x in {-1,1} \foreach \y in {5} {\draw [red, fill=red] (\x,\y) circle [radius=0.1] ;}
\foreach \x in {0} \foreach \y in {6} {\draw [red, fill=red] (\x,\y) circle [radius=0.1] ;}

\foreach \x in {0,1,2,3} {\draw [fill=black] (-\x,\x) circle [radius=0.15] ; \draw [fill=black] (0,2*\x) circle [radius=0.15] ;}
\end{tikzpicture}
&\scalebox{2}{$\rightsquigarrow$}&
\begin{tikzpicture}[scale=0.65]
\foreach \y in {0,1,2,3,4,5,6} \draw[blue] (-3,\y) -- (3,\y) ;

\foreach \x in {0} \foreach \y in {2,4} {\draw[>>] (\x-2,\y) -- (\x,\y) ;}
\foreach \x in {1} \foreach \y in {1,3,5} {\draw[>>] (\x-2,\y) -- (\x,\y) ;}
\foreach \x in {-1} \foreach \y in {3} {\draw[>>] (\x-2,\y) -- (\x,\y) ;}
\foreach \x in {2} \foreach \y in {2,4} {\draw[>>] (\x-2,\y) -- (\x,\y) ;}
\foreach \x in {3} \foreach \y in {3} {\draw[>>] (\x-2,\y) -- (\x,\y) ;}

\node[below] at (0,0) {$a_{0,0}$};
\node[below] at (-1,1) {$a_{1,0}$};
\node[below] at (1,1) {$a_{0,1}$};
\node[below] at (-2,2) {$a_{2,0}$};
\node[below] at (0,2) {$a_{1,1}$};
\node[below] at (2,2) {$a_{0,2}$};
\node[below] at (-3,3) {$a_{3,0}$};
\node[below] at (-1,3) {$a_{2,1}$};
\node[below] at (1,3) {$a_{1,2}$};
\node[below] at (3,3) {$a_{0,3}$};
\node[below] at (-2,4) {$a_{3,1}$};
\node[below] at (0,4) {$a_{2,2}$};
\node[below] at (2,4) {$a_{1,3}$};
\node[below] at (-1,5) {$a_{3,2}$};
\node[below] at (1,5) {$a_{2,3}$};
\node[below] at (0,6) {$a_{3,3}$};

\foreach \x in {0} \foreach \y in {0} {\draw [red, fill=red] (\x,\y) circle [radius=0.1] ;}
\foreach \x in {-1,1} \foreach \y in {1} {\draw [red, fill=red] (\x,\y) circle [radius=0.1] ;}
\foreach \x in {-2,0,2} \foreach \y in {2} {\draw [red, fill=red] (\x,\y) circle [radius=0.1] ;}
\foreach \x in {-3,-1,1,3} \foreach \y in {3} {\draw [red, fill=red] (\x,\y) circle [radius=0.1] ;}
\foreach \x in {-2,0,2} \foreach \y in {4} {\draw [red, fill=red] (\x,\y) circle [radius=0.1] ;}
\foreach \x in {-1,1} \foreach \y in {5} {\draw [red, fill=red] (\x,\y) circle [radius=0.1] ;}
\foreach \x in {0} \foreach \y in {6} {\draw [red, fill=red] (\x,\y) circle [radius=0.1] ;}

\foreach \x in {0,1,2,3} {\draw [fill=black] (-\x,\x) circle [radius=0.15] ; \draw [fill=black] (-\x,6-\x) circle [radius=0.15] ;}
\end{tikzpicture}\\
\centering (a)&& \centering (b)&&\centering (c)
\end{tabular}
\end{center}
\caption{From $T$ to $T_r$ using flips.} \label{fig:flips2}
\end{figure}

Let us define the vertical lines: $V_{-1}=(x+(m-i)\vu+(m-i-1)\vv)_{0\leq i\leq m-1}, V_{0}=(x+(m-i)\vu+(m-i)\vv)_{0\leq i\leq m}, V_{1}=(x+(m-i-1)\vu+(m-i)\vv)_{0\leq i\leq m-1}$. From \eqref{eq:tourne}, we deduce that:
\begin{align*}
&\prob{\eta(V_1)=(a_{m,m-1},\ldots,a_{1,0}) \, | \; \eta(V_0)=(a_{m,m},\ldots,a_{0,0}), \eta(V_{-1})=(a_{m,m-1},\ldots,a_{1,0})}\\
&=\prod_{i=0}^{m-1}T_r(a_{m-i,m},a_{m-i,m-i-1},a_{m-i-1,m-i-1};a_{m-i-1,m-i})
\end{align*}

Since this is true for any $x \in \Ze$ and any $m\in\NN$, it follows that the rotation of $G(A,\pi_p)$, the space-time diagram of $A$, by the rotation $r$ is a space-time diagram of ${A}_r$, whose transition kernel is $T_r$, under one of its invariant measure that we denote by $\mu = (\pi_p)_r$ (observe that we do not specify the dependence on $A$ in that last notation, although the measure depends on $A$). Furthermore, we can express explicitly the finite-dimensional of $\mu$. For any $m\in\NN$, we have:
\begin{align}
&\mu((a_{i+1,i})_{0 \leq i \leq m-1},(a_{i,i})_{0 \leq i \leq m})=\prob{\eta(V_0)=(a_{m,m},\ldots,a_{0,0}), \eta(V_1)=(a_{m,m-1},\ldots,a_{1,0})}\notag \\
& = \sum_{(a_{i,j}\, : \; i,i+1 \neq j)} \; \prod_{i=0}^m p(a_{i,0}) \prod_{j=1}^m p(a_{0,j}) \prod_{1\leq i,j\leq m} T(a_{i,j-1},a_{i-1,j-1},a_{i-1,j};a_{i,j}) \label{calcul_expl1}\\
& = \sum_{(a_{i,j}\, : \; {j<i})} \; \prod_{i=0}^m p(a_{i,0}) \prod_{j=1}^m p(a_{m,j}) \prod_{1\leq j\leq i+1\leq m} T_r(a_{i+1,j},a_{i+1,j-1},a_{i,j-1};a_{i,j}) \label{calcul_expl2}
\end{align}
\end{proof}

Note that in Prop.~\ref{prop:r1rev} and Prop.~\ref{prop:r3rev}, the reverse PCA is not necessary an element of $\triang{S}{p}$. In the space-time diagram $G(A,\pi_p)$, the points $x,x+\vu,x+2\vu,\ldots,x+m\vu,x+m\vu+\vv,\ldots,x+m\vu+m\vv$ consist in independent $\ber(p)$ random variables. But if we now consider only the three points $x,x+\vv,x+\vu+\vv$, they have no reason to be independent, so that $\mu$ can be different from $\pi_p$. Next theorem specifies the cases for which the reverse PCA $A_r$ is an element of $\triang{S}{p}$, meaning that $\mu=\pi_p$.

But before, let us prove that in the space-time diagram $G(A,\pi)$, each vertical line $V_i$ consists in independent variables. This means that even if the reverse PCA does not have necessarily an invariant $p$-HZPM, the measure $\mu$ is at least such that each horizontal (straight) line consists in independent $\ber(p)$ random variables.

\begin{proposition}\label{prop-vertical} Let $A\in\triang{S}{p}$ be an $r$-quasi-reversible PCA (resp. an $r^{-1}$-quasi-reversible PCA). Then, for any $x\in\Ze$, the vertical line $V=\{x+k\vu+k\vv : k\in\Z\}$ consists in independent $\ber(p)$ random variables.
\end{proposition}

\begin{proof} We assume that $A$ is $r$-quasi-reversible, the case $r^{-1}$-quasi-reversible being similar. Let us consider again the Fig.~\ref{fig:flips2}.
Precisely, let us do the succession of flips leading to Fig.~\ref{fig:flips2} (b). Then, by summing over $(a_{i,j})_{i<j}$ and then over $(a_{i,j})_{i>j>0}$, we obtain :
$$\mu((a_{i,0})_{0 < i \leq m},(a_{i,i})_{0 \leq i \leq m})=\prod_{0 < i \leq m}p(a_{i,0})\prod_{0 \leq i \leq m}p(a_{i,i}).$$
This means that the points $x+m\vu, x+(m-1)\vu,\ldots,x+\vu,x,x+(\vu+\vv),x+2(\vu+\vv),\ldots,x+m(\vu+\vv)$ consist in independent $\ber(p)$ random variables. As a consequence, the points of the vertical line $V_0$ are independent $\ber(p)$ random variables.
\end{proof}

As already mentioned, in Prop.~\ref{prop:r1rev} and Prop.~\ref{prop:r3rev}, if a PCA $A \in \triang{S}{p}$ satisfies \C~\ref{cond:r1rev} and not \C~\ref{cond:r3rev} (or the reverse), we get a PCA $C=A_r$ (or $A_{r^{-1}}$) for which we can compute exactly the marginals of an invariant measure, although it does not have a well-identified form. As a consequence of the previous results, we thus obtain next theorem, which gives conditions on the transitions of a PCA $C$ for being of the form $C=A_r$, with $A$ having an invariant $p$-HZMP. In that case, the measure $\mu=(\pi_p)_r$ is an invariant measure for $C$, and we have explicit formula for the computation of its marginals, see \eqref{calcul_expl1} and \eqref{calcul_expl2}.

\begin{theorem}\label{theo:cool}
Let $C$ be a PCA with transition kernel $T$. If there exists a probability distribution $p$ such that \C~\ref{cond:r1rev} and \C~\ref{cond:r3rev} hold, then there exists a unique probability distribution $\mu$ on $S^{\ZZ_0} \times S^{\ZZ_{1}}$ such that 
\begin{itemize}
\item $\mu$ is invariant by $C$,
\item $(C,\mu)$ is $\{r^{-1},r\}$-quasi-reversible and its $r^{-1}$-reverse is $(C_{r^{-1}},\pi_p)$ with $C_{r^{-1}} \in \triang{S}{p}$, same hold for the $r$-reverse,
\item $\mu_{|\ZZ_0} = \ber(p)^{\otimes \ZZ_0}$ and $\mu_{|\ZZ_{1}} = \ber(p)^{\otimes \ZZ_{1}}$.
\end{itemize}
Moreover, we have explicit formula for the computation of the marginals of $\mu$.
\end{theorem}

In Section~\ref{sec:binary}, Example~\ref{exemple:binaire} provides an example of a PCA satisfying only \C~\ref{cond:r1rev}, so that its $r$-reverse $A_r$ satisfies the conditions of Theorem~\ref{theo:cool} above. In contrast, next theorem describes the family of PCA satisfying both \C~\ref{cond:r1rev} and \C~\ref{cond:r3rev}.

\begin{theorem} \label{theo:r-rev}
Let $A \in \triang{S}{p}$. The following properties are equivalent:
\begin{enumerate}
\item $A$ is $\{r,r^{-1}\}$-quasi-reversible.
\item $A$ is $r$-quasi-reversible and $A_r \in \triang{S}{p}$,
\item $A$ is $r^{-1}$-quasi-reversible and $A_{r^{-1}} \in \triang{S}{p}$, 
\item \C~\ref{cond:r1rev} and \C~\ref{cond:r3rev} hold,
\item $A$ is $D_4$-quasi-reversible.
\end{enumerate}
\end{theorem}

\begin{proof}$\;$ \\
\begin{itemize}
\item[$1 \Rightarrow 2$] If $A$ is $r$-quasi-reversible, then its $r$-reverse $A_r$ is defined by the transition kernel
\begin{displaymath}
T_r(d,a,b;c) = \frac{p(c)}{p(d)} T(a,b,c;d).
\end{displaymath}
We thus have:
\begin{displaymath}
\sum_{a\in S} p(a) T_r(d,a,b;c) = p(c) \frac{\sum_{a\in S} p(a)T(a,b,c;d)}{p(d)} = p(c),
\end{displaymath}
using \C~\ref{cond:r3rev}, since $A$ is $r^{-1}$-quasi-reversible. Thus, \C~\ref{cond-pm} holds for $T_r$ and, by Theorem~\ref{theo:gen0}, $A_r \in \triang{S}{p}$.
\item[$1 \Leftarrow 2$] Since $A_r \in \triang{S}{p}$, by Theorem~\ref{theo:r2rev}, $A_r$ is $r^2$-quasi-reversible. Then, by the property $4.$ of Prop.~\ref{prop:obvious}, $A$ is $r^3=r^{-1}$-quasi-reversible.
\item[$1 \Leftrightarrow 3$] Same proof as $1 \Leftrightarrow 2$.
\item[$1 \Leftrightarrow 4$] It is a consequence of Prop.~\ref{prop:r1rev} and Prop.~\ref{prop:r3rev}.
\item[$1 \Leftrightarrow 5$] It is a consequence of the points $2.$ and $4.$ of Prop~\ref{prop:obvious}, together with Theorem~\ref{theo:r2rev} (see also Corollary~\ref{cor:allquasirev}) .
\end{itemize}
\end{proof}

\begin{remark}\label{rmk:expr-reverse} It follows from the previous results that for $g\in D_4$, if $A$ is $g-$quasi-reversible, then the transition kernel $T_g$ of its $g$-reverse $A_g$ is given by:
\begin{displaymath}
T_g(\sigma_g(a,b,c;d)) = \frac{p(\pi_4(\sigma_g(a,b,c;d)))}{p(d)} T(a,b,c;d),
\end{displaymath}
where $\sigma_g$ is the permutations of the four vertices $a, b, c, d$ induced by the transformation $g\in D_4$, and where $\pi_4$ is the projection on the fourth letter, so that $\pi_4(a,b,c;d) = d$.
\end{remark}

\subsection{Reversible PCA with $p$-HZPM invariant} \label{sec:rPCAinv}
As a consequence of the previous results, we obtain the following characterization of reversible PCA.

\begin{theorem} \label{theo:rev}
Let $A \in \triang{S}{p}$.
\begin{enumerate}
\item $A$ is $v$-reversible iff $T(a,b,c;d) = T(c,b,a;d)$ for any $a,b,c,d\in S$.
\item $A$ is $r^2$-reversible iff $p(b) T(a,b,c;d) = p(d) T(c,d,a;b)$ for any $a,b,c,d\in S$.
\item $A$ is $h$-reversible iff $p(b) T(a,b,c;d) = p(d) T(a,d,c;b)$ for any $a,b,c,d\in S$.
\item $A$ is $<r^2,v>$-reversible iff $T(a,b,c;d) = T(c,b,a;d)$ and $p(b) T(a,b,c;d) = p(d) T(a,d,c;b)$ for any $a,b,c,d\in S$.
\item $A$ is $<r>$-reversible iff $p(a) T(a,b,c;d) = p(d) T(b,c,d;a)$ for any $a,b,c,d\in S$.
\item $A$ is $<r \circ v>$-reversible iff $p(a) T(a,b,c;d) = p(d) T(d,c,b;a)$ for any $a,b,c,d\in S$.
\item $A$ is $D_4$-reversible iff $T(a,b,c;d) = T(c,b,a;d)$ and $p(a) T(a,b,c;d) = p(d) T(b,c,d;a)$ for any $a,b,c,d\in S$. 
\end{enumerate}
\end{theorem}

\begin{proof}
Let $A \in \triang{S}{p}$.
\begin{enumerate}
\item This is an elementary property, true even if $A \notin \triang{S}{p}$.
\item $A$ is $r^2$-reversible iff $A$ is $r^2$-quasi-reversible and its $r^2$-reverse is $A$. Now, by Theorem~\ref{theo:r2rev}, if $A \in \triang{S}{p}$, then $A$ is $r^2$-quasi-reversible and the transition kernel $T_{r^2}$ of its $r^2$-reverse is $T_{r^2}(c,d,a;b) = \frac{p(b)}{p(d)} T(a,b,c;d)$ for any $a,b,c,d\in S$.
\item $A$ is $h$-reversible iff $A$ is $h$-quasi-reversible and its $h$-reverse is $A$. By Corollary~\ref{cor:allquasirev}, if $A \in \triang{S}{p}$, then $A$ is $h$-quasi-reversible and, as mentionned in Remark~\ref{rmk:expr-reverse}, we can show that $T_h(a,d,c;b) = \frac{p(b)}{p(d)} T(a,b,c;d)$ for any $a,b,c,d\in S$.
\item It is an easy consequence of the previous points.
\item $A$ is $r$-reversible iff it is $r$-quasi-reversible and its $r$-reverse is $A$. Hence, by Prop.~\ref{prop:r1rev}, if $A\in \triang{S}{p}$, then $A$ is reversible iff \C~\ref{cond:r1rev} is satisfied and $T(b,c,d;a) = \frac{p(a)}{p(d)} T(a,b,c;d)$ for any $a,b,c,d\in S$. It is in fact sufficient to have $T(b,c,d;a) = \frac{p(a)}{p(d)} T(a,b,c;d)$ for any $a,b,c,d\in S$, since we then have:
\begin{displaymath}
  \sum_{a\in S} p(a) T(a,b,c;d) = \sum_{a\in S} p(a) \frac{p(d)}{p(a)}T(b,c,d;a) = p(d),
\end{displaymath}
meaning that \C~\ref{cond:r1rev} is satisfied
\item $A$ is $r \circ v$-reversible iff it is $r$-quasi-reversible and its $r \circ v$-reverse is $A$. Since $A\in \triang{S}{p}$, by Prop~\ref{prop:r1rev}, $A$ is $r$-quasi reversible iff \C~\ref{cond:r1rev}, and we can prove that the  transition kernel of its $r \circ v$-reverse is then given by $T_{r\circ v}(d,c,b;a) = \frac{p(a)}{p(d)} T(a,b,c;d)$ for any $a,b,c,d\in S$. As in the above point, it is sufficient to have $T(d,c,b;a) = \frac{p(a)}{p(d)} T(a,b,c;d)$ for any $a,b,c,d\in S$, since it implies \C~\ref{cond:r1rev}.
\item It follows from points 1 and 5.
\end{enumerate}
\end{proof}

\subsection{Independence properties of the space-time diagram} \label{sec:ind}

\begin{theorem}\label{theo:iid_lines}
Let us consider a PCA $A \in \triangE{S}(p)$ and its stationary space-time diagram $G=(A,\pi_p)$. Then for any $|a|\leq 1$, the points of $G$ indexed by the discrete line $L_{a,b}=\{(x,y)\in \Ze \, : \; y=ax+b\}$ consist in i.i.d.\ random variables.
\end{theorem}

\begin{proof} This is a consequence of Prop.~\ref{prop:ind}. We can assume without loss of generality that $b=0$ and that $0< a \leq 1$. Let $(x,y)\in\Ze$ be the first point with positive coordinates belonging to the integer line, so that we have in particular $0< y \leq x$. Let us define the sequence $(t_i)_{\in \Z}$ by $t_{i+kx}=i+ky$ for $i\in\{0,\ldots,y-1\}$ and $t_{i+kx}=y+{(-1)^{i-y}-1\over 2}+ky$ for $i\in\{y,\ldots,x-1\}$, and any $k\in\Z$. This sequence satisfies the conditions of Prop.~\ref{prop:ind}, so that $(\eta_{t_i}(i) : i\in\Z)\sim \ber(p)^{\otimes \Z}$. Since $L_{a,b}\subset \{(i,t_i) : i\in \Z\}$, the result follows.
\end{proof}

\begin{theorem}\label{theo:iid_lines2} 
Let us consider a PCA $A\in \triangE{S}(p)$ satisfying \C~\ref{cond:r1rev} or \C~\ref{cond:r3rev}. Then, for any line of its stationary space-time diagram $G=(A,\pi_p)$, nodes on that line are i.i.d. 
\end{theorem}

\begin{proof} We prove the result for a PCA $A\in \triangE{S}(p)$ satisfying \C~\ref{cond:r1rev}. In that case, $A$ is $r$-quasi-reversible. Now, take any line $L$ in $G$.

If the equation of $L$ is $y=ax+b$ with $|a| \leq 1$, then by Theorem~\ref{theo:iid_lines}, nodes on that line are i.i.d. By Prop.~\ref{prop-vertical}, the same property holds if the equation of $L$ is $x=c$.

Let us now consider an equation of the form $y=ax+b$ with $|a|>1$. We can assume without loss of generality that $b=0$. Let $(x,y) \in \Ze$ be the first point with a positive coordinates belonging to the integer line, so that we have in particular $0<|x|<y$. Then we can perform flips, similarly as the ones done in Prop.~\ref{prop-vertical} (see Fig.~\ref{fig:flips2}), to get that, for any $m$, the points $m\vu , (m-1)\vu , \ldots , \vu , (0,0) , (x,y), (2x,2y) , \dots, (kx,ky)$ (with $k = \lfloor m/(x+y) \rfloor$) are i.i.d. In particular, $(0,0) , (x,y), (2x,2y) , \dots, (kx,ky)$ are i.i.d.
\end{proof}

\begin{remark} Observe that as a consequence of Theorem~\ref{theo:cool}, the same result holds for a PCA that does not belong to $\triangE{S}(p)$ but satisfies both \C~\ref{cond:r1rev} and~\C~\ref{cond:r3rev}.
\end{remark}

\paragraph{PCA with strong independence.} Let us recall that we denote $\vu= (-1,1), \vv = (1,1)$. 

\begin{definition} \label{def:lrind}
Let $G=(A,\mu)$ be a stationary space-time diagram of a PCA $A$ under one of its invariant measure $\mu$. We say that $G$ is \emph{top} (resp. \emph{bottom}, \emph{left}, \emph{right}) \emph{i.i.d.}\ if, for any $x \in \Ze$, $\{\eta(x), \eta(x-\vu),\eta(x-\vv)\}$ (resp.$\{\eta(x),\eta(x+\vu), \eta(x+\vv)\}$, $\{\eta(x),\eta(x-\vu),\eta(x+\vv)\}$ $\{\eta(x),\eta(x+\vu),\eta(x-\vv)\}$) are i.i.d. A PCA is said to be 3-to-3 i.i.d.\ if it is top, bottom, left and right i.i.d.
\end{definition}

\begin{proposition}
$G=(A,\mu)$ is both top and bottom i.i.d.\ if and only if $A \in \triangE{S}$ and $\mu$ is its invariant HZPM.
\end{proposition}

\begin{proof}
Let $G=(A,\mu)$ be a top and bottom i.i.d.\ PCA. We denote by $p$ the one-dimensional marginal of $\mu$. Then, we have, for any $a,b,c,d \in S$,
\begin{align*}
\prob{\eta(x)=d,\eta(x-\vu)=a,\eta(x-\vv)=c} & = p(a) p(d) p(c) \text{ (top i.i.d.)} \\
& = \sum_{b\in S} p(a) p(b) p(c) T(a,b,c;d) \text{ (bottom i.i.d.)}. 
\end{align*}
Hence, \C~\ref{cond-pm} holds and, by Theorem~\ref{theo:gen0}, $A \in \triang{S}{p}$, and $\mu=\pi_p$.
The reverse statement is trivial.
\end{proof}

\begin{proposition}
$G=(A,\mu)$ is 3-to-3 i.i.d.\ if and only if $A$ is a $D_4$-quasi-reversible PCA of $\triangE{S}$ and $\mu$ is its invariant HZPM.
\end{proposition}

\begin{proof}
Let $(A,\mu)$ be a 3-to-3 i.i.d.\ PCA, then $A \in \triangE{S}$ because $A$ is both top and bottom i.i.d. Moreover,
\begin{displaymath}
p(a)p(b)p(d) = \sum_{c\in S} p(a)p(b)p(c) T(a,b,c;d) \text{ and } p(c)p(b)p(d) = \sum_{a\in S} p(a)p(b)p(c) T(a,b,c;d),
\end{displaymath} 
using the fact that $A$ is top and left (resp. right) i.i.d.
But these are respectively \C~\ref{cond:r1rev} and \C~\ref{cond:r3rev} and, so, by Theorem~\ref{theo:r-rev}, $A$ is $D_4$-quasi-reversible. The reverse statement is trivial.\end{proof}

Note that $G=(A,\mu)$ is top, bottom, and left (resp. right) i.i.d.\ if and only if $A$ is a $r$-quasi-reversible PCA (resp. $r^{-1}$-quasi-reversible PCA) of $\triangE{S}$ and $\mu$ is its invariant HZPM. 

\paragraph{Connection with previous results for PCA with memory one.} In the special case when the PCA has memory one, meaning that the probability transitions $T(a,b,c;d)$ do no depend on $b\in S$,  \C~\ref{cond-pm} reduces to: $\forall a,c,d\in S$, $p(d)=T(a,\cdot, c ; d).$ 
So, the only PCA having an invariant HZPM are trivial ones (no time dependence at all). In that context, it is in fact more relevant to study PCA having simply an invariant horizontal product measure, as done in \cite{mm_ihp}. Observe that when there is no dependence on $b\in S$, \C~\ref{cond:r1rev}  et \C~\ref{cond:r3rev} become:
$$\forall a,d \in S,\ \sum_{c \in S} p(c) T(a,\cdot,c;d) = p(d) \hspace{1cm} \mbox{ and } \hspace{1cm} \forall c,d \in S,\ \sum_{a \in S} p(a) T(a,\cdot,c;d) = p(d).$$
We recover the two sufficient conditions for having an horizontal product measure, as described in Theorem 5.6 of \cite{mm_ihp}. In that article, the space-time diagrams are represented on a regular triangular lattice, which is more adapted to the models that are considered. The authors show that under one or the other of these two conditions, there exists a transversal PCA, so that after an appropriate rotation of the triangular lattice, the stationary space-time diagram can also be described as the one of another PCA. With our terminology, this corresponds to a quasi-reversibility property.

\section{Horizontal zigzag Markov chains}\label{sec:markov}

\subsection{Conditions for having an invariant HZMC} \label{sec:invHZMC}

In this section, we recall some previous results obtained in~\cite{Casse17} about PCA with memory two having an invariant measure which is a Horizontal Zigzag Markov Chain. Our purpose is to keep the present article as self-contained as possible.

First, let us recall what is a $(F,B)$-HZMC distribution. This is the same notion as $(D,U)$-HZMC in~\cite{Casse17}, but to be consistent with the orientation chosen here for the space-time diagrams, we prefer using the notations $F$ for forward in time, and $B$ for backward in time (rather than $D$ for down and $U$ for up). The definition we give below relies on the following lemma.

\begin{lemma} Let $S$ be a finite set, and let $F = (F(a;b) : a,b \in S)$ and $B = (B(b;c) : b,c \in S)$ be two positive transition matrices from $S$ to $S$. We denote by $\rho_B$ (resp. $\rho_F$) the invariant probability distribution of $B$ (resp. $F$), that is, the normalised left-eigenvector of $B$ (resp. $F$) associated to the eigenvalue $1$. If $FB=BF,$ then $\rho_B=\rho_F$.
\end{lemma}

\begin{proof}
Note that by Perron-Frobenius, $B$ and $F$ have a unique invariant probability distribution, satisfying respectively $\rho_BB=B$ and $\rho_F F=F$. Since $FB=BF$, we have
$ \rho_B FB = \rho_B B F = \rho_B F,$ 
so that the vector $\rho_B F$ is an invariant probability distribution of $B$. By uniqueness, we obtain $\rho_B F = \rho_B$. Since the invariant probability distribution of $F$ is also unique, we obtain $\rho_B=\rho_F$.
\end{proof}

\begin{definition}\label{def:hzmc}
Let $S$ be a finite set, and let $F$ and $B$ be two transition matrices from $S$ to $S$, such that $FB=BF$. We denote by $\rho$ their (common) left-eigenvector associated to the eigenvalue $1$.
The $(F,B)$-HZMC (for \emph{Horizontal Zigzag Markov Chain}) on $S^{\ZZ_{t}} \times S^{\ZZ_{t+1}}$ is the distribution $\zeta_{F,B}$ such that, for any $n \in \ZZ_t$, for any $a_{-n},a_{-n+2},\ldots,a_n \in S, b_{-n+1},n_{n+3},\ldots,b_{n-1} \in S$, 
\begin{displaymath}
\prob{(\zeta_{F,B}(i,t)= a_i, \zeta_{F,B}(i,t+1) = b_i :-n\leq i\leq n)} = \rho(a_{-n}) \prod_{i=-n+1}^{n-1} F(a_{i-1};b_i) B(b_i;a_{i+1}).
\end{displaymath}
\end{definition}

\begin{figure}
\begin{center}
\begin{tikzpicture}
\foreach \y in {1,2} \draw[blue] (-5,\y) -- (5,\y) ;
\foreach \x in {-4,-2,...,4} \foreach \y in {2} {\draw [fill=red] (\x,\y) circle [radius=0.1] ;}
\foreach \x in {-5,-3,...,5} \foreach \y in {1} {\draw [fill=red] (\x,\y) circle [radius=0.1] ;}
\foreach \x in {-4,-2,...,4} \foreach \y in {2} {\draw[->-,thick] (\x,\y) -- (\x+1,\y-1) ;}
\foreach \x in {-4,-2,...,4} \foreach \y in {2} {\draw[->-,thick] (\x-1,\y-1) -- (\x,\y) ;}
\node at (6,2) {$\eta_{t+1}$};
\node at (6,1) {$\eta_{t}$};

\node[below] at (-5,1) {$a_{-n}$};
\node[below] at (-3,1) {$a_{-n+2}$};
\node[below] at (3,1) {$a_{n-2}$};
\node[below] at (5,1) {$a_{n}$};
\node[above] at (-4,2) {$b_{-n+1}$};
\node[above] at (4,2) {$b_{n-1}$};

\node[left] at (-4.5,1.5) {$F$};
\node[left] at (-3.5,1.5) {$B$};
\node[left] at (-2.5,1.5) {$F$};
\node[left] at (-1.5,1.5) {$B$};

\end{tikzpicture}
\end{center}
\caption{Illustration of Def.~\ref{def:hzmc}.}
\end{figure}

We give a simple necessary and sufficient condition that depends on both $T$ and $(F,B)$ for a $(F,B)$-HZMC to be an invariant measure of a PCA with transition kernel $T$.
\begin{proposition}[Lemma 5.10 of~\cite{Casse17}] \label{prop:Casse17}
Let $S$ be a finite set. Let $A$ be a PCA with positive rates and let $F$ and $B$ be two transition matrices from $S$ to $S$. The $(F,B)$-HZMC distribution is an invariant probability distribution of $A$ iff
\begin{cond} \label{cond:DU-HZMC}
for any $a,c,d \in S$,
\begin{displaymath}
F(a;d)B(d;c) = \sum_{b \in S} B(a;b)F(b;c) T(a,b,c;d).
\end{displaymath}
\end{cond}
\end{proposition}

In the context of PCA having an invariant $(F,B)$-HZMC, Prop.~\ref{prop:ind} can be extended as follows. The proof being similar, we omit it.

\begin{proposition}\label{prop:ind_m}
Let $A$ be a PCA having a $(F,B)$-HZMC invariant measure, of stationary space-time diagram $G(A,\zeta_{F,B})=(\eta_t(i) : t \in \ZZ, i \in \ZZ_t)$. For any zigzag polyline $(i,t_i)_{m\leq i\leq n}$, and any $(a_i )_{m \leq i \leq n}\in S^{n+1}$, we have:
\begin{align*}
\prob{\eta(i,t_i) = a_i : m \leq i \leq n} = \rho(a_0) \mathop{\prod_{m\in\{0,\ldots,n-1\}}}_{t_{i+1}=t_i+1} F(a_i;a_{i+1}) \mathop{\prod_{m\in\{0,\ldots,n-1\}}}_{t_{i+1}=t_i-1} B(a_i;a_{i+1}).
\end{align*}
\end{proposition}

In general, the knowledge of the transition kernel $T$ alone is not sufficient to be able to tell if the PCA $A$ admits or not an invariant $(F,B)$-HZMC. Until now, the characterization of PCA having an invariant $(F,B)$-HZMC is known in only two cases: when $|S|=2$~\cite[Theorem~5.3]{Casse17}, and when $F=B$~\cite[Theorem~5.2]{Casse17}. In the other cases ($F \neq B$ and $|S|>2$), it is an open problem.

\subsection{Quasi-reversibility and reversibility}
This section is devoted to PCA having an HZMC invariant measure, and that are \mbox{(quasi-)}reversible.

\begin{proposition} \label{prop:hrv-qr}
Let $A$ be a PCA having a $(F,B)$-HZMC invariant distribution. Then, the stationary space-time diagram $(A,\zeta_{F,B})$ is $\{h,r^2,v\}$-quasi-reversible, and we have the following.

\begin{itemize}
\item The $h$-reverse is $(A_{h},\zeta_{B,F})$ with, for any $a,b,c,d \in S$,
\begin{displaymath}
T_h(a,d,c;b) = \frac{B(a;b)F(b;c)}{F(a;d)B(d;c)} T(a,b,c;d).
\end{displaymath}

\item The $v$-reverse is $(A_v,\zeta_{B_h,F_h})$ where, for any $a,b,c,d \in S$,
\begin{displaymath}
T_v(c,b,a;d) = T(a,b,c;d),\ B_h(b;a) = \frac{\rho(a)}{\rho(b)} B(a;b) \text{ and } F_h(b;a) = \frac{\rho(a)}{\rho(b)} F(a;b).
\end{displaymath}

\item The $r^2$-reverse is $(A_{r^2},\zeta_{F_h,B_h})$ with, for any $a,b,c,d \in S$,
\begin{displaymath}
T_{r^2}(c,d,a;b) = \frac{B(a;b)F(b;c)}{F(a;d)B(d;c)} T(a,b,c;d).
\end{displaymath}
\end{itemize}
\end{proposition}

\begin{proof} The proof of the $h$-quasi-reversibility is similar to the one of Theorem~\ref{theo:r2rev}, so we omit it, and the fact that $(A,\zeta_{F,B})$ is $v$-quasi-reversible is obvious (see Prop.~\ref{prop:obvious}). Let us denote by $(A_v,\mu_v)$ the $v$-reverse and let us prove that $\mu_v = \zeta_{B_h,F_h}$. For any $i,j \in \Z$, $x_i,y_i,\dots,y_{j-1},x_j \in S$,
\begin{align*}
\mu_v(x_i,y_i,\dots,y_{j-1},x_j) & = \zeta_{F,B}(x_j,y_{j-1},\dots,y_i,x_i) \\
& = \rho(x_j) F(x_j;y_{j-1}) B(y_{j-1};x_{j-1}) \dots F(y_i;x_i)\\
& = F_h(y_{j-1};x_j) \rho(y_{j-1}) B(y_{j-1};x_{j-1}) \dots F(y_i;x_i)\\
& = \dots \\
& = \rho(x_i) F_h(x_i;y_i) B_h(y_i,x_{i+1}) \dots B_h(y_{j-1};x_j)
\end{align*}

The fact that $(A,\zeta_{F,B})$ is $r^2$-quasi-reversible is due to the fact that $r^2 = h \circ v$.
\end{proof}

\begin{proposition} \label{prop:r1DU}
Let $A$ be a PCA having a $(F,B)$-HZMC invariant distribution. $(A,\zeta_{F,B})$ is $r$-quasi-reversible iff
\begin{cond} \label{cond:r1DU}
for any $a,c,d \in S$,
\begin{displaymath}
F(a;d) = \sum_{c\in S} F(b;c) T(a,b,c;d).
\end{displaymath}
\end{cond} 

In that case, the transition kernel of the reverse $A_r$ is given, for any $a,b,c,d \in S$, by:
\begin{equation} \label{eq:rrevtm}
T_r(d,a,b;c) = \frac{F(b;c)}{F(a;d)} T(a,b,c;d).
\end{equation}
\end{proposition}

\begin{proof} The proof follows the same idea as the proof of Prop.~\ref{prop:r1rev}, and uses~Prop.~\ref{prop:ind_m}, the analog of Prop.~\ref{prop:ind}.

$\bullet$ Suppose that $A$ is $r$-reversible.
Then, for any $a,b,c,d$, any $x \in \Ze$,
\begin{align}
& T_r(d,a,b;c)= \prob{\eta(x+\vv)=c | \eta(x+\vu+\vv)=d,\eta(x+\vu)=a,\eta(x)=b} \nonumber \\
& = \frac{\rho(a) B(a;b) F(b;c) T(a,b,c;d)}{\sum_{c'\in S} \rho(a) B(a;b) F(b;c') T(a,b,c';d)} = \frac{F(b;c) T(a,b,c;d)}{\sum_{c'\in S} F(b;c') T(a,b,c';d)}. \label{eq-rev-mark}
\end{align}

For some $x \in \Ze$, let us reintroduce the pattern $L=(x,x+\vu,x+\vv,x+2\vu,x+\vu+\vv,x+2\vu+\vv)$, see Fig.~\ref{fig:pattern}. For $a_0,b_0,b_1,c_0,c_1,d_0 \in S$, we are interested in the quantity:
$Q(a_0,b_0,b_1,c_0,c_1,d_0)=\prob{\eta(L)=(a_0,b_0,b_1,c_0,c_1,d_0)}.$

On the one hand, we have: 
$$Q(a_0,b_0,b_1,c_0,c_1,d_0)=\rho(c_0) B(c_0;b_0) B(b_0;a_0) F(a_0;b_1) T(b_0,a_0,b_1;c_1) T(c_0,b_0,c_1;d_0).$$ 
On the other hand, we have:
$$Q(a_0,b_0,b_1,c_0,c_1,d_0)=\sum_{b'_1,c'_1\in S}Q(a_0,b_0,b'_1,c_0,c'_1,d_0)T_r(d_0,c_0,b_0;c_1)T_r(c_1,b_0,a_0;b_1).$$

Using the expressions of $T_r(d_0,c_0,b_0;c_1)$ and $T_r(c_1,b_0,a_0;b_1)$ given by \eqref{eq-rev-mark} and simplifying, we get:

\begin{align*}
&\left( \sum_{b\in S} F(a_0;b) T(b_0,a_0,b;c_1) \right) \left(\sum_{c\in S} F(b_0;c) T(c_0,b_0,c;d_0) \right)\\
&= F(b_0;c_1) \sum_{b,c\in S} F(a_0;b) T(b_0,a_0,b;c) T(c_0,b_0,c;d_0).
\end{align*}
Now summing on $d_0 \in S$, we find, for any $b_0,c_1\in S$,
$F(b_0;c_1) = \sum_{b\in S} F(a_0;b) T(b_0,a_0,b;c_1)$.

$\bullet$ Conversely, suppose that $A$ is a PCA having an invariant measures $(F,B)$-HZMC and that~\C~\ref{cond:r1DU} holds. Then, we can perform flips thanks to \eqref{eq:rrevtm} as in Fig.~\ref{fig:flips1} and~\ref{fig:flips2}.  
\end{proof}

\begin{proposition}
Let $A$ be a PCA whose an invariant probability distribution is a $(F,B)$-HZMC distribution. $(A,\zeta_{(F,B)})$ is $r^{-1}$-quasi-reversible iff
\begin{cond} \label{cond:r3DU}
for any $b,c,d$,
\begin{displaymath}
\frac{p(d)}{p(c)} B(d;c) = \sum_{a\in S} \frac{p(a)}{p(b)} B(a;b) T(a,b,c;d).
\end{displaymath}
\end{cond}

In that case, the transition kernel of the reverse $A_{r^{-1}}$ is given, for any $a,b,c,d \in S$, by:
\begin{equation} 
T_{r^{-1}}(b,c,d;a) = \frac{B_h(b;a)}{B_h(c;d)} T(a,b,c;d).
\end{equation}
\end{proposition}

\subsection{PCA with an explicit invariant law that is not Markovian}\label{sec:notmarkov}

As evocated in Section~\ref{sec:qrPCAinv}, there exist PCA $A$ of $\triangE{S}(p)$ that are $r$-quasi-reversible, and for which the $r$-reverse $A_r$ does not belong to $\triangE{S}(p)$. In that case, $A_r$ has an invariant measure $\mu = (\pi_p)_r$ which is not a product measure, and for which we know formula allowing to compute exactly all the marginals, see equations \eqref{calcul_expl1} and \eqref{calcul_expl2}. Let us point out that the measure $\mu$ can not be HZMC. Consider indeed the stationary space-time diagram $(A,\pi_p)=(\eta(i,t): (i,t) \in \Ze)$, and assume that $\mu$ is a $(F,B)$-HZMC measure. Then, the marginal of size one of $\mu$ is equal to $\rho=p$, and for any $x\in\Ze$, we have: $\prob{\eta(x+\vu)=a,\eta(x)=b}=p(a)p(b)=\rho(a)F(a;b)$ and $\prob{\eta(x)=b,\eta(x+\vv)=c}=p(b)p(c)=\rho(c)B(c;b)$. Thus, we obtain $F(a;b)=B(c;b)=p(b)$ for any $a,b,c\in S$, meaning that the $(F,B)$-HZMC is in fact a $p$-HZMP, which is not possible since $A_r$ does not belong to $\triangE{S}(p)$.

So, the PCA $A_r$ has an invariant measure that we can compute, and that has neither a product form nor a Markovian one. That was a real surprise of this work. We give an explicit example of such a PCA in Section~\ref{sec:binary}, see Example~\ref{exemple:binaire}.

Similarly, if a PCA $A$ with a $(F,B)$-invariant HZMC is $r$-quasi-reversible, and is such that its $r$-reverse $A_r$ does not have an invariant HZMC, then we can compute exactly the invariant measure of $A_r$, although it does not have a well-known form. This provides an analogous of Theorem~\ref{theo:cool}, in the Markovian case. Precisely, next theorem gives conditions on the transitions of a PCA $C$ for being of the form $C=A_r$, with $A$ having a $(F,B)$-invariant HZMC. To the best knowledge of the authors, this is the first time that we can compute from the transition kernel an invariant law that is not Markovian.

\begin{theorem}\label{theo:cool_hzmc}
Let $C$ be a PCA of transition kernel $T$. For any $a \in S$, let $(F(a;b))_{b\in S}$ be the left eigenvector of $(T(b,a,a;c))_{b,c \in S}$ associated to the eigenvalue $1$ and $(B(a;b))_{b\in S}$ be the left-eigenvector of $(T(a,a,b;c))_{b,c \in S}$. The two following conditions
\begin{cond} \label{cond:eqrDU-D}
for any $a,b,c \in S$,
$F(b;c) = \sum_{d\in S} F(a;d) T(d,a,b;c)$;
\end{cond}

\begin{cond} \label{cond:eqrDU-U}
for any $a,c,d \in S$,
$B(d;c) = \sum_{b\in S} B(a;b) T(d,a,b;c)$;
\end{cond}
are equivalent to: there exists a probability measure $\mu$ on $S$ such that 
\begin{enumerate}[(i)]
\item $\mu$ is invariant by $C$,
\item $(C,\mu)$ is $\{r,r^{-1}\}$-quasi-reversible and
\item the $r^{-1}$-reverse is $(C_{r^{-1}},\zeta_{(F,B)})$ with
  $T_{r^{-1}}(a,b,c;d) = \frac{F(a;d)}{F(b;c)} T(d,a,b;c)$.
\item the $r$-reverse is $(C_{r},\zeta_{(B_h,F_h)})$ with
  $T_{r}(a,b,c;d) = \frac{B(c;d)}{B(b;a)} T(b,c,d;a)$ and $B_h$ and $F_h$ as defined in Prop~\ref{prop:hrv-qr}.. 
\end{enumerate}
Moreover, we have explicit formula for the computation of the marginals of $\mu$.
\end{theorem}

\begin{remark}
In general, $\mu$ is not a Markovian law, nevertheless, sometimes it is. In that case, the PCA $C$ is of the form $C=A_r$, with a PCA $A$ that is not only $r$-quasi-reversible but also $r^{-1}$-quasi-reversible. Note also that in Theorem~\ref{theo:cool_hzmc}, we are able to find the expression of the invariant HZMC of $C_{r^{-1}}$ (resp. $C_r$) from the transition kernel $T$ of $C$, whereas in all generality, given the values of the transition kernel $T_{r^{-1}}$ (resp. $T_r$), we are unable to say if the associated PCA has an invariant HZMC.
\end{remark}

\begin{proof}
Let us assume that there exists a probability measure $\mu$ on $S$ satisfying {\it (i)}, {\it (ii)}, and {\it (iii)}. 
Then, summing the equation of {\it (iii)} on $d\in S$, we find, for any $a,b,c\in S$,
\begin{displaymath}
\sum_{d\in S} F(a;d) T(d,a,b;c) = F(b;c) \sum_{d\in S} T_{r^{-1}}(a,b,c;d) = F(b;c).
\end{displaymath}
For $a=b$, this equation shows that $(F(a;d))_{d\in S}$ is a left-eigenvector of $(T(d,a,a;c))_{d,c \in S}$ associated to $1$.

Moreover, as $\zeta_{(F,B)}$ is invariant by $C_{r^{-1}}$, by Prop.~\ref{prop:Casse17}, for any $a,c,d\in S$,
\begin{align*}
F(a;d) B(d;c) & = \sum_{b\in S} B(a;b) F(b;c) T_{r^{-1}}(a,b,c;d)\\
& = \sum_{b\in S} B(a;b) F(a;d) T(d,a,b;c).
\end{align*}
Dividing by $F(a,d)$ on both sides, we get \C~\ref{cond:eqrDU-U}, and for $d=a$ we obtaint that $(B(a;c))_{c\in S}$ is the left-eigenvector of $(T(a,a,b;c))_{b,c \in S}$.

Conversely, let us define
\begin{displaymath}
\tilde{T}(a,b,c;d) = \frac{F(a;d)}{F(b;c)}T(d,a,b;c).
\end{displaymath}
Then, as \C~\ref{cond:eqrDU-D} and \C~\ref{cond:eqrDU-U} hold, we can check that $\tilde{T}$ is a transition kernel and satisfies \C~\ref{cond:DU-HZMC}, so $\zeta_{(F,B)}$ is an invariant measure of $\tilde{C}$, PCA whose transition kernel is $\tilde{T}$. Moreover,
\begin{displaymath}
\sum_{c\in S} F(b;c) \tilde{T}(a,b,c;d) = \sum_{c\in S} F(a;d) T(d,a,b;c) = F(a;d).
\end{displaymath}
That is \C~\ref{cond:r1DU}, and we conclude by application of Prop.~\ref{prop:r1DU} and by Lemma~\ref{lem-std} (uniqueness of the $r$-reverse). Finally, multidimensional laws of $\mu$ are deduced from the space-time diagram $(C,\mu)$. 
Indeed, we know that rotated by $\pi/2$, it has the same distribution as $(C_r,\zeta_{F,B})$. So, we can compute all the finite-dimensional marginals of the space-time diagram, and in particular the multidimensional laws of $\mu$.
\end{proof}

\section{Applications to statistical physics}\label{sec:sp}

We now develop four examples of PCA with memory two, that are inspired from statistical physics. The first two ones are defined on a finite symbol set, while the third one is defined on the alphabet $\ZZ$, and the last one is defined on a continuous set of symbols. Formal definitions making rigorous these last two models will be given in Section~\ref{sec:general}, in a more general context.

\subsection{The $8$-vertex models}

Let us recall the notations $\vu=(1,1), \vv=(-1,1)$. For some $n\in 2\Z$, we consider the graph $G_n$ whose set of vertices is $V_n=\Ze\cap [-n,n]^2$, the restriction of the even lattice to a finite box, and whose set of edges is
$E_n=\{(x,x+\vu) : x,x+\vu\in V_n\}\cup\{(x,x+\vv) : x,x+\vv\in V_n\}.$ We define the boundary of $V_n$ by $\partial V_n=\{(x_1,x_2)\in V_n : \max(|x_1|,|x_2|)=n\}.$

For each edge of $G_n$, we choose an orientation. This defines an orientation $O$ of $G_n$, and we denote by $O_n$ the set of orientations of $G_n$. For a given orientation $O\in O_n$, and an edge $e \in E_n$, we denote:
\begin{displaymath}
o(e) = \begin{cases}
0 & \text{if the edge $e$ is oriented from top to bottom in $O$ ($\searrow$ or $\swarrow$)},\\
1 & \text{if the edge $e$ is oriented from bottom to top in $O$ ($\nwarrow$ or $\nearrow$)}.
\end{cases}
\end{displaymath}
Hence, an orientation $O\in O_n$ can be seen as an element $(o(e))_{e \in E_n}$ of $\{0,1\}^{E_n}$.

Around each vertex $x\in V_n\setminus \partial V_n$, there are 4 oriented edges, giving a total of 16 possible local configurations, defining the \emph{type} of the vertex $x$. In the 8-vertex model case, we consider only the orientations $O$ such that around each vertex $x\in V_n\setminus \partial V_n$, there is an even number ($0$, $2$ or $4$) of incoming edges, so that only $8$ local configurations remain, see Fig.~\ref{fig:8v}. To each local configuration $i$ among these 8 local configurations, we associate a local weight $w_i$. This allows to define a global weight $W$ on the set $\tilde O_n$ of admissible orientations, by:
\begin{equation}
W(O) = \prod_{x\in V_n\setminus \partial V_n} w_{\text{type}(x)}, \quad \mbox{ for } O\in\tilde O_n.
\end{equation}  
Thanks to these weights, we finally define a probability distribution $\mathbb{P}_W$ on $\tilde O_n$, by:
\begin{equation}
\pw{O} = \frac{W(O)}{\sum_{O \in \tilde O_n} W(O)}.
\end{equation}

\begin{figure}
\begin{center}
\begin{tabular}{c|c|c|c}
$w_1=w_2=a$ & $w_3=w_4=b$ & $w_5=w_6=c$ & $w_7=w_8=d$ \\ \hline
\espace \haut{1} \includegraphics[scale=0.5]{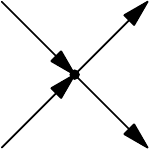} & \haut{3}\includegraphics[scale=0.5]{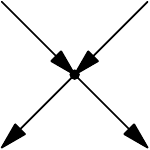} &  \haut{5} \includegraphics[scale=0.5]{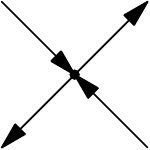} &  \haut{7} \includegraphics[scale=0.5]{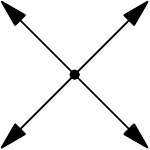}\\
\haut{2} \includegraphics[scale=0.5]{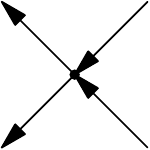} &  \haut{4} \includegraphics[scale=0.5]{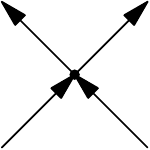} &  \haut{6}\includegraphics[scale=0.5]{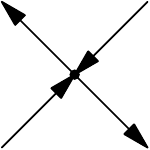} & \haut{8}\includegraphics[scale=0.5]{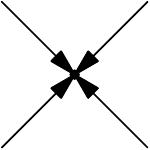}
\end{tabular}
\end{center}
\caption{The 8 possible local configurations around any vertex.}\label{fig:8v}
\end{figure}

As usual in statistical physics, the partition function $\sum_{O \in \tilde O_n} W(O)$ is denoted by $Z_n$. In the following, we consider the more studied 8-vertex model, for which the parameters satisfy $w_1=w_2=a$, $w_3=w_4=b$, $w_5=w_6=c$ and $w_7=w_8=d$. We furthermore assume that $a+c=b+d$.

The 8-vertex model was introduced by Sutherland~\cite{Sutherland70} and Fan and Wu~\cite{FW70} in 1970 as a generalization of the 6-vertex model (for which $d=0$), which was introduced by Pauling in 1935 to study the ice in two dimension~\cite{Pauling35}. In~\cite{Baxter72}, Baxter computes the partition function via Bethe's ansatz methods and deduces 5 asymptotic behaviours for the 8-vertex model~\cite[Section 10.11]{Baxter82}. For the interested reader, we recommend~\cite[Chapter 8]{Baxter82},~\cite{DGHMT16} and reference therein for more information on 6-vertex model and~\cite[Chapter 10]{Baxter82} and reference therein for more information on 8-vertex model.

In~\cite[Section~10.2]{Baxter82}, Baxter presents a ``two-to-one'' map $\mathcal{C}_8$ between 2-colorings of faces of $G_n$ and admissible orientations of the 8-vertex model on $G_n$. Let $F_n$ be the set of faces of $G_n$, that is, the set of quadruplet $(x,x+\vu,x+\vu+\vv,x+\vv)\in (\Ze)^4$ for which at least 3 of the 4 vertices belong to $G_n$.
 The map is the following. Let $C\in\{0,1\}^{F_n}$ be a 2-coloring of faces of $G_n$, and take any edge $e \in E_n$. We denote by $f_e$ and $f'_e$ the two adjacent faces of $e$. Then, we define: 
\begin{equation}
o(e) = \begin{cases}
1 & \text{if } C(f_e) = C(f'_e), \\
0 & \text{otherwise ({\it i.e.} if $C(f_e) \neq C(f'_e)$)}.
\end{cases}
\end{equation}
It is a ``two-to-one'' map because from an admissible orientation $O$, we obtain two 2-colorings in $\mathcal{C}_8^{-1}(O) = \{C_0,C_1\}$. These two colorings have the following properties $C_0(f)=1-C_1(f)$ for any $f \in F_n$, see Fig.~\ref{fig:col}. 

\begin{figure}
\begin{center}
\includegraphics{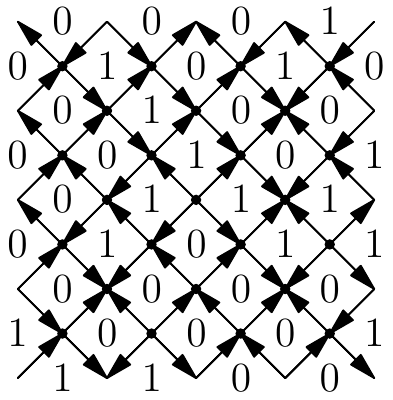} \hspace{10mm} \includegraphics{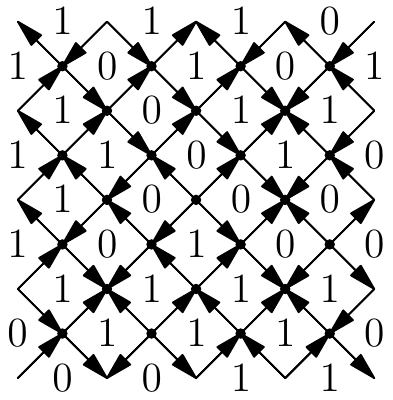}\par
\end{center}
\caption{An orientation $O$ and its two possible 2-colorings.}\label{fig:col}
\end{figure}

Let us set $q = a/(a+c)$ and $r = b/({b+d})$, and consider the PCA $A_8$ whose transition kernel $T$ is defined, by
\begin{align*}
&T(0,0,1; \cdot)=T(1,0,0; \cdot)=\ber(q),\\
&T(0,1,1; \cdot)=T(1,1,0; \cdot)=\ber(1-q),\\
&T(0,1,0; \cdot)=T(1,1,1; \cdot)=\ber(r),\\
&T(1,0,1; \cdot)=T(0,0,0; \cdot)=\ber(1-r).
\end{align*}
This is the PCA presented as an introductory example in Section~\ref{sec:introdef}. With this PCA, we define a random 2-coloring of $G_n$ in the following way. First, we color the faces centered on points of ordinate $-n$ and $-n+1$ (first two lines) and the faces centered on points of abscisse $-n$ and $n$ (left and right boundary conditions), independently, with common law $\mathcal{B}(1/2)$. 
Then, we color the other faces by applying successively the PCA $A_8$, from bottom to top.
We denote by $\mathcal{F}_n$ the law of the random 2-coloring of $G_n$ obtained.

\begin{proposition} [\cite{Casse17}]
If $C \sim \mathcal{F}_n$, then $\mathcal{C}_8(C) \sim {\mathbb P}_W$.
\end{proposition}

This proposition is the first application of PCA with memory two in the literature. One can check that the PCA $A_8$ satisfies~\C~\ref{cond-pm} with $p(0)=p(1)=1/2$. The proof of Theorem~\ref{theo:ergo} implies that this PCA is ergodic. When $n \to \infty$, the center of the square has the same behaviour whatever are the boundary conditions~\cite[Proposition~1.6]{Casse17}.

Note that in what precedes, we have assumed that the weights satisfy the relation $a+c=b+d$. If we now assume that they rather satisfy $a+d=b+c$, we can design a PCA that, when iterated from left to right (or equivalently, from right to left), generates configurations distributed according to the required distrbution ${\mathbb P}_W$. When $a=b$ and $c=d$, so that both relations are satisfied, we obtain $q=r$, and the dynamics is $D_4$-reversible.

\subsection{Directed animals and gaz models}

A directed animal on the square lattice (resp. on the triangular lattice) is a set $A \subset \Ze$ such that $(0,0) \in A$ and, for any $z \in A$ there exists a directed path $w=((0,0) = x_0, x_1 , \dots, x_{m-1}, x_m=z)$ such that, for any $1 \leq k \leq m$,
\begin{displaymath}
x_k-x_{k-1} \in \{u,v\} \text{ (resp. $\{u,v,u+v\}$)}.
\end{displaymath}
Let us denote by $\mathcal{A}_S$ (resp. $\mathcal{A}_T$) the set of directed animals on the square (resp. triangular) lattice.

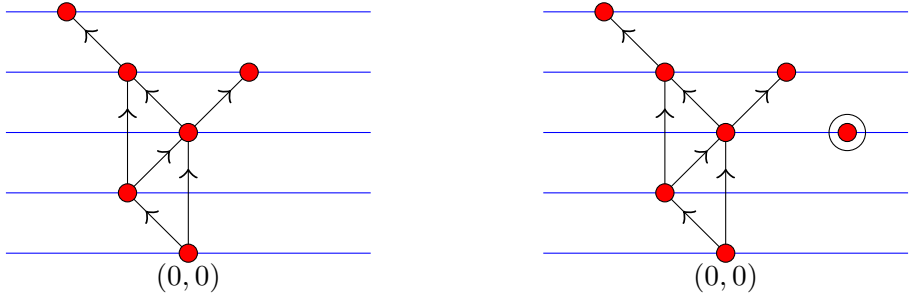
\begin{figure}\label{figure:da}
\begin{center}
\begin{tikzpicture}[scale=0.8]
\foreach \y in {0,1,2,3,4} \draw[blue] (-3,\y) -- (3,\y) ;
\foreach \x in {0} \foreach \y in {2} {\draw[->-] (\x,\y-2) -- (\x,\y) ;}
\foreach \x in {-1} \foreach \y in {3} {\draw[->-] (\x,\y-2) -- (\x,\y) ;}
\foreach \x in {-1} \foreach \y in {3} {\draw[->-] (\x+1,\y-1) -- (\x,\y) ;}
\foreach \x in {-1} \foreach \y in {1} {\draw[->-] (\x+1,\y-1) -- (\x,\y) ;}
\foreach \x in {-2} \foreach \y in {4} {\draw[->-] (\x+1,\y-1) -- (\x,\y) ;}
\foreach \x in {0} \foreach \y in {2} {\draw[->-] (\x-1,\y-1) -- (\x,\y) ;}
\foreach \x in {1} \foreach \y in {3} {\draw[->-] (\x-1,\y-1) -- (\x,\y) ;}
\node[below] at (0,0) {$(0,0)$};
\draw [fill=red] (0,0) circle [radius=0.15];
\draw [fill=red] (-1,1) circle [radius=0.15];
\draw [fill=red] (0,2) circle [radius=0.15];
\draw [fill=red] (-1,3) circle [radius=0.15];
\draw [fill=red] (-2,4) circle [radius=0.15];
\draw [fill=red] (1,3) circle [radius=0.15];
\end{tikzpicture}
\hspace{2cm}
\begin{tikzpicture}[scale=0.8]
\foreach \y in {0,1,2,3,4} \draw[blue] (-3,\y) -- (3,\y) ;
\foreach \x in {0} \foreach \y in {2} {\draw[->-] (\x,\y-2) -- (\x,\y) ;}
\foreach \x in {-1} \foreach \y in {3} {\draw[->-] (\x,\y-2) -- (\x,\y) ;}
\foreach \x in {-1} \foreach \y in {3} {\draw[->-] (\x+1,\y-1) -- (\x,\y) ;}
\foreach \x in {-1} \foreach \y in {1} {\draw[->-] (\x+1,\y-1) -- (\x,\y) ;}
\foreach \x in {-2} \foreach \y in {4} {\draw[->-] (\x+1,\y-1) -- (\x,\y) ;}
\foreach \x in {0} \foreach \y in {2} {\draw[->-] (\x-1,\y-1) -- (\x,\y) ;}
\foreach \x in {1} \foreach \y in {3} {\draw[->-] (\x-1,\y-1) -- (\x,\y) ;}
\node[below] at (0,0) {$(0,0)$};
\draw [fill=red] (0,0) circle [radius=0.15];
\draw [fill=red] (-1,1) circle [radius=0.15];
\draw [fill=red] (0,2) circle [radius=0.15];
\draw [fill=red] (-1,3) circle [radius=0.15];
\draw [fill=red] (-2,4) circle [radius=0.15];
\draw [fill=red] (1,3) circle [radius=0.15];
\draw [fill=red] (2,2) circle [radius=0.15];
\draw (2,2) circle [radius=0.3];
\end{tikzpicture}
\caption{The set on the left is a directed animal on the triangular lattice, while the set on the right is not.}
\end{center}
\end{figure}

The \emph{area} of an animal $A$ is the cardinal of $A$ and the \emph{perimeter} of an animal $A$ is the cardinal of
$P(A) = \{x : x \notin A, \{x\} \cup A \text{ is a directed animal}\}$. 
Let us introduce the generating functions of directed animals enumerated according to their area, on the square lattice and on the triangular lattice:
\begin{equation}
G_S(z) = \sum_{A \in \mathcal{A}_S} z^{|A|} \qquad \qquad G_T(z) = \sum_{A \in \mathcal{A}_T} z^{|A|}. 
\end{equation}
The computation of $G_S$ was done by Dhar in 1982 via the study of hard-particles model~\cite{Dhar82}. Here, we will present this work
using PCA, see also~\cite{mm_tcs} for details.  
Let $B_S$ be the binary state PCA with memory one whose transition kernel $T_S$ is given, for any $a,b \in \{0,1\}$, by
\begin{displaymath}
T_S(a,b;1) = 
\begin{cases}
p_S & \text{if } a=b=0,\\
0 & \text{else}.
\end{cases}
\end{displaymath} 

\begin{theorem}[\cite{Dhar82,BousquetMelou98,LBM07}] \label{theo:Dhar82}
For any $p_S$, let $(\eta(i,t): (i,t) \in \Ze)$ be the space-time diagram of $B_S$ under its invariant probability measure, 
then
\begin{equation}
\prob{\eta(0,0) = 1 } = - G_S(-p_S)
\end{equation}
\end{theorem}

Note that the uniqueness of the invariant measure of $B_S$, for any choice of $p_S\in(0,1)$, was proven in~\cite{hmm}. Theorem~\ref{theo:Dhar82} was generalized by~\cite{LBM07} for directed animals on any ``admissible'' graph. 
In the case of directed animal on the square lattice, the invariant measure of $B_S$ has a simple Markovian form (see~\cite{Dhar82,BousquetMelou98,LBM07,CM15}), so that we can recover the following result.
\begin{theorem}[\cite{Dhar82}]
The area generating function of directed animals on the square lattice is
\begin{equation}
G_S(z) =  \frac{1}{2} \left( \left(1 - \frac{4z}{1+z} \right)^{-1/2} - 1 \right)
\end{equation}
\end{theorem}

In~\cite{BousquetMelou98}, the enumeration of directed animals on the square and triangular lattices was done according to others statistics.

\begin{theorem}[\cite{Dhar82,BousquetMelou98}]
The area generating function of directed animals on the triangular lattice is
\begin{equation}
G_T(z) = \frac{1}{2} \left( \left(1 - 4z \right)^{-1/2} - 1 \right)
\end{equation}
\end{theorem}

Observe that the following property holds.

\begin{lemma} \label{lem:eq}
\begin{equation} \label{eq:eq}
G_T\left(\frac{z}{1+z}\right) = G_S(z).
\end{equation}
\end{lemma}

Here, we will give a proof of this lemma using only results of~\cite{LBM07} that generalize Theorem~\ref{theo:Dhar82}, and \cite[Theorem~5.3]{Casse17} on PCA. 

\begin{proof}
Let $B_T$ be the binary state PCA with memory two of transition kernel $T_T$ given, for any $a,b,c\in \{0,1\}$ by
\begin{displaymath}
T_T(a,b,c;1) = 
\begin{cases}
p_T & \text{if} \; a=b=c=0,\\
0 & \text{otherwise}.
\end{cases}
\end{displaymath}

Then by Theorem~2.7 of~\cite{LBM07} applied to the triangular lattice, we get that:
if $(\eta(i,t): (i,t) \in \Ze)$ is the space-time diagram of $B_T$ taken under its invariant measure, then 
\begin{equation} \label{eq:LBM07}
G_T(-p_T) = \prob{\eta(i,t) = 1}.
\end{equation}

Now, let us prepare to apply~\cite[Theorem~5.3]{Casse17} to $B_T$. For any $a,c \in \{0,1\}$, the left eigenvector of $(T_T(a,b,c;d))_{b,d \in \{0,1\}}$ is
\begin{displaymath}
T(a,c;1) =
\begin{cases}
\displaystyle \frac{p_T}{1+p_T} & \text{if } a=c=0,\\
0 & \text{otherwise}.
\end{cases}
\end{displaymath}
Hence, the associated PCA with memory one is $B_S$ with $\displaystyle p_S = \frac{p_T}{1+p_T}$. As $B_S$ satisfy conditions of~\cite[Theorem~5.3]{Casse17}, we obtain that the invariant measure of $B_T$ and of $B_S$ with $\displaystyle p_S = \frac{p_T}{1+p_T}$ is the same. So, by Theorem~\ref{theo:Dhar82} and~\ref{eq:LBM07},
\begin{equation}
-G_T(-p_T) = -G_S \left( -\frac{p_T}{1+p_T} \right).
\end{equation}
Taking $x=-p_T$, we obtain $\displaystyle G_T(x) = G_S \left(\frac{x}{1-x}\right)$, which is equivalent to~\eqref{eq:eq} when $\displaystyle x = \frac{z}{1+z}$.

An attentive reader would have seen that we have used~\cite[Theorem~5.3]{Casse17} for a PCA with non-positive rates. This is possible under some conditions on $T_T$ and for this transition it works well. Nevertheless, the necessary and sufficient condition are not known in general. We refer the interested reader to~\cite[Section~4.4]{BousquetMelou98} and~\cite[Section~2.2]{CM15} for some sufficient conditions and remarks about PCA with non-positive rates.
\end{proof}

Some words about the enumeration of directed animal by area and perimeter. Let 
\begin{equation}
\tilde G_S(x,y) = \sum_{A \in \mathcal{A}_S} x^{|A|} y^{|P(A)|} \quad \mbox{ and } \quad
\tilde G_T(x,y) = \sum_{A \in \mathcal{A}_T} x^{|A|} y^{|P(A)|}
\end{equation}
be the generating function of directed animal enumerated according to their area and perimeter on, respectively, square and triangular lattice. Let us introduce two PCA $\tilde B_S$ and $\tilde B_T$ of alphabet $S=\{0,1\}$. The PCA $\tilde B_S$ has memory one and transition kernel $\tilde T_S$, and the PCA $\tilde B_T$ has memory two and transition kernel $\tilde T_T$, with:
\begin{displaymath}
\tilde T_S(a,b;1) = 
\begin{cases}
  p+q & \text{if } a=b=1,\\
  p & \text{otherwise}
\end{cases}
\quad \mbox{ and }\quad
\tilde T_T(a,b,c;1) = 
\begin{cases}
  p+q & \text{if } a=b=c=1,\\
  p & \text{otherwise}.
\end{cases}
\end{displaymath}
For $p$ sufficiently close to $0$, these PCA can be proven to be ergodic, and we have the following result.
\begin{theorem}[{\cite[Theorem~4.3]{LBM07}}]
For $p$ sufficiently close to $0$, let $\tilde \eta_S$ (resp. $\tilde \eta_T$) be the space-time diagram of $\tilde B_S$ (resp. $\tilde B_T$) taken under its invariant measure. Then
\begin{displaymath}
\prob{\tilde \eta_S(0,0)=1} = q + \tilde G_S(p,q) \qquad (\mbox{resp. } \prob{\tilde \eta_T(0,0)=1} = q + \tilde G_T(p,q)).
\end{displaymath}
\end{theorem}

Unfortunately, we have no explicit description of the invariant measures of these PCA.

\subsection{Synchronous TASEP of order two}

The TASEP (Totally ASymmetric Exclusion Process) describes the evolution of some particles that go from the left to the right on a line without overtaking. There are various kinds of models of TASEP models, with discrete or continuous time and space, and one or more types of particles.  We refer the interested readers to the following articles~\cite{belitsky,gray} for the description of some models with discrete time and space. Here, we present a new (to the best knowledge of the authors) generalization of TASEP called TASEP of order two on real line and discrete time.

The TASEP presented here models the behaviour of an infinite number of particles (indexed by $\ZZ$) 
on the real line, that move to the right, that do not bypass and that do not overlap. For $i,t\in\Z$, we denote by $x_i(t)\in \RR$ the position of particle $i$ at time $t$. Time is discrete, and at time $t$, each particle $i\in\ZZ$ moves with a random speed $v_i(t)$, independently of the others. The random speed $v_i(t)$ depends on the distance $x_{i+1}(t)-x_i(t)$ between the particle $i$ and the particle $i+1$ in front of it, and of the speed $v_{i+1}(t-1)=x_{i+1}(t)-x_{i+1}(t-1)$ of the particle $i+1$ at time $t-1$. Formally, 
the evolution of $(x_i(t))_{i \in \ZZ}$ is defined by: 
\begin{equation}
\forall i\in\ZZ, \quad x_i(t+1) = x_i(t) + v_i(t),
\end{equation}   
where $v_i(t)$ is random and distributed following $\mu_{(x_{i+1}(t)-x_i(t),v_{i+1}(t-1))}$, 
a probability distribution on $\RR^+$, and $(v_i(t))_{i \in \ZZ}$ are independent,  knowing $(x_i(t))_{i\in \ZZ}$ and $(x_{i}(t-1))_{i \in \ZZ}$.\par
It is known that TASEP with discrete time can be represented by PCA~\cite{mm_tcs,Casse16}. We adopt the sight presented in~\cite{Casse16} to show that the TASEP of order two can be represented by a PCA with memory two: take $\eta(i,t) = x_i(t)$, then $(\eta(i,t) : i \in \ZZ, t \in \NN)$ is the space-time diagram of a PCA with memory two whose transition kernel $T$ is, for any $a \in \RR, x,y,v \in \RR^+$,
\begin{equation}
T(a,a+x,a+x+y;a+v) = \mu_{(x+y,y)}(v). 
\end{equation}
\par
Now, we will focus as an example on the simplest case where 
$v \in \{0,1\}$ a.s.\ and particles move on the integer line (for any $i\in\ZZ,t\in\NN$, $x_i(t) \in \ZZ$). The constraints we have on $T$ are the following:
\begin{itemize}
\item $T(a,a+1,a+1;a) =1$ for any $a\in \ZZ$ (and $T(a,b,b+i;c)$ does not matter if $i \neq 0,1$ or $b \leq a$),
\item $T(a,a+k,a+k+i;c) = 0$ for any $a$, $k \geq 0$, $i \in \{0,1\}$ and $c \notin \{a,a+1\}$,
\item $T(a,a+k,a+k+i;a) = T(b,b+k,b+k+i;b)$ for any $k\geq 1$, $i \in \{0,1\}$, $a,b \in \ZZ$. 
\end{itemize}
The first two points signify that the next position has to be empty for a particle to move, and that a particle can only move of one unit forward. The last point is an hypothesis of translation invariance. 
Hence, the PCA can be described by the transitions $(T(0,k,k;0))_{k \geq 2}$ and $(T(0,k,k+1;0))_{k \geq 1}$.\par

\begin{figure}
\begin{center}
\includegraphics[scale=0.8]{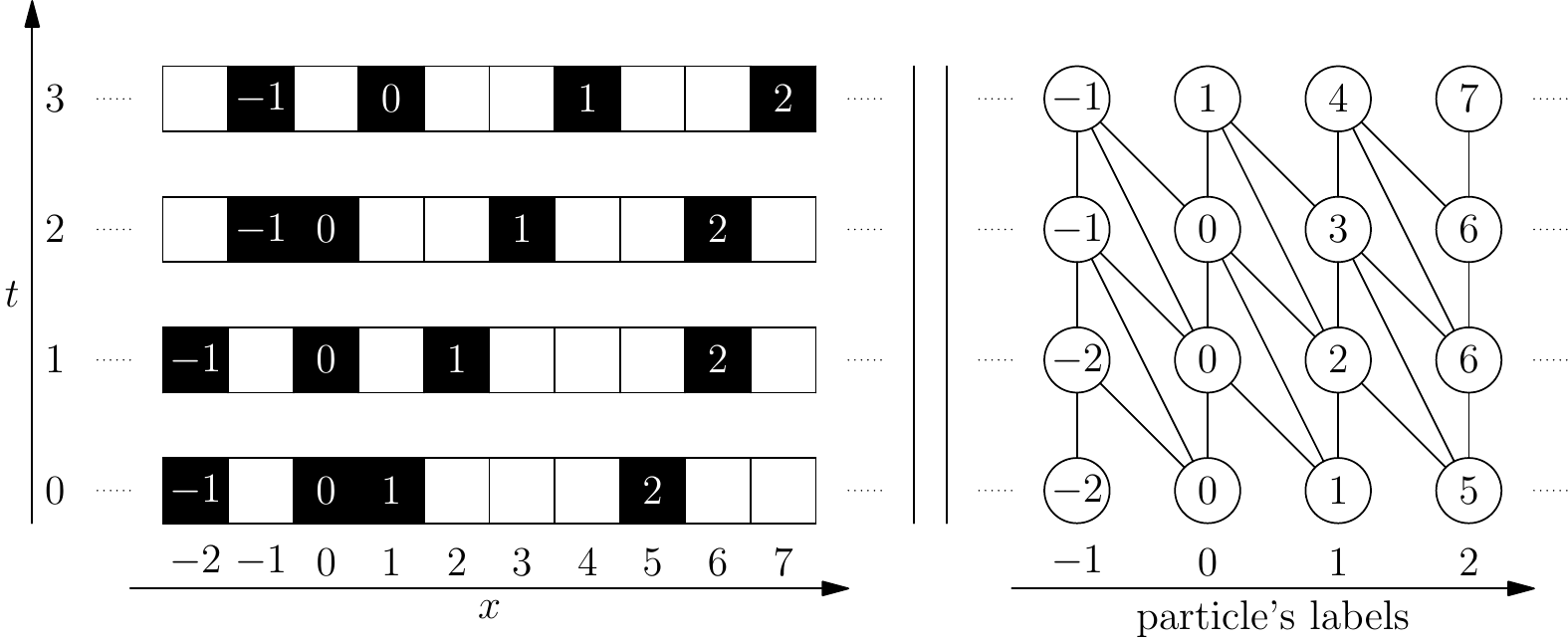}
\end{center}
\caption{On the left, the classical representation of a TASEP: a white square is an empty square; a black square is a square that contains a particle, the white number is the label of this particle. On the right, the PCA that represents this TASEP; each column represent the trajectory of a particle.}
\end{figure}

The first result of this section is about the fact that there exists a family of $(F,B)$-HZMC that is stable by this PCA. 
First, let us define a $q,p$-HZMC for any probability $q$ and $p$ on $\ZZ$: a $q,p$-HZMC is a $(F,B)$-HZMC such that, for any $a,k \in \ZZ$, $F(a;a+k) = q(k)$ and $B(a;a+k) = p(k)$.

\begin{lemma} \label{lem:TASEP}
For any transition kernel $T$, if their exist $p$ a probability on $\NN^*$ and $q$ on $\{0,1\}$ such that, for any $k \geq 1$, 
\begin{equation} \label{eq:TASEP}
p(k) q(1) T(0,k,k+1;0) + p(k+1) q(0) T(0,k+1,k+1;0) = p(k+1)q(0)
\end{equation}
then if we start under the law such that $(\eta_0,\eta_1)$ is a $q,p$-HZMC with $\eta_0(0)=0$ a.s., then any double line $(\eta_t,\eta_{t+1})$ is also distributed as a $p,q$-HZMC but the starting point is now $\eta_t(0)$ with 
\begin{displaymath}
\prob{\eta_t(0) = k} = \binom{t}{k} q(1)^k q(0)^{t-k}.
\end{displaymath}
\end{lemma}

Note that \eqref{eq:TASEP} also implies
\begin{equation}\label{eq2:TASEP}
p(k) q(1) T(0,k,k+1;1) + p(k+1) q(0) T(0,k+1,k+1;1) = p(k)q(1),
\end{equation}
both equations being equivalent to: 
\begin{equation}\label{eqequiv:TASEP}
p(k) q(1) T(0,k,k+1;0) = p(k+1) q(0) T(0,k+1,k+1;1).
\end{equation}
These two conditions \eqref{eq:TASEP} and \eqref{eq2:TASEP} are similar to \C~\ref{cond:DU-HZMC} of Prop.~\ref{prop:Casse17}.

\begin{proof} The proof is done by induction on $t\in\NN$. For $t=0$,  we assume that $(\eta_0,\eta_1)$ is a $q,p$-HZMC  with $\eta_0(0)=0$ a.s. 
Now, let us suppose that $(\eta_t,\eta_{t+1})$ is a $q,p$-HZMC with $\displaystyle \prob{\eta_t(0) = k} = \binom{t}{k} q(1)^k q(0)^{t-k}$. Then, by conditionning by the possible values $(a_i)_{0 \leq i \leq k+1}$ for $(\eta_t(i))_{0 \leq i \leq k+1}$, we obtain that the finite dimensional laws of $(\eta_{t+1},\eta_{t+2})$ are given by:
\begin{align*}
& \prob{(\eta_{t+1}(i)=b_i)_{0 \leq i \leq k},(\eta_{t+2}(i)=c_i)_{0 \leq i \leq k-1}}\\
= & \sum_{a_0,\dots,a_{k+1}} \binom{t}{a_0} q(1)^{a_0} q(0)^{t-a_0}  q(b_0-a_0) \prod_{i=0}^{k-1}  p(a_{i+1}-b_i) q(b_{i+1}-a_{i+1}) T(b_i,a_{i+1},b_{i+1};c_i) \\
= & \left(\sum_{a_0
} \binom{t}{a_0} q(1)^{a_0} q(0)^{t-a_0} q(b_0-a_0)\right) \; \left(\prod_{i=0}^{k-1} \; \sum_{a_{i+1}
} p(a_{i+1}-b_i) q(b_{i+1}-a_{i+1}) T(b_i,a_{i+1},b_{i+1};c_i)\right).
\end{align*}
Since the only non-zero terms correspond to $a_i\in\{b_i,b_i-1\}$, the left parenthesis is equal to:
$$\sum_{a_0\in\{b_0,b_0-1\}} \binom{t}{a_0} q(1)^{a_0} q(0)^{t-a_0} q(b_0-a_0)=\binom{t+1}{b_0} q(1)^{b_0} q(0)^{t+1-b_0},$$
and the right one to:
\begin{align*}
&\prod_{i=0}^{k-1} \; \sum_{a_{i+1}\in\{b_{i+1},b_{i+1}-1\}} p(a_{i+1}-b_i) q(b_{i+1}-a_{i+1}) T(b_i,a_{i+1},b_{i+1};c_i)\\
&=\prod_{i=0}^{k-1} \; \sum_{a_{i+1\in\{b_{i+1},b_{i+1}-1\}}} p(a_{i+1}-b_i) q(b_{i+1}-a_{i+1}) T(0,a_{i+1}-b_i,b_{i+1}-b_i;c_i-b_i) \\
&=\prod_{i=0}^{k-1} p(b_{i+1}-c_i) q(c_i-b_i), \quad \mbox{using \eqref{eq:TASEP} and \eqref{eq2:TASEP}}.
\end{align*}
\end{proof}

We can remark that $q$ is the speed law of a particle under the stationary regime and $p$ the distance law between two successive particles (to be precise the left one at current time $t$ and the right one at previous time $t-1$).\par

\begin{theorem} \label{theo:TASEP}
For any $T$, for any distribution $q$ on $\{0,1\}$ such that 
\begin{equation} \label{eq:TASEPZ}
Z = \sum_{k=0}^\infty \left(\frac{q(1)}{q(0)}\right)^k \prod_{m=1}^k \frac{T(0,m,m+1;0)}{T(0,m+1,m+1;1)}  < \infty,
\end{equation} 
there exists a unique distribution $p$ on $\NN^*$ such that \eqref{eq:TASEP} hold.\par
Moreover, this distribution $p$ is, for any $k \geq 1$,
\begin{equation} \label{eq:TASEPp}
p(k) = \frac{\displaystyle \left(\frac{q(1)}{q(0)}\right)^{k-1} \prod_{m=1}^{k-1} \frac{T(0,m,m+1;0)}{T(0,m+1,m+1;1)}}{Z}.
\end{equation}
\end{theorem}

\begin{proof}
Let $q$ be a probability measure on $\{0,1\}$. 
By \eqref{eqequiv:TASEP}, we have: 
\begin{equation} \label{eq:TASEPind}
\forall k\geq 1, \quad p(k+1) = p(k) \frac{T(0,k,k+1;0)}{T(0,k+1,k+1;1)}  \frac{q(1)}{q(0)}.
\end{equation}
By induction, we obtain:
\begin{equation} 
\forall k\geq 1, \quad p(k+1)=\left(\frac{q(1)}{q(0)}\right)^k \prod_{m=1}^k \frac{T(0,m,m+1;0)}{T(0,m+1,m+1;1)} p(1)
\end{equation}
As $\sum_{k\in\NN^*} p(k) = 1$, we need~\eqref{eq:TASEPZ}. In that case, \eqref{eq:TASEPp} follows.
\end{proof}

In the classical case of synchronous TASEP (presented in~\cite[Sections 2.3 \& 4.3]{mm_tcs},~\cite[Section 3.3]{Casse16}), we have 
\begin{equation} \label{eq:classTASEP}
T(0,k,k+1;1) = T(0,k+1,k+1;1) = p.
\end{equation}
With Theorem~\ref{theo:TASEP}, we recover the invariant measures of the classical synchronous TASEP.\par

This example is interesting because it does not enter in our previous framework for many reasons. First, it is easy to see that studying an invariant measure for this PCA is not interesting because it corresponds to the overloaded state where nobody move ($q(0)=1$). That's why we focused here on a family of distributions that is stable by the PCA and not only on one distribution.\par
Moreover, the PCA has not positive rates for any $\mu$,
because we cannot get any configuration starting from any configuration. Nevertheless, studying carefully their eigenvectors on the good subspace, we solve the algebraic issues to find interesting results.

In addition, we find some results about a $(F,B)$-HZMC family (with $F \neq B$) and a PCA with an infinite alphabet, whereas our main result on PCA is about PCA with invariant $(F,B)$-HZMC but alphabet of size $2$ or PCA with a general alphabet but with invariant $(F,F)$-HZMC (see end of Section~\ref{sec:invHZMC}).

\subsection{Eden model on the triangular lattice} 

The Eden model is an aggregation model that was defined by Murray Eden in 1961~\cite{Eden61}. It describes a growth model on $\ZZ^2$ which growths by perimeter starting from a point. Here, we develop an Eden model on the triangular lattice as the one of~\cite{AP15} on the square lattice.\par
Let $\mu$ be a probability measure on $[0,\infty)$. The graph $G$ we consider is the one with set of nodes $\{(x,y) \in \Ze : y \geq 0\}$ and set $E$ of edges that are the ones supported by vectors $\vu$, $\vv$ and $\vu+\vv$. To each edge $e$ of $G$, we associate a positive random variable $\omega(e)$ with distribution $\mu$. For every directed path $\gamma$ in $G$, we denote $\lambda(\gamma) = \sum_{e \in \gamma} \omega(e)$ the passage time of $\gamma$. The passage time between two connected nodes $x$ and $y$ is
\begin{displaymath}
d(x;y) = \inf_{\gamma \in \Gamma(x;y)} \lambda(\gamma) 
\end{displaymath}
where $\Gamma(x;y)$ is the set of directed paths starting from $x$ and finishing in $y$. Finally, we define the passage time $\eta(x)$ on a node $x = (x_1,x_2) \in \Ze$ as
\begin{displaymath}
\eta(x) = \begin{cases}
0 & \text{if } x_2 \in \{0,1\},\\
\displaystyle \min_{y \in \ZZ_0\times\{0\}\cup\ZZ_1\times\{1\}} d(y;x) & \text{otherwise.}
\end{cases}
\end{displaymath}

This is a stationary version of the first-passage percolation on a directed triangular lattice with $\mu$ as the law of the time to travel an edge. In the special case, where $\mu$ is distributed as an exponential random variable we obtain a stationary version of the Eden model on a directed triangular lattice.\par

This model can be seen as a PCA $A$ of order two with alphabet $E = [0,\infty)$ where the state of a node $x \in \Ze$ is $\eta(x)$. In that case, for any $a,b,c\in E$, the law $T(a,b,c;.)$ is the one of $\min (a + \omega_1, b+\omega_2, c + \omega_3)$, where $\omega_1$, $\omega_2$ and $\omega_3$ are i.i.d.\ of common law $\mu$. Unfortunately, the theorems presented in this article do not apply to this PCA.

\section{The binary case}\label{sec:binary}

In this section, we specify the conditions obtained in Section~\ref{sec:rqrPCA} to the case of a binary symbol set: $S=\{0,1\}$. 

\begin{proposition}\label{prop:condbinaire} For a binary symbol set $S=\{0,1\}$, 
\begin{enumerate}
\item \C~\ref{cond-pm} is equivalent to  
\begin{cond} \label{cond1:2col}
$\forall a,c\in S, \quad p(0)T(a,0,c;1)=p(1)T(a,1,c;0).$ 
\end{cond}
\item \C~\ref{cond:r1rev} is equivalent to 
\begin{cond} \label{cond2:2col}
$\forall a,b\in S, \quad p(0)T(a,b,0;1)=p(1)T(a,b,1;0).$ 
\end{cond}
\item \C~\ref{cond:r3rev} is equivalent to 
\begin{cond} \label{cond3:2col}
$\forall b,c\in S, \quad p(0)T(0,b,c;1)=p(1)T(1,b,c;0).$ 
\end{cond}
\end{enumerate}
\end{proposition}

\begin{proof} In the binary case, \C~\ref{cond-pm} reduces to: $\forall a,c\in S, \; p(1)=p(0)T(a,0,c;1)+p(1)T(a,1,c;1),$ which is itself equivalent to \C~\ref{cond1:2col}. The proof is analogous for \C~\ref{cond2:2col} and \C~\ref{cond3:2col}.
\end{proof}

Let $p$ be a probability measure on $S=\{0,1\}$. If we specify some results of Table~\ref{table:pres} to the case $|S|=2,$ we obtain:
\begin{align*}
& \dim \left( \triang{S}{p} \right) = \dim \left( \{A \in \triang{S}{p} : \text{$A$ is $h$-reversible}\} \right) = 4,\\
& \dim \left( \{A \in \triang{S}{p} : \text{$A$ is $v$-reversible}\} \right) = \dim \left( \{A \in \triang{S}{p} : \text{$A$ is $r^2$-reversible}\} \right) = 3,\\
& \dim \left( \{A \in \triang{S}{p} : \text{$A$ is $D_4$-quasi-reversible}\} \right) = \dim \left( \{A \in \triang{S}{p} : \text{$A$ is $D_4$-reversible}\} \right) = 1.
\end{align*}

In this section, we will describe more precisely these different sets, which will give an alternative proof of the value of their dimension, in the binary case. First, next result shows that in the binary case, the sets above having the same dimension are equal.

\begin{proposition}
Let $p$ be any positive probability on $S=\{0,1\}$, and let $A\in\triang{S}{p}$. Then, we have the following properties.
\begin{enumerate}
\item $A$ is $h$-reversible.
\item $A$ is $v$-reversible iff $A$ is $r^2$-reversible.
\item $A$ is $D_4$-quasi-reversible iff $A$ is $D_4$-reversible.
\end{enumerate}
\end{proposition}

\begin{proof}\begin{enumerate}
\item Since $A$ is in $\triang{S}{p}$, $A$ is $h$-quasi-reversible, and the transition kernel $T_h$ of its $h$-reverse satisfies, for any $a,b,c,d \in S$,
\begin{displaymath}
T_h(a,d,c;b) = \frac{p(b)}{p(d)} T(a,b,c;d).
\end{displaymath}
For $b=d,$ this gives $T_h(a,b,c;d)=T(a,b,c;d),$ and for $b\not=d$, \C~\ref{cond1:2col} provides the result.
\item It is a corollary of 1. Indeed, if $A$ is in $\triang{S}{p}$ and $v$-reversible, then it is $h$ and $v$-reversible, and so also $r^2 = v \circ h$-reversible. And conversely, if it is $r^2$-reversible, then it is $v=r^2\circ h$-reversible.
\item This will be a consequence of Theorem~\ref{theo:drev2col}.
\end{enumerate}
\end{proof}

As a consequence of Prop.~\ref{prop:condbinaire}, we obtain the following descriptions of binary PCA having an invariant HZPM.

\begin{theorem} \label{theo:r-rev2col}
  Let $A$ be a PCA with transition kernel $T$ (with positive rates). Then $A$ has an invariant HZPM iff 
  \begin{cond} \label{cond:HZPM2}
    there exists $k \in ]0,\infty[$ such that for any $a,c \in S$,
    \begin{displaymath}
      \frac{T(a,1,c;0)}{T(a,0,c;1)} = k.
    \end{displaymath}
  \end{cond}
More explicitly, this is equivalent to the following condition.
  \begin{cond} \label{cond:HZPM2-exp}
    there exists $k \in ]0,\infty[$ and
    $\begin{cases} 
      q_{0,0},q_{0,1},q_{1,0},q_{1,1} \in (0,1) & \text{if } k \in (0,1]\\ 
      q_{0,0},q_{0,1},q_{1,0},q_{1,1} \in (1-k^{-1},1) & \text{if } k \in [1,\infty) 
    \end{cases}$ 
    such that, for any $a,c\in S$,
    \begin{align*}
      &T(a,0,c;0) = q_{a,c},\\
      &T(a,1,c;0) = k (1-q_{a,c}) = k-kq_{a,c}.
    \end{align*}
  \end{cond}

In that case, the $p$-HZPM invariant is $(p(0),p(1))$ where $\displaystyle p(1) = 1 - p(0) = \frac{1}{1+k}$.
\end{theorem}

\begin{proof} The PCA $A$ has an invariant HZPM iff there exists a probability $p$ on $S$ such that \C~\ref{cond1:2col} is satisfied, which can easily be shown to be equivalent to the above conditions.
\end{proof}

\begin{proposition} \label{prop:r2rev2}
Let $p$ be a positive probability on $S$, and let $k=p(0)/p(1)$. 

Then, the PCA $A$ is a $r$-quasi-reversible of $\triang{S}{p}$ iff
  \begin{cond} \label{cond:unorder3}
    there exists 
    $\begin{cases} 
      q_0,q_1 \in (0,1) & \text{if } k \in (0,1]\\ 
      q_0,q_1 \in (1-k^{-1},1-k^{-1}+k^{-2}) & \text{if } k \in [1,\infty) 
    \end{cases}$ 
    such that, for any $a \in S$,
    \begin{align*}
      &T(a,0,0;0) = q_a,\\
      &T(a,0,1;0) = T(a,1,0;0) = k (1-q_c) = k-kq_a\\
      &T(a,1,1;0) = k (1-k(1-q_c)) = k-k^2+k^2q_a
    \end{align*}
 \end{cond}   
Similarly, the PCA $A$ is a $r^{-1}$-quasi-reversible of $\triang{S}{p}$ iff
  \begin{cond} \label{cond:unorder4}
    there exists 
    $\begin{cases} 
      q_0,q_1 \in (0,1) & \text{if } k \in (0,1]\\ 
      q_0,q_1 \in (1-k^{-1},1-k^{-1}+k^{-2}) & \text{if } k \in [1,\infty) 
    \end{cases}$ 
    such that, for any $c \in S$,
    \begin{align*}
      &T(0,0,c;0) = q_c,\\
      &T(0,1,c;0) = T(1,0,c;0) = k (1-q_c) = k-kq_c\\
      &T(1,1,c;0) = k (1-k(1-q_c)) = k-k^2+k^2q_c
    \end{align*}
  \end{cond}
\end{proposition}

\begin{proof} We prove the first statement. Let $A$ be a $r$-quasi-reversible PCA in $\triang{S}{p}$. 
Then $T$ satisfies \C~\ref{cond1:2col} and~\ref{cond2:2col}, meaning that for any $a,b,c \in S$,
    \begin{displaymath}
      \frac{T(a,1,c;0)}{T(a,0,c;1)} = \frac{T(a,b,1;0)}{T(a,b,0;1)} = k = \frac{p(0)}{p(1)}.
    \end{displaymath}
Taking $b=c=0$, we find, for any $a \in S$,
  \begin{displaymath}
    T(a,1,0;0) = T(a,0,1;0) = k T(a,0,0;1) 
  \end{displaymath}
  and taking $b=c=1$, we find that for any $a \in S$,
  \begin{displaymath}
    T(a,0,1;1) = T(a,1,0;1) = k^{-1} T(a,1,1;0).
  \end{displaymath}
Hence, for any $a \in S$, we get
  \begin{displaymath}
    T(a,1,1;0) = k (1-T(a,0,1;0)) = k (1 - k (1-T(a,0,0;0))).
  \end{displaymath}
  
  Then, every $T(a,b,c;d)$ can be express in terms of $T(0,0,0;0) =q_0$, $T(0,0,1;0) =q_1$ and $k$, which gives \C~\ref{cond:unorder3}, and the range of $q_0,q_1$ is deduced from the fact that, for any $a,b,c,d \in S$, $T(a,b,c;d) \in (0,1)$.

Conversely, let $A$ be such that \C~\ref{cond:unorder3} holds. Then \C~\ref{cond1:2col} and \ref{cond2:2col} hold, so $A \in \triang{S}{p}$ and $A$ is $r$-quasi-reversible.
\end{proof}

\begin{theorem} \label{theo:drev2col}
Let $p$ be a positive probability on $S$, and let $k=p(0)/p(1)$. 

Then, the PCA $A$ is a $\{r,r^{-1}\}$-quasi-reversible of $\triang{S}{p}$ iff
  \begin{cond} \label{cond:unorder2}
    there exists 
    $\begin{cases} 
      q_0 \in (0,1) & \text{if } k \in (0,1]\\ 
      q_0 \in (1-k^{-1}+k^{-2}-k^{-3},1-k^{-1}+k^{-2}) & \text{if } k \in [1,\infty) 
    \end{cases}$ 
    such that
    \begin{align*}
      &T(0,0,0;0) = q_0,\\
      &T(0,0,1;0) = T(0,1,0;0) = T(1,0,0;0) = k (1-q_0) = k-kq_0\\
      &T(0,1,1;0) = T(1,1,0;0) = T(1,0,1;0) = k (1-k(1-q_0)) = k-k^2+k^2q_0\\
      &T(1,1,1;0) = k(1-k(1-k(1-q_0))) = k -k^2 + k^3 -k^3q_0.
    \end{align*}
  \end{cond}

  Moreover, in that case, $A$ is $D_4$-reversible.
\end{theorem}

\begin{proof} The PCA $A$ is a $\{r,r^{-1}\}$-quasi-reversible PCA of $\triang{S}{p}$ iff \C~\ref{cond1:2col}, \ref{cond2:2col} and~\ref{cond3:2col} are satisfied, which can easily be shown to be equivalent to the above condition.

Now, we prove the $D_4$-reversibility of $A$. First, $A$ is symmetric, so $A$ is $v$-reversible. Second, the $r$-reverse of $A$ is the PCA 
$A_r$ with transition kernel $T_r$ given by~\eqref{eq:r1rev} on p.\pageref{eq:r1rev}. In particular,~\eqref{eq:r1rev} provides: ${T_r}(0,0,0;0)={T}(0,0,0;0)$ and ${T_r}(0,0,1;0)={T}(0,1,0;0)={T}(0,0,1;0)$. 
Furthermore, $A_r$ is $r^{-1}$-quasi-reversible, and by Theorem~\ref{theo:r-rev}, we have: $A_r\in\triang{S}{p}$. So, $A_r$ must satisfy \C~\ref{cond:unorder4}, which allows to express all the transitions of  $T_r$ from the values of $T_r(0,0,0;0)$ and $T_r(0,0,1;0)$. These values being the same as for $A$, which is also $r^{-1}$-quasi-reversible, it follows that $T_r=T$. Since $A$ is $r$ and $v$-reversible, by $(5)$ of Prop.~\ref{prop:obvious}, $A$ is $D_4$-reversible.
\end{proof}

\begin{example} Let us consider the special case when $p$ is the uniform distribution on $S$, meaning that $p(0)=p(1)=1/2$. Then, $k=1$, and the family of PCA above corresponds to:
$$\forall a,b,c,d\in S, \quad T(a,b,c;d)=
\begin{cases} 
q_0 \; \mbox{ if } \; d=a+b+c \mod 2\\
1-q_0 \; \mbox{ otherwise. }
\end{cases}
$$
In the deterministic case ($q_0=1$), we get a linear CA. Such CA have been intensively studied. Here, in the probabilstic setting, the PCA we obtain can be seen as noisy versions of that linear CA (with a probability $1-q_0$ of doing an error, independently for different cells). This is a special case of the $8$-vertex PCA, with $p=r$.
\end{example}

\begin{example}\label{exemple:binaire} 
Let us consider the probability distribution on $S$ given by $p(0) = 1/3$ and $p(1) = 2/3$, so that $k=2$. When specifying \C~\ref{cond:unorder3} to $q_0=3/4$ and $q_1 = 4/5$, we obtain:
$$
\begin{array}{ll}
\begin{array}{l}
T(0,0,0;0) = 3/4,\\
T(0,0,1;0) = T(0,1,0;0) = 1/8,\\
T(0,1,1;0) = 7/16,
\end{array}
&
\begin{array}{l}
T(1,0,0;0) = 4/5,\\
T(1,0,1;0) = T(1,1,0;0) = 1/10,\\
T(1,1,1;0) = 9/20.
\end{array}
\end{array}
$$
The PCA $A$ of transition kernel $T$  is $r$-quasi-reversible, but one can check that it does not satisfy \C~\ref{cond3:2col}, so that it is not $r^{-1}$-quasi-reversible. So, $T_r$ does not belong to $\triang{S}{p}$, and following the argument developed in Section~\ref{sec:notmarkov}, $T_r$ does not have an invariant HZMC either. Nevertheless, one can compute exactly the marginals of its invariant measure $\mu$, see \eqref{calcul_expl1} and \eqref{calcul_expl2}. The transtions of $T_r$ are the following one:
$$\begin{array}{ll}
\begin{array}{l}
T_r(0,0,0;0) = 3/4,\\
T_r(0,0,1;0) = 1/8,\\
T_r(0,1,0;0) = 4/5,\\
T_r(0,1,1;0) = 1/10,
\end{array}
&
\begin{array}{l}
T_r(1,0,0;0) = 1/8,\\
T_r(1,0,1;0) = 7/16,\\
T_r(1,1,0;0) = 1/10,\\
T_r(1,1,1;0) = 9/20.
\end{array}
\end{array}$$
\end{example}

\section{Extension to general alphabet} \label{sec:general}
We now present some extensions of our methods and results to general set of symbols. First of all, we extend the definition of PCA to any Polish space $S$, as it has been done in~\cite{Casse16} for PCA with memory one. The transition kernel $T$ of a PCA with memory two must now satisfies:
\begin{itemize}
\item for any Borel set $D \in \borel{S}$, the map $\app{T_D}{S^3}{\RR}{(a,b,c)}{T(a,b,c;D)}$ is $\mathcal{B}(S^3)$-mesurable;
\item for any $a,b,c\in S$, the function $\app{T_{a,b,c}}{\borel{S}}{\RR}{D}{T(a,b,c;D)}$ is a probability measure on~$S$.
\end{itemize}

For any $\sigma$-finite measure $\mu$ on $S$, the transition kernel $T$ is said to be $\mu$-positive if, for $\mu^3$-almost every $(a,b,c)\in S^3$, $T(a,b,c;.)$ is absolutely continuous according to $\mu$ and $\mu$ is absolutely continuous according to $T(a,b,c;.)$. In that case, thanks to Radon-Nikodym theorem, we can define the density of $T$ according to $\mu$, that is a $\mu^4$-measurable positive function where, for $\mu^3$-almost every $(a,b,c)\in S^3$, 
\begin{equation}
t(a,b,c;d) = \frac{\text{d} T(a,b,c;.)}{\text{d} \mu}(d)
\end{equation}
where $\displaystyle \frac{\text{d} T(a,b,c;.)}{\text{d} \mu}$ is the Radon-Nikodym derivative of $T(a,b,c;.)$ according to $\mu$.

\begin{theorem}
 Let $\mu$ be any $\sigma$-finite measure on a Polish space $S$. Let $A$ be a PCA with a $\mu$-positive transition kernel $T$ on $S$. Then, $A$ has an invariant $\mu$-positive HZPM iff
\begin{cond} \label{cond:HZPM-polonais}
there exist $\mu$-measurable positive function $p$ on $S$ such that, for $\mu^3$-almost every $(a,c,d) \in S^3$,
\begin{equation}
p(d) = \int_E p(b) t(a,b,c;d) \diff{\mu(b)}
\end{equation}
and
\begin{equation}
\mu(p) = \int_E p(b) \diff{\mu(b)} < \infty,
\end{equation}
where $t$ is the $\mu$-density of $T$.
\end{cond}

Then, the $P$-HZPM is invariant by $A$ where $p(.)/\mu(p)$ is the $\mu$-density of $P$.
\end{theorem}

\begin{proof}
The proof follows the same idea that the one of Theorem~\ref{theo:gen0}, except that we are now on a Polish space $S$. Let $A$ be a $\mu$-positive triangular PCA with alphabet $S$.\par

$\bullet$ Suppose that $A$ has an invariant $\mu$-positive $P$-HZPM and that $(\eta_t,\eta_{t+1})$ follows a $P$-HZPM distribution. Then, for any $\tilde A,\tilde B, \tilde C,\tilde D \in \borel{S}$, 
\begin{align*}
& \prob{\eta_t(i-1) \in \tilde A , \eta_{t+1}(i) \in \tilde D, \eta_t(i+1) \in \tilde C} \\
& \quad = \int_{\tilde A \times \tilde C \times \tilde D} p(a) p(c) p(d) \diff{\mu^3(a,c,d)} & \text{ on the one hand,}\\
& \quad = \int_{\tilde A \times \tilde C \times \tilde D} \left( \int_S p(a) p(b) p(c) t(a,b,c;d) \diff{\mu(b)} \right) \diff{\mu^3(a,c,d)} & \text{ on the other hand.}
\end{align*}
Hence, for $\mu$-almost $a,c,d \in S$,
\begin{align*}
p(a) p(c) \int_S p(b) t(a,b,c;d) \diff{\mu(b)} = p(a) p(d) p(c)
\end{align*}
and so, as $p(a),p(c) > 0$ for $\mu$-almost $a,c\in S$, \C~\ref{cond:HZPM-polonais} holds.\par

$\bullet$ Conversely, assume that \C~\ref{cond:HZPM-polonais} is satisfied, and that $(\eta_{t-1},\eta_t)$ follows a $\mu$-positive $P$-HZPM distribution. For some given choice of $n\in\Z_{t}$, let us denote: $X_i=\eta_{t-1}(n+1+2i), Y_i=\eta_{t}(n+2i), Z_i=\eta_{t+1}(n+1+2i),$ for $i\in \Z$, see Fig.~\ref{fig:hzpm_proof} on p. \pageref{fig:hzpm_proof} for an illustration. Then, for any $k\geq 1$, for any $\mu$-measurable Borel sets $B_0,B_1,\dots,B_k, C_0,\dots,C_{k-1}$,  
\begin{align*}
& \prob{(Y_i)_{0 \leq i \leq k} \in B_0\times \cdots \times B_k, (Z_i)_{0 \leq i \leq k-1} \in C_0 \times \cdots \times C_{k-1}} \\
& = \int_{C_0 \times \dots \times C_{k-1}} \int_{B_0 \times \dots \times B_{k}} \left( \prod_{i=0}^{k-1} \int_{S} t(y_i,x_{i},y_{i+1};z_i) p(x_{i})  \diff{\mu(x_{i})} \right) \nonumber \\
& \qquad \qquad p(y_0) \dots p(y_{k}) \diff{\mu(y_0,\dots,y_{k})} \diff{\mu(z_0,\dots,z_{k-1})}\\
& = \int_{C_0 \times \dots \times C_{k-1}} \int_{B_0 \times \dots \times B_{k}} \left( \prod_{i=0}^{k-1} p(z_{i}) \right) p(y_0) \dots p(y_{k}) \diff{\mu(y_0,\dots,y_{k})} \diff{\mu(z_0,\dots,z_{k-1})}
\end{align*}
thus, the $P$-HZPM distribution is invariant by $A$.
\end{proof}

Now, the problem is reduced to find eigenfunction associated to the eigenvalue $1$ of some integral operator. If this problem is solved by Gauss elimination in the case of a finite space, this is more complicated in the general case. Indeed, such a function does not always exist, but, when it is the case, the solution is unique (up to a multiplicative constant), see the following lemma.
\begin{lemma}[{Durrett~\cite[Theorem 6.8.7]{Durrett10}}]
Let $\mathcal{A}$ be an integral operator of kernel $m$:
\begin{displaymath}
  \mathcal{A}: f \to \left(\mathcal{A}(f) : y \to \int_S f(x) m(x;y) \diff \mu(x)\right).
\end{displaymath}
If $m$ is the $\mu$-density of a $\mu$-positive t. k.\ $M$ from $S$ to $S$, then $\mathcal{A}$ possesses at most one positive eigenfunction in $L^1(\mu)$ (up to a multiplicative constant).
\end{lemma}

Moreover, the previous results concerning the characterization of reversible and quasi-reversible PCA extend for PCA with general alphabet. The difference is that we are considering $\mu$-positive PCA and that \C~\ref{cond:r1rev} and~\ref{cond:r3rev} become respectively
\begin{cond}
for $\mu^3$-almost every $(a,b,d) \in S^3$, $\int_{S} p(c) t(a,b,c;d) \diff{\mu(c)} = p(d)$
\end{cond}
and
\begin{cond}
for $\mu^3$-almost every $(b,c,d) \in S^3$, $\int_{S} p(a) t(a,b,c;d) \diff{\mu(a)} = p(d)$.
\end{cond}

Following the same idea as in~\cite{Casse16}, many results on PCA with invariant $(F,B)$-HZMC can also be generalized to PCA on general alphabets.
 
\section{Dimension of the manifolds}\label{sec:dim}
In this section, we give the dimensions of $\triang{S}{p}$ and of its subsets of (quasi)-reversible PCA (see Table~\ref{table:pres}). But first, we need some results about dimensions of sets of matrices.

\subsection{Preliminaries: dimensions of sets of matrices with a given eigenvector}
Let $S$ be a finite set and $u,v$ be two probabilities on $S$. We denote $\mathcal{M}_S(u,v)$ the set of positive matrices $M = (m_{ij})_{i,j \in S}$ such that $M$ is a stochastic matrix and $uM=v$, i.e.
\begin{align*}
\mathcal{M}_S(u,v) = \quad \{ M = (m_{ij})_{i,j \in S} : & \quad \text{for any } i,j \in S,\ 0 < m_{ij} < 1;\\ 
& \quad \text{for any } i \in S,\ \sum_{j\in S} m_{ij} = 1;\\
& \quad \text{for any } j \in S,\ \sum_{i\in S} u(i) m_{ij} = v(j)\}.
\end{align*}
A particular case is when $u=v=p$, in that case, $p$ is a left-eigenvector of $M$ associated to the eigenvalue $1$ and the set is denoted $\mathcal{M}_S(p)$. Moreover, we will need to know the dimension of the subset $\mathcal{M}^{\sym}_S(p)$ of $\mathcal{M}_S(p)$ defined by
\begin{equation}
\mathcal{M}^{\sym}_S(p) = \{ M \in \mathcal{M}_S(p) : \forall i,j \in S,\ p(i) m_{ij} = p(j) m_{ji}\}.
\end{equation}

Our first lemma is about the dimension of $\mathcal{M}_S(u,v)$.
\begin{lemma} \label{lem:dim-mp}
Let $S$ be a finite set of size $n$. Then, 
\begin{equation}
\dim \mathcal{M}_S(u,v) = (n-1)^2.
\end{equation}
\end{lemma}

\begin{proof}
First, we prove $\dim \mathcal{M}_S(u,v) \leq (n-1)^2$. $\mathcal{M}_S(u,v)$ is defined by the $2n$ linear equations $\forall i\in S$, $\sum_{j\in S} m_{ij} = 1$ and $\forall j\in S$, $\sum_{i\in S} u(i) m_{ij} = v(j)$.  This gives $2n-1$ independent linear equations on the $n^2$ variables $(m_{ij})_{i,j \in S}$. So $\dim \mathcal{M}_S(p) \leq n^2 - (2n-1) = (n-1)^2$.

We do not have the equality yet because we have the additional condition: $\forall i,j\in S$, $m_{ij}>0$. Hence, we have to ensure that $\mathcal{M}_S(u,v)$ is not empty, and that we are not in any other degenerate for which the dimension would be strictly smaller than $(n-1)^2$. For that, we first exhibit a solution of the system such that $m_{ij}>0$ and then find a neighbourhood around this solution having the dimension we want.

First, the matrix $M = (v(j))_{i,j \in S}$ is in $M_S(u,v)$. Now, let $s \in S$ be a distinguished element of $S$. Let us set: $S^\star = S \backslash \{s\}$. One can check that there exists a neighbourhood $V_0$ of $0$ in $\RR^{(S^\star)^2}$ such that for any $(\epsilon_{ij} : i,j \in S^\star) \in V_0$, the matrix $M_{\epsilon} = (m_{ij})_{i,j \in S}$ defined by: 
\begin{align*}
& m_{ij} = v(j)+\epsilon_{ij} \text{ for any } i,j \in S^\star,\\
& m_{is} = v(s) - \sum_{j' \in S^\star} \epsilon_{ij'},\\
& m_{sj} = v(j) - \frac{\sum_{i' \in S^\star} u(i') \epsilon_{i'j}}{u(s)},\\
& m_{ss} = v(s) + \frac{\displaystyle\sum_{i' \in S^\star} \sum_{j' \in S^\star} u(i') \epsilon_{i'j'}}{u(s)},
\end{align*}
is positive, stochastic, and satisfies $uM=v$. So $\dim \mathcal{M}_S(p) \geq \dim (S^\star)^2 = (n-1)^2$.
\end{proof}

In the particular case when $u=v=p$, we get
\begin{corollary} \label{cor:dim-mp}
Let $S$ be a finite set of size $n$. Then, 
\begin{equation}
\dim \mathcal{M}_S(p) = (n-1)^2.
\end{equation}
\end{corollary}

Now we give two properties about families of matrices in $\mathcal{M}_S(p)$.
\begin{lemma} \label{lem:div-mp}
Let $S$ be a finite set and $p$ be a probability on $S$. Let $(M_k = (m_{k,ij})_{i,j \in S}:k\in S)$ be a collection of positive matrices indexed by $S$ such that, for any $i,j \in S$,
\begin{equation} \label{eq:div-mp}
\sum_{k \in S} p(k) m_{k,ij} = p(j).
\end{equation}
Let $s \in S$ and define $S^\star = S \backslash \{s\}$. If, for any $k \in S^\star$, $M_k \in \mathcal{M}_S(p)$, then $M_s \in \mathcal{M}_S(p)$.
\end{lemma}

\begin{proof}
By~\eqref{eq:div-mp}, coefficients of the matrix $M_s$ according to the ones of the other matrices is, for any $i,j \in S$,
\begin{displaymath}
m_{s,ij} = \frac{p(j) - \sum_{k \in S^\star} p(k) m_{k,ij}}{p(s)}.
\end{displaymath}

First, let us prove that $M_s$ is stochastic: for any $i$,
\begin{displaymath}
\sum_{j\in S} m_{s,ij} = \frac{1 - \sum_{k\in S^\star} p(k)}{p(s)} = \frac{1 - (1-p(s))}{p(s)} =1;
\end{displaymath}

then, that $p$ is a left-eigenvector of $M_s$: for any $j$,
\begin{align*}
\sum_{i \in S} p(i) m_{s,ij} & = \frac{p(j) - \sum_{k \in S^\star} p(k) \sum_{i\in S} p(i) m_{k,ij}}{p(s)}\\
& = \frac{p(j) - \sum_{k \in S^\star} p(k) p(j)}{p(s)}\\
& = p(j) \frac{1 - (1-p(s))}{p(s)} =p(j).
\end{align*}
\end{proof}

\begin{lemma} \label{lem:div-rev}
Let $S$ be a finite set and let $p$ be a probability on $S$. Let $M = (m_{ij})_{i,j \in S}$ be a matrix in $\mathcal{M}_S(p)$. Then $\tilde{M} = \left(\frac{p(j)}{p(i)} m_{ji}\right)_{i,j \in S} \in \mathcal{M}_S(p)$.
\end{lemma}

\begin{proof}
First, let us prove that $\tilde{M}$ is stochastic: for any $i\in S$,
$\sum_{j \in S} \tilde{m}_{ij} = \sum_{j \in S} \frac{p(j)}{p(i)} m_{ji} = 1$;
then, that $p$ is a left-eigenvector of $\tilde{M}$: for any $j\in S$,
$\sum_{i \in S} p(i) \tilde{m}_{ij} = \sum_{i \in S} p(j) m_{ji} = p(j)$.
\end{proof}

Finally, we get the dimension of $\mathcal{M}^{\sym}_S(p)$.
\begin{lemma} \label{lem:div-sym}
Let $S$ be a finite set of size $n$ and $p$ be a probability on $S$,
\begin{displaymath}
\dim \mathcal{M}^{\sym}_S(p) = \frac{(n-1)n}{2}.
\end{displaymath}
\end{lemma}

\begin{proof}
First, $\dim \mathcal{M}^{\sym}_S(p) \leq \frac{(n-1)n}{2}$ because we know by proof of Lemma~\ref{lem:dim-mp} that we can describe a matrix in the manifold $\mathcal{M}_S(p)$ by knowing $(m_{ij} : i,j \in S^*)$. But, with the new contrain $p(i) m_{ij} = p(j) m_{ji}$ for any $i,j \in S$, it is sufficient to know only $(m_{ij} : i \leq j,\ i,j \in S^*)$.

Conversely, let us take $(m_{ij} : i \leq j,\ i,j \in S^*)$ in a neighbourhood $V$ of $(m_{ij} = p(j) : i \leq j,\ i,j \in S^*)$ in $\mathbb{R}^{n(n-1)/2}$. Let us take:
\begin{itemize}
\item for any $i,j \in S^*$, $i>j$, $m_{ij} = \displaystyle \frac{p(j)}{p(i)} m_{ji}$;
\item for any $i \in S^*$, $m_{is} = 1 - \sum_{j\in S^*} m_{ij}$;
\item for any $j \in S^*$, $m_{sj} = \displaystyle \frac{p(j) - \sum_{i \in S^*} p(i) m_{ij}}{p(s)}$;
\item $m_{ss} = 1 - \sum_{j \in S} m_{sj}$.
\end{itemize}
By the same argument as in the proof of Lemma~\ref{lem:dim-mp}, there exists a neighboorhood $V$ of dimension $\frac{(n-1)n}{2}$ such that for any point on it, $M = (m_{ij})_{ij\in S} \in \mathcal{M}^{\sym}_S(p)$. And, so, 
\begin{displaymath}
\dim \mathcal{M}^{\sym}_S(p) \geq \frac{(n-1)n}{2}.
\end{displaymath}
\end{proof}

These preliminary results will be useful to prove dimensions of sets of (quasi-)reversible PCA.

\subsection{Dimensions of $\triang{S}{p}$ and its subsets}
\begin{theorem} \label{theo:dim}
Let $S$ be a finite set of size $n$ and $p$ be a positive probability on $S$. 
\begin{enumerate}
\item $\dim \left( \triang{S}{p} \right) = n^2 (n-1)^2$.
\item $\dim \left( \{A \in \triang{S}{p} : \text{$A$ is $r$-quasi-reversible}\} \right) = n (n-1)^3$,
\item $\dim \left( \{A \in \triang{S}{p} : \text{$A$ is $r^{-1}$-quasi-reversible}\} \right)  = n (n-1)^3$,
\item $\dim \left( \{A \in \triang{S}{p} : \text{$A$ is $D_4$-quasi-reversible}\} \right) = (n-1)^4$.
\item $\displaystyle \dim \left( \{A \in \triang{S}{p} : \text{$A$ is $v$-reversible}\} \right) = \frac{(n-1)^2 n (n+1)}{2}$.
\item $\displaystyle \dim \left( \{A \in \triang{S}{p} : \text{$A$ is $r^2$-reversible}\} \right) = \frac{(n-1)^2 n (n+1)}{2}$.
\item $\displaystyle \dim \left( \{A \in \triang{S}{p} : \text{$A$ is $h$-reversible}\} \right)  = \frac{n^3 (n-1)}{2}$.
\item $\displaystyle \dim \left( \{A \in \triang{S}{p} : \text{$A$ is $<r^2,v>$-reversible}\} \right) = \frac{(n-1) n^2 (n+1)}{4}$.
\item $\displaystyle \dim \left( \{A \in \triang{S}{p} : \text{$A$ is $<r>$-reversible}\} \right) = \frac{n(n-1) (n^2-3n+4)}{4}$.\footnote{OEIS A006528}
\item $\displaystyle \dim \left( \{A \in \triang{S}{p} : \text{$A$ is $<r \circ v>$-reversible}\} \right) =  \frac{(n-1)^2(n^2-2n+2)}{2}$.\footnote{OEIS A037270}
\item $\displaystyle \dim \left( \{A \in \triang{S}{p} : \text{$A$ is $D_4$-reversible}\} \right) = \frac{n(n-1)(n^2-n+2)}{8}$.\footnote{OEIS A002817}
\end{enumerate}
\end{theorem}

\begin{proof}
Let $s \in S$, $S^\star = S \backslash \{s\}$ and $|S|=n$.
\begin{enumerate}
\item By Theorem~\ref{theo:gen0}, a PCA $A$ is in $\triang{S}{p}$ if for all $a,c \in S$, $(T(a,b,c;d))_{b,d \in S} \in \mathcal{M}_S(p)$. It follows, by Corollary~\ref{cor:dim-mp}, that: $\dim \triang{S}{p} = |S|^2 \dim \mathcal{M}_S(p) = n^2 (n-1)^2$.
\item By Theorem~\ref{theo:gen0}, as $A \in \triang{S}{p}$, for any $a,c \in S$, $(T(a,b,c;d))_{b,d \in S} \in \mathcal{M}_S(p)$. Moreover, $A$ is $r$-reversible so, by Prop.~\ref{prop:r1rev},
\begin{displaymath}
\sum_{c\in S} p(c) T(a,b,c;d) = p(d). 
\end{displaymath}
By Lemma~\ref{lem:div-mp}, for any $a \in S$,  we can choose freely $(T(a,b,c;d))_{b,d \in S} : c \in S^\star) \in \mathcal{M}_S(p)$ and $(T(a,b,s;d))_{b,d \in S}$ is then uniquely obtained from them and in $\mathcal{M}_S(p)$. That's why 
\begin{displaymath}
\dim \left( \{A \in \triang{S}{p} : \text{$A$ is $r$-quasi-reversible}\} \right) = |S| |S^*| \dim \mathcal{M}_S(p) = n (n-1)^3.
\end{displaymath}

\item The proof is similar to the previous one. 

\item As before, for any $a,c \in S$, $(T(a,b,c;d))_{b,d \in S} \in \mathcal{M}_S(p)$. By Theorem~\ref{theo:r-rev}, we need in addition that, for any $b,d \in S$,
\begin{displaymath}
\sum_{a\in S} p(a) T(a,b,c;d) = p(d) \text{ for any $c \in S$,} \text{ and } \sum_{c\in S} p(c) T(a,b,c;d) = p(d) \text{ for any $a \in S$.}
\end{displaymath}

Hence, we can choose freely a collection of $|S^*|^2$ matrices $((T(a,b,c;d))_{b,d \in S} \in \mathcal{M}_S(p) : a,c \in S^\star)$. Then, by Lemma~\ref{lem:div-mp}, matrices $((T(s,b,c;d))_{b,d\in S} :c \in S^\star)$ and $(T(a,b,s;d)_{b,d \in S} : a \in S^\star)$ are uniquely defined and in $\mathcal{M}_S(p)$. Finally, the last matrix $(T(s,b,s;d))_{b,d \in S}$ can be obtained from two various methods but define the same matrix at the end (the proof is similar to the one of Lemma~\ref{lem:dim-mp}). Hence,
\begin{displaymath}
\dim \left( \{A \in \triang{S}{p} : \text{$A$ is $D_4$-quasi-reversible}\} \right) = |S^*|^2 \dim \mathcal{M}_S(p) = (n-1)^4.
\end{displaymath}

\item If $A \in \triang{S}{p}$ is $v$-reversible, then $(T(a,b,c;d))_{b,d \in S} \in \mathcal{M}_S(p)$ and $T(a,b,c;d) = T(c,b,a;d)$. So, matrices $\{(T(a,b,a;d))_{b,d \in S} : a \in S\}$ can be chosen freely in $\mathcal{M}_S(p)$, but as $T(a,b,c;d) = T(c,b,a;d)$ when $a \neq c$, hence only $\{ (T(a,b,c;d))_{b,d \in S} : a<c \}$ can be choosen freely in $\mathcal{M}_S(p)$, $\{ (T(a,b,c;d))_{b,d \in S} : a>c \}$ are imposed by $\{ (T(c,b,a;d))_{b,d \in S} : c<a \}$. Hence
\begin{align*}
\dim \left( \{A \in \triang{S}{p} : \text{$A$ is $v$-reversible}\} \right)
& = \left( |S| + \binom{|S|}{2} \right) \dim  \mathcal{M}_S(p) \\
& = \left( n + \frac{n(n-1)}{2} \right) (n-1)^2\\
& = \frac{(n-1)^2 n (n+1)}{2}.
\end{align*}

\item If $A \in \triang{S}{p}$ is $r^2$-reversible, then $(T(a,b,c;d))_{b,d \in S} \in \mathcal{M}_S(p)$ and $T(c,d,a;b) = \frac{p(b)}{p(d)}T(a,b,c;d)$. Hence, if we take a matrix $(T(a,b,c;d))_{b,d \in S} \in \mathcal{M}_S(p)$ with $a<c$, then $(T(c,b,a;d))_{b,d \in S}$ is known and $\in \mathcal{M}_S(p)$ by Lemma~\ref{lem:div-rev}. So, we can just choose freely matrices $(T(a,b,c;d))_{b,d \in S} \in \mathcal{M}_S(p)$ with $a \leq c$. That is why the dimension is the same as for $v$-reversible matrices.

\item If $A \in \triang{S}{p}$ is $h$-reversible, then $(T(a,b,c;d))_{b,d \in S} \in \mathcal{M}_S(p)$ and $T(a,d,c;b) = \frac{p(b)}{p(d)} T(a,b,c;d)$. Then, for any $a,c \in S$, $(T(a,b,c;d))_{b,d \in S} \in \mathcal{M}^{\sym}_S(p)$ and, moreover, they can be choosen freely. So, by Lemma~\ref{lem:div-sym},
\begin{displaymath}
\dim \left( \{A \in \triang{S}{p} : \text{$A$ is $h$-reversible}\} \right)  = |S|^2 \dim \mathcal{M}^{\sym}_S(p) = \frac{(n-1) n^3}{2}.
\end{displaymath}

\item If $A \in \triang{S}{p}$ is $<r^2,v>$-reversible, then $(T(a,b,c;d))_{b,d \in S} \in \mathcal{M}_S(p)$, $T(a,b,c;d) = T(c,b,a;d)$ and $T(a,d,c;b) = \frac{p(b)}{p(d)} T(a,b,c;d)$. Then, it is equivalent to choose freely $\{ (T(a,b,c;d))_{b,d \in S} : a \leq c \}$ in $\mathcal{M}^{\sym}_S(p)$. That's why,
\begin{align*}
\dim \left( \{A \in \triang{S}{p} : \text{$A$ is $<r^2,v>$-reversible}\} \right)
& = \left( |S| + \binom{|S|}{2} \right) \dim \mathcal{M}^{\sym}_S(p) \\
& = \frac{n(n+1)}{2} \frac{(n-1) n}{2} \\
& = \frac{(n-1)n^2(n+1)}{4}.
\end{align*}

\item[9,10,11.] Proofs are long and relatively similar. They are done in Section~\ref{annex:longdim}.
\end{enumerate}
\end{proof}

\begin{corollary} \label{cor:dimp}
  Let $S$ be a finite set of size $n$ 
  \begin{displaymath} 
    \dim \left( \cup_p \triang{S}{p} \right) = (n^3-n^2+1)(n-1).
  \end{displaymath}
\end{corollary}

\begin{proof}
  We just add to the previous result the dimension of the set of positive probability measures on $S$ that is $n-1$.
\end{proof}

\begin{remark}
If one prefer to know the dimension of the set of $D$-(quasi-)reversible PCA, for any $D \subset D_4$, it is sufficient to add $n-1$ to the result of Theorem~\ref{theo:dim} corresponding to this set to find the dimension as we have done in Corollary~\ref{cor:dimp}, . 
\end{remark}

A word about the dimension of the set of PCA having a $(F,B)$-HZMC invariant distribution. Let us denote $\triang{S}{(F,B)}$ this set.
\begin{proposition}
Let $S$ be a finite set of size $n$. For any $(F,B)$ such that $FB=BF$,
\begin{equation}
\dim \triang{S}{(F,B)} = n^2 (n-1)^2. 
\end{equation}
\end{proposition}

\begin{proof}
For any $(F,B)$, $A$ is in $\triang{S}{(F,B)}$ iff the two following conditions hold (see Prop.~\ref{prop:Casse17})
\begin{enumerate}
\item for any $a,b,c \in S$, $\sum_{d\in S} T(a,b,c;d) = 1$,
\item for any $a,c,d \in S$, 
\begin{displaymath}
\frac{F(a;d)B(d;c)}{(FB)(a;c)} = \sum_{b\in S} \frac{B(a;b)F(b;c)}{(FB)(a;c)} T(a,b,c;d).
\end{displaymath}
\end{enumerate}

Now, by Lemma~\ref{lem:dim-mp}, for any $u, v$,
\begin{displaymath}
\dim \mathcal{M}_S(u,v) = (n-1)^2.
\end{displaymath}

To conclude, we just have to say that, for any $a,c$, we can take freely $(T(a,b,c;d))_{b,d \in S} \in \mathcal{M}_S(u,v)$ with $\displaystyle u= \left( \frac{B(a;b)F(b;c)}{(FB)(a;c)} \right)_{b\in S}$ and $\displaystyle v = \left( \frac{F(a;d)B(d;c)}{(FB)(a;c)}\right)_{d \in S}$.
\end{proof}

But getting the dimension of $\displaystyle \cup_{\{(F,B) : FB=BF\}} \triang{S}{(F,B)}$ is complicated due to the fact that the set $\{(F,B) : FB=BF\}$ is not really well known yet even by algebraists, see~\cite{MT55,Gerstenhaber61,Guralnick92} for references on this subject. 

\subsection{Annex: proofs of points 9, 10 and 11 of Theorem~\ref{theo:dim}} \label{annex:longdim}

In this annex, let $S$ be any finite set, $s$ be any point on $S$ and let $p$ be any probability measure on $S$. We denote $S^* = S \backslash \{s\}$ and $n = |S|$.

The proofs of the last three points of Theorem~\ref{theo:dim} are long because they consist in reducing an affine system with $|S|^4$ equations and $|S|^4$ variables (containing some redundant equations) into one with only free equations describing the same manifold. Furthermore, we must ensure that there exists a solution with positive coefficients and that we are not in a degenerate case (see the discussion in the middle of the proof of Lemma~\ref{lem:dim-mp}). Since the proofs of the three points are similar, but not exactly the same, we first define some conditions that are useful for the three cases, then we detail the proof of point 9 and finally, we focus on the differences for the two other cases in comparison with the point 9.

\subsubsection{Preliminary results}

This section is technical and must be seen as a reference for the sections that are following, so it can be omitted in a first lecture.

First, we define some conditions on the transition kernel $T$.
\begin{cond} \label{cond:Ts}
For any $a,b,c,d \in S^*$, we have
\begin{equation} \label{eq:Ts-abcs}
T(a,b,c;s) = 1 - \sum_{d \in S^*} T(a,b,c;d);
\end{equation}
\begin{equation} \label{eq:Ts-sbcd}
T(s,b,c;d) = \frac{p(d)}{p(s)} - \sum_{a \in S^*} \frac{p(a)}{p(s)} T(a,b,c;d);
\end{equation}
\begin{equation} \label{eq:Ts-ascd}
T(a,s,c;d) = \frac{p(d)}{p(s)} - \sum_{b \in S^*} \frac{p(b)}{p(s)} T(a,b,c;d);
\end{equation}
\begin{equation} \label{eq:Ts-absd}
T(a,b,s;d) = \frac{p(d)}{p(s)} - \sum_{c \in S^*} \frac{p(c)}{p(s)} T(a,b,c;d);
\end{equation}
\begin{equation} \label{eq:Ts-sbcs}
T(s,b,c;s) = p(s) \left( 1 - \left(\frac{1-p(s)}{p(s)} \right)^2 \right) + \sum_{a \in S^*} \sum_{d \in S^*} \frac{p(a)}{p(s)}T(a,b,c;d);
\end{equation}
\begin{equation} \label{eq:Ts-ascs}
T(a,s,c;s) = p(s) \left( 1 - \left(\frac{1-p(s)}{p(s)} \right)^2 \right) + \sum_{b \in S^*} \sum_{d \in S^*} \frac{p(b)}{p(s)}T(a,b,c;d);
\end{equation}
\begin{equation} \label{eq:Ts-abss}
T(a,b,s;s) = p(s) \left( 1 - \left(\frac{1-p(s)}{p(s)} \right)^2 \right) + \sum_{c \in S^*} \sum_{d \in S^*} \frac{p(c)}{p(s)}T(a,b,c;d);
\end{equation}
\begin{equation} \label{eq:Ts-sscd}
T(s,s,c;d) =  p(d)  \left( 1 - \left(\frac{1-p(s)}{p(s)} \right)^2 \right) + \sum_{a \in S^*} \sum_{b \in S^*} \frac{p(a)p(b)}{p(s)^2}T(a,b,c;d);
\end{equation}
\begin{equation} \label{eq:Ts-sbsd}
T(s,b,s;d) =  p(d)  \left( 1 - \left(\frac{1-p(s)}{p(s)} \right)^2 \right) + \sum_{a \in S^*} \sum_{c \in S^*} \frac{p(a)p(c)}{p(s)^2}T(a,b,c;d);
\end{equation}
\begin{equation} \label{eq:Ts-assd}
T(a,s,s;d) =  p(d)  \left( 1 - \left(\frac{1-p(s)}{p(s)} \right)^2 \right) + \sum_{b \in S^*} \sum_{c \in S^*} \frac{p(b)p(c)}{p(s)^2}T(a,b,c;d);
\end{equation}
\begin{equation} \label{eq:Ts-sscs}
T(s,s,c;s) =  p(s) \left( 1 + \left(\frac{1-p(s)}{p(s)} \right)^3 \right) - \sum_{a \in S^*} \sum_{b \in S^*} \sum_{d \in S^*} \frac{p(a)p(b)}{p(s)^2}T(a,b,c;d);
\end{equation}
\begin{equation} \label{eq:Ts-sbss}
T(s,b,s;s) =  p(s) \left( 1 + \left(\frac{1-p(s)}{p(s)} \right)^3 \right) - \sum_{a \in S^*} \sum_{c \in S^*} \sum_{d \in S^*} \frac{p(a)p(c)}{p(s)^2}T(a,b,c;d);
\end{equation}
\begin{equation} \label{eq:Ts-asss}
T(a,s,s;s) =  p(s) \left( 1 + \left(\frac{1-p(s)}{p(s)} \right)^3 \right) - \sum_{b \in S^*} \sum_{c \in S^*} \sum_{d \in S^*} \frac{p(b)p(c)}{p(s)^2}T(a,b,c;d);
\end{equation}
\begin{equation} \label{eq:Ts-sssd}
T(s,s,s;d) = p(d) \left( 1 + \left(\frac{1-p(s)}{p(s)} \right)^3\right) - \sum_{a \in S^*} \sum_{b \in S^*} \sum_{c \in S^*} \frac{p(a)p(b)p(c)}{p(s)^3}T(a,b,c;d);
\end{equation}
\begin{equation} \label{eq:Ts-ssss}
T(s,s,s;s) =  p(s) \left( 1 - \left( \frac{1-p(s)}{p(s)} \right)^4 \right) + \sum_{a \in S^*} \sum_{b \in S^*} \sum_{c \in S^*} \sum_{d \in S^*} \frac{p(a)p(b)p(c)}{p(s)^3}T(a,b,c;d).
\end{equation}
\end{cond}

\begin{cond} \label{cond:prob}
For any $a,b,c \in S$, $\sum_{d\in S} T(a,b,c;d) = 1$.
\end{cond}

\begin{cond} \label{cond:vp}
For any $a,c,d \in S$, $\sum_{b\in S} p(b) T(a,b,c;d) = p(d)$.
\end{cond}

\begin{cond} \label{cond:prob2}
For any $a,b,c,d \in S$, $0< T(a,b,c;d) <1$.
\end{cond}

\begin{lemma} \label{lem:TsProb}
\C~\ref{cond:Ts} $\Rightarrow$ \C~\ref{cond:prob}
\end{lemma}

\begin{proof}
\begin{itemize}
\item For any $a,b,c \in S^*$,
\begin{align*}
\sum_{d \in S} T(a,b,c;d) & = T(a,b,c;s) + \sum_{d \in S^*} T(a,b,c;d) \\
& = 1 - \sum_{d \in S^*} T(a,b,c;d) +  \sum_{d \in S^*} T(a,b,c;d) \\
& = 1.
\end{align*}

\item For any $a,b \in S^*$,
\begin{align*}
\sum_{d \in S} T(a,b,s;d) & = T(a,b,s;s) + \sum_{d \in S^*} T(a,b,s;d) \\
& =  p(s) \left( 1 - \left(\frac{1-p(s)}{p(s)} \right)^2 \right) + \sum_{c \in S^*} \sum_{d \in S^*} \frac{p(c)}{p(s)}T(a,b,c;d) \\
& \qquad + \sum_{d \in S^*} \left(\frac{p(d)}{p(s)} - \sum_{c \in S^*} \frac{p(c)}{p(s)} T(a,b,c;d)\right) \\
& =  p(s) \left( 1 - \left(\frac{1-p(s)}{p(s)} \right)^2 \right)  + \frac{1-p(s)}{p(s)} \\
& =   p(s) - (1-p(s)) \frac{1-p(s)}{p(s)} + \frac{1-p(s)}{p(s)} \\
& =   p(s) + (1 - p(s)) =1. 
\end{align*}
\item Similarly, for any $a,b,c \in S^*$,
\begin{displaymath}
\sum_{d\in S} T(a,s,c;d) =1 \text{ and } \sum_{d\in S} T(s,b,c;d) =1. 
\end{displaymath}

\item For any $a \in S^*$,
\begin{align*}
\sum_{d \in S} T(a,s,s;d) & = T(a,s,s;s) + \sum_{d \in S^*} T(a,s,s;d) \\
& = p(s) \left( 1 + \left(\frac{1-p(s)}{p(s)} \right)^3 \right) - \sum_{b \in S^*} \sum_{c \in S^*} \sum_{d \in S^*} \frac{p(b)p(c)}{p(s)^2}T(a,b,c;d) \\
& \qquad + \sum_{d \in S^*} \left( p(d)  \left( 1 - \left(\frac{1-p(s)}{p(s)} \right)^2 \right) + \sum_{b \in S^*} \sum_{c \in S^*} \frac{p(b)p(c)}{p(s)^2}T(a,b,c;d) \right)\\
& = p(s) \left( 1 + \left(\frac{1-p(s)}{p(s)} \right)^3 \right)+ (1-p(s)) \left( 1 + \left(\frac{1-p(s)}{p(s)} \right)^2 \right) \\
& = p(s) + \frac{(1-p(s))^3}{p(s)^2} + 1 - p(s) - \frac{(1-p(s))^3}{p(s)^2}\\
& = 1.
\end{align*}
\item Similarly, for any $b,c \in S^*$,
\begin{displaymath}
\sum_{d\in S} T(s,b,s;d) =1 \text{ and } \sum_{d\in S} T(s,s,c;d) =1. 
\end{displaymath}

\item Finally,
\begin{align*}
\sum_{d \in S} T(s,s,s;d) & = T(s,s,s;s) + \sum_{d \in S^*} T(s,s,s;d) \\
& = p(s) \left( 1 - \left( \frac{1-p(s)}{p(s)} \right)^4 \right) + \sum_{a \in S^*} \sum_{b \in S^*} \sum_{c \in S^*} \sum_{d \in S^*} \frac{p(a)p(b)p(c)}{p(s)^3}T(a,b,c;d) \\
& \qquad + \sum_{d \in S^*} \left( p(d) \left( 1 + \left(\frac{1-p(s)}{p(s)} \right)^3\right) - \sum_{a \in S^*} \sum_{b \in S^*} \sum_{c \in S^*} \frac{p(a)p(b)p(c)}{p(s)^3}T(a,b,c;d) \right) \\
& = p(s) - \frac{(1-p(s))^4}{p(s)^3} + 1 - p(s) + \frac{(1-p(s))^4}{p(s)^3} \\
& = 1.
\end{align*}
\end{itemize}
\end{proof}

\subsubsection{Proof of 9 of Theorem~\ref{theo:dim} ($r$-reversible)}
To prove point 9, we define now the following condition:
\begin{cond} \label{cond:rrevS}
For any $a,b,c,d \in S$, $p(a) T(a,b,c;d) = p(d) T(b,c,d;a)$.
\end{cond}
Hence, by Theorem~\ref{theo:gen0} and~\ref{theo:rev},
\begin{align} 
& \dim \left( \{A \in \triang{S}{p} : \text{$A$ is $<r>$-reversible}\} \right) \nonumber \\
&  = \dim \{(T(a,b,c;d) : a,b,c,d \in S) : \C~\ref{cond:prob} + \C~\ref{cond:vp} + \C~\ref{cond:prob2} + \C~\ref{cond:rrevS}\} \label{eq:p5th-1}
\end{align}

Now, we define a condition similar to \C~\ref{cond:rrevS} but only on $S^*$:
\begin{condd} \label{cond:rrevS*}
For any $a,b,c,d \in S^*$, $p(a) T(a,b,c;d) = p(d) T(b,c,d;a)$.
\end{condd}

We have some properties that links all these previous conditions by the two following lemmas.
\begin{lemma}
(\C~\ref{cond:prob} + \C~\ref{cond:vp} + \C~\ref{cond:rrevS}) $\Leftrightarrow$ (\C~\ref{cond:prob} + \C~\ref{cond:rrevS})
\end{lemma}
\begin{proof}
$\Rightarrow$ is obvious. Now to prove $\Leftarrow$ we do the following computation: for any $a,c,d\in S$,
\begin{align*}
\sum_{b\in S} p(b) T(a,b,c;d) & = \sum_{b\in S} p(b) \frac{p(d)}{p(a)} T(b,c,d;a) \\
  & = \sum_{b\in S} p(b) \frac{p(d)}{p(a)} \frac{p(a)}{p(b)} T(c,d,a;b) = p(d).
\end{align*}
\end{proof}

\begin{lemma} \label{lem:equiv-5}
(\C~\ref{cond:prob} + \C~\ref{cond:rrevS}) $\Leftrightarrow$ (\C~\ref{cond:Ts} + \C~\ref{cond:rrevS}*)
\end{lemma}

\begin{proof}
This proof is algebraic.\par

$\Rightarrow$: Let suppose that $T$ satisfy \C~\ref{cond:prob} + \C~\ref{cond:rrevS}. Then \C~\ref{cond:rrevS}* obviously holds. Now, we will prove that \C~\ref{cond:Ts} holds too.
\begin{itemize}
\item For any $a,b,c \in S^*$, 
\begin{displaymath}
T(a,b,c;s) = 1 - \sum_{d \in S^*} T(a,b,c;d).
\end{displaymath}
\item For any $a,b,d \in S^*$,
\begin{displaymath}
T(a,b,s;d) = \frac{p(d)}{p(s)} T(d,a,b;s) = \frac{p(d)}{p(s)} \left(1 - \sum_{c \in S^*} T(d,a,b;c)\right) = \frac{p(d)}{p(s)} - \sum_{c \in S^*} \frac{p(c)}{p(s)} T(a,b,c;d).
\end{displaymath}
\item Similarly, we get~\eqref{eq:Ts-sbcd} and~\eqref{eq:Ts-ascd}.
\item For any $a,b \in S^*$, 
\begin{align*}
T(a,b,s;s) & = 1 - \sum_{d \in S^*} T(a,b,s;d) = 1 - \sum_{d \in S^*} \left(\frac{p(d)}{p(s)} - \sum_{c \in S^*} \frac{p(c)}{p(s)}T(a,b,c;d) \right)\\
& = \frac{2p(s)-1}{p(s)}  + \sum_{c \in S^*} \sum_{d \in S^*} \frac{p(c)}{p(s)}T(a,b,c;d)\\
& = p(s) - \frac{(1-p(s))^2}{p(s)}  + \sum_{c \in S^*} \sum_{d \in S^*} \frac{p(c)}{p(s)}T(a,b,c;d)\\
& = p(s) \left(1 - \left(\frac{1-p(s)}{p(s)}\right)^2 \right)  + \sum_{c \in S^*} \sum_{d \in S^*} \frac{p(c)}{p(s)}T(a,b,c;d).
\end{align*}
\item Similarly, we get~\eqref{eq:Ts-sbcs} and~\eqref{eq:Ts-ascs}.
\item For any $a,d \in S^*$,
\begin{align*}
T(a,s,s;d)  & = \frac{p(d)}{p(s)} T(d,a,s;s) = \frac{p(d)}{p(s)} \left( 2 - \frac{1}{p(s)}  + \sum_{b \in S^*} \sum_{c \in S^*} \frac{p(b)}{p(s)}T(d,a,b;c)\right)\\
& =  \frac{p(d)}{p(s)} \left( 2 -\frac{1}{p(s)} \right)  + \sum_{b \in S^*} \sum_{c \in S^*} \frac{p(b)p(c)}{p(s)^2}T(a,b,c;d)\\
& =  p(d)  \left( 1 - \left(\frac{1-p(s)}{p(s)} \right)^2 \right) + \sum_{b \in S^*} \sum_{c \in S^*} \frac{p(b)p(c)}{p(s)^2}T(a,b,c;d);
\end{align*}
\item Similarly, we get~\eqref{eq:Ts-sbsd} and~\eqref{eq:Ts-sscd}.
\end{itemize}
\begin{itemize}
\item For any $a \in S^*$, 
\begin{align*}
T(a,s,s;s) & = 1 - \sum_{d \in S^*} T(a,s,s;d) = 1 - \sum_{d \in S^*} \left( \frac{p(d)}{p(s)} \left( 2 - \frac{1}{p(s)}\right)  + \sum_{b \in S^*} \sum_{c \in S^*} \frac{p(b)p(c)}{p(s)^2}T(a,b,c;d) \right)\\
& = \frac{3 p(s)^2 - 3p(s) + 1}{p(s)^2} - \sum_{b \in S^*} \sum_{c \in S^*} \sum_{d \in S^*} \frac{p(b)p(c)}{p(s)^2}T(a,b,c;d)\\
& = p(s) \left( 1 + \left(\frac{1-p(s)}{p(s)} \right)^3 \right) - \sum_{b \in S^*} \sum_{c \in S^*} \sum_{d \in S^*} \frac{p(b)p(c)}{p(s)^2}T(a,b,c;d);
\end{align*}
\item Similarly, we get~\eqref{eq:Ts-sbss} and~\eqref{eq:Ts-sscs}.
\item For any $d \in S^*$, 
\begin{align*}
T(s,s,s;d) & = \frac{p(d)}{p(s)} T(s,s,d;s) = p(d) \left( 1 - \frac{(1-p(s))^3}{p(s)^3} \right) - \sum_{a \in S^*} \sum_{b \in S^*} \sum_{c \in S^*} \frac{p(b)p(c)p(d)}{p(s)^3}T(b,c,d;a)\\
& =  p(d) \left( 1 + \left(\frac{1-p(s)}{p(s)} \right)^3\right) - \sum_{a \in S^*} \sum_{b \in S^*} \sum_{c \in S^*} \frac{p(a)p(b)p(c)}{p(s)^3}T(a,b,c;d)
\end{align*}
\end{itemize}
and, finally,
\begin{align*}
T(s,s,s;s) & =  1 - \sum_{d \in S^*} T(s,s,s;d) \\
& = 1 - \sum_{d \in S^*} \left(p(d) \left( 1 + \frac{(1-p(s))^3}{p(s)^3} \right) - \sum_{a \in S^*} \sum_{b \in S^*} \sum_{c \in S^*} \frac{p(a)p(b)p(c)}{p(s)^3}T(a,b,c;d) \right)\\
& = 1 - (1-p(s)) \left( 1 + \frac{(1-p(s))^3}{p(s)^3} \right) + \sum_{a \in S^*} \sum_{b \in S^*} \sum_{c \in S^*} \sum_{d \in S^*} \frac{p(a)p(b)p(c)}{p(s)^3}T(a,b,c;d) \\
& = p(s) \left( 1 - \left( \frac{1-p(s)}{p(s)} \right)^4 \right) + \sum_{a \in S^*} \sum_{b \in S^*} \sum_{c \in S^*} \sum_{d \in S^*} \frac{p(a)p(b)p(c)}{p(s)^3}T(a,b,c;d).
\end{align*}
So, \C~\ref{cond:Ts} holds.

$\Leftarrow$: Now, suppose that \C~\ref{cond:Ts} and \C~\ref{cond:rrevS}* hold. \C~\ref{cond:prob} hold by Lemma~\ref{lem:TsProb}. We will prove that~\C~\ref{cond:rrevS} holds. It is obvious when $a,b,c,d \in S^*$ by \C~\ref{cond:rrevS}*. Furthermore, we have the following properties.
\begin{itemize}
\item For any $a,b,c \in S^*$,
  \begin{align*}
    p(a) T(a,b,c;s) & = p(a) \left(1 - \sum_{d \in S^*} T(a,b,c;d) \right)\\
    & = p(a) - \sum_{d\in S^*} p(a) T(a,b,c;d) \\
    & = p(s) \left( \frac{p(a)}{p(s)} - \sum_{d \in S^*} \frac{p(d)}{p(s)} T(b,c,d;a) \right) \\
    & = p(s) T(b,c,s,a).
  \end{align*}
\item For any $a,b,d \in S^*$,
  \begin{align*}
    p(a) T(a,b,s;d) & = \frac{p(a)p(d)}{p(s)} - \sum_{c \in S^*} \frac{p(c)}{p(s)} p(a) T(a,b,c;d) \\
    & = \frac{p(a)p(d)}{p(s)} - \sum_{c \in S^*} \frac{p(c)}{p(s)} p(d) T(b,c,d;a) \\
    & = p(d) \left( \frac{p(a)}{p(s)} - \sum_{c \in S^*} \frac{p(c)}{p(s)} T(b,c,d;a) \right) \\
    & = p(d) T(b,s,d;a).
  \end{align*}

\item Similarly, for any $a,c,d \in S^*$, $p(a) T(a,s,c;d) = p(d) T(s,c,d;a)$.  
\item For any $b,c,d \in S^*$,
  \begin{align*}
    p(s) T(s,b,c;d) & = p(d) - \sum_{a\in S^*} p(a) T(a,b,c;d)\\
                    & = p(d) (1 - \sum_{a \in S^*} T(b,c,d;a) \\
                    & = p(d) T(b,c,d;s).
  \end{align*}
\item For any $a,b \in S^*$,
  \begin{align*}
    p(a) T(a,b,s;s) & = p(a) p(s) \left(1 - \left(\frac{1-p(s)}{p(s)}\right)^2 \right) + \sum_{c \in S^*} \sum_{d \in S^*} \frac{p(c)}{p(s)} p(a) T(a,b,c;d)\\
                    & = p(a) p(s) \left(1 - \left(\frac{1-p(s)}{p(s)}\right)^2 \right) + \sum_{c \in S^*} \sum_{d \in S^*} \frac{p(c)}{p(s)} p(d) T(b,c,d;a)\\
                    & = p(s) \left( p(a) \left(1 - \left(\frac{1-p(s)}{p(s)}\right)^2 \right) + \sum_{c \in S^*} \sum_{d \in S^*} \frac{p(c)p(d)}{p(s)^2} T(b,c,d;a) \right)\\
                    & = p(s) T(b,s,s;a).
  \end{align*}
\item The other cases are similar and left to the readers.
\end{itemize}
That ends the proof.
\end{proof}

Due to Eq~\ref{eq:p5th-1} and the two preceding lemmas, we obtain:
\begin{equation}
\dim \left( \{A \in \triang{S}{p} : \text{$A$ is $<r>$-reversible}\} \right) \leq \dim \{(T(a,b,c;d) : a,b,c,d \in S^*) : \C~\ref{cond:rrevS}^*\}.
\end{equation}

Now, we compute $\dim \{(T(a,b,c;d) : a,b,c,d \in S^*) : \C~\ref{cond:rrevS}^*\}$ to find the upper bound.
\begin{lemma} \label{lem:dim-5*}
For any finite set $S$,
\begin{equation}
\dim \{(T(a,b,c;d) : a,b,c,d \in S^*) : \C~\ref{cond:rrevS}^*\} =  \frac{n (n-1) (n^2-3n+4)}{4}.
\end{equation}
\end{lemma}

\begin{proof}
This proof is half-algebraic and half-combinatorics. The goal is to use \C~\ref{cond:rrevS}* to split the set $\{T(a,b,c;d) : a,b,c,d \in S^*\}$ in some subsets such that variables in each subset depend of only one free parameter. The partition is the following one:
\begin{align*}
& \left \{ \left\{ T(i,i,i;i) \right\}: i \in S^* \right\} \\
\bigcup & \left \{ \left\{ T(i,i,i;j), T(i,i,j;i), T(i,j,i;i), T(j,i,i;i) \right\}: i,j \in S^*, i \neq j  \right\} \\
\bigcup & \left \{ \left\{ T(i,j,i;j), T(j,i,j;i) \right\}: i,j \in S^*, i<j  \right\} \\
\bigcup & \left \{ \left\{ T(i,i,j;j), T(i,j,j;i) , T(j,j,i;i) , T(j,i,i;j) \right\}: i,j \in S^*, i \neq j  \right\} \\
\bigcup & \left \{ \left\{ T(i,k,i;j), T(k,i,j;i) , T(i,j,i;k) , T(j,i,k;i) \right\}: i,j,k \in S^*, i \neq j,k , j<k  \right\} \\
\bigcup & \left \{ \left\{ T(i,i,j;k), T(i,j,k;i) , T(j,k,i;i) , T(k,i,i;j) \right\}: i,j,k \in S^*, i \neq j \neq k \neq i \right\} \\
\bigcup & \left \{ \left\{ T(a,b,c;d), T(b,c,d;a) , T(c,d,a;b) , T(d,a,c;b) \right\}: a,b,c,d \in S^*, a \neq b \neq c \neq d \neq a \neq c, d \neq b \right\}
\end{align*}

One can check that, in each subset of this partition, there is exactly only one free variable according to \C~\ref{cond:rrevS}*, see Table~\ref{table:partition1} to find the equations that link them. Now, the dimension is just the size of this partition. Enumeration is done in Table~\ref{table:partition1}. By adding the fourth column, we find
\begin{displaymath}
\dim \{(T(a,b,c;d) : a,b,c,d \in S^*) : \C~\ref{cond:rrevS}^*\} = \displaystyle \frac{n (n-1) (n^2-3n+4)}{4}.
\end{displaymath}
\end{proof}

\begin{table}
\begin{center}
\begin{tabular}{|c|c|c|c|}
\hline
Subset type & Involved equations & Conditions on & Number of subsets \\
& & the arguments & of this type \\
\hline
$\{T(i,i,i;i)\}$ & $T(i,i,i;i)$ & $i \in S^*$ & $|S^*| = n-1$ \\
\hline
$\begin{matrix} \{T(i,i,i;j), \\ T(i,i,j;i), \\ T(i,j,i;i),\\ T(j,i,i;i)\} \end{matrix}$ & $\begin{matrix} \phantom{=}\, p(i) T(i,i,i;j) \\ = p(j) T(i,i,j;i) \\ = p(j) T(i,j,i;i) \\ = p(j) T(j,i,i;i) \end{matrix}$ & $\begin{matrix} i,j \in S^* \\ i \neq j \end{matrix}$ & $\begin{matrix} \displaystyle \binom{|S^*|}{1} \binom{|S^*|-1}{1} \\ \displaystyle = (n-1) (n-2) \end{matrix}$ \\
\hline
$\begin{matrix} \{T(i,j,i;j),\\ T(j,i,j;i)\} \end{matrix}$ & $\begin{matrix} \phantom{=}\, p(i) T(i,j,i;j) \\ = p(j) T(j,i,j;i) \end{matrix}$ & $\begin{matrix} i,j \in S^* \\ i < j \end{matrix}$ & $\displaystyle \binom{|S^*|}{2} = \frac{(n-1)(n-2)}{2}$ \\
\hline
$\begin{matrix} \{T(i,i,j;j),\\ T(i,j,j;i),\\ T(j,j,i;i),\\ T(j,i,i;j)\} \end{matrix}$ & $\begin{matrix} \phantom{=}\, p(i) T(i,i,j;j) \\ = p(j) T(i,j,j;i) \\ = p(j) T(j,j,i;i) \\ = p(i) T(j,i,i;j) \end{matrix}$ & $\begin{matrix} i,j \in S^* \\ i \neq j \end{matrix}$ & $\displaystyle \binom{|S^*|}{2} = \frac{(n-1)(n-2)}{2}$ \\
\hline
$\begin{matrix} \{T(i,k,i;j),\\ T(k,i,j;i),\\ T(i,j,i;k),\\ T(j,i,k;i)\} \end{matrix}$ & $\begin{matrix} \phantom{=}\, p(i) p(k) T(i,k,i;j) \\ = p(j) p(k) T(k,i,j;i) \\ = p(i) p(j) T(i,j,i;k) \\ = p(j) p(k) T(j,i,k;i) \end{matrix}$ & $\begin{matrix} i,j,k \in S^* \\ i \neq j,k \\ j<k \end{matrix}$ & $\begin{matrix} \displaystyle \binom{|S^*|}{1} \binom{|S^*|-1}{2} \\  = \displaystyle \frac{(n-1)(n-2)(n-3)}{2} \end{matrix}$ \\
\hline
$\begin{matrix} \{T(i,i,j;k),\\ T(i,j,k;i),\\ T(k,j,i;i),\\ T(j,i,i;k)\} \end{matrix}$ & $\begin{matrix} \phantom{=}\, p(i) p(j) T(i,i,j;k) \\ = p(j) p(k) T(i,j,k;i) \\ = p(j) p(k) T(j,k,i;i) \\ = p(i) p(k) T(k,i,i;j) \end{matrix}$ & $\begin{matrix} i,j,k \in S^* \\ i \neq j \neq k \neq i \end{matrix}$ & $\begin{matrix} \displaystyle \binom{|S^*|}{1} \binom{|S^*-1|}{1} \binom{|S^*|-2}{1} \\ = (n-1) (n-2)(n-3) \end{matrix}$ \\
\hline
$\begin{matrix} \{T(a,b,c;d),\\ T(b,c,d;a),\\ T(c,d,a;b),\\ T(d,a,b;c)\} \end{matrix}$ & $\begin{matrix} \phantom{=}\, p(a) p(b) p(c) T(a,b,c;d) \\ = p(b) p(c) p(d) T(b,c,d;a) \\ = p(a) p(c) p(d) T(c,d,a;b) \\ = p(a) p(b) p(d) T(d,a,b;c) \end{matrix}$ & $\begin{matrix} a,b,c,d \in S^* \\ a < b,c,d \\ b \neq c \neq d \neq b \end{matrix}$ & $\begin{matrix} \displaystyle \frac{1}{4} |S^*| (|S^*|-1) (|S^*|-2) (|S^*|-3) \\ = \displaystyle \frac{(n-1)(n-2)(n-3)(n-4)}{4} \end{matrix}$ \\
\hline
\end{tabular}
\end{center}
\caption{Partition of $\{T(a,b,c;d) : a,b,c,d \in S^*\}$ according to~\C~\ref{cond:rrevS}*. On each line, we detail one of the type of the subset involved in the partition. The first column is the subset type. The second gives the equations that link the variables in the subset; these equations are obtained by specifications of \C~\ref{cond:rrevS}*. The third column gives conditions on the arguments to get independent sets when we enumerate them. The fourth column is the enumeration of subsets of that type.} \label{table:partition1}
\end{table}

To get the lower bound for $\dim \left( \{A \in \triang{S}{p} : \text{$A$ is $<r>$-reversible}\} \right)$, we use a similar trick that we have done in the proof of Lemma~\ref{lem:dim-mp}. We first remark that $T(a,b,c;d) = p(d)$ is a solution and, then by all the previous equations, it is not difficult to construct a neighboorhood whose dimension is $\dim \{(T(a,b,c;d) : a,b,c,d \in S^*) : \C~\ref{cond:rrevS}^*\}$ and for which we do not lose positivities of $T(a,b,c;d)$ for any $(a,b,c,d) \in S^4$. Then, we get that 
\begin{equation}
\dim \left( \{A \in \triang{S}{p} : \text{$A$ is $<r>$-reversible}\} \right) \geq \dim \{(T(a,b,c;d) : a,b,c,d \in S^*) : \C~\ref{cond:rrevS}^*\}.
\end{equation}
That ends the proof of point 9 of Theorem~\ref{theo:dim}. \qed\par

\subsubsection{Proof of 10 of Theorem~\ref{theo:dim} ($r \circ v$-reversible)}

The proof of 10 is similar to the one of 9. Hence, we will omit some parts of the proof that are the same. We only detail the partition in Lemma~\ref{lem:dim-6*}, because it differs from the one of Lemma~\ref{lem:dim-5*}.

The conditions we will need here are the two following ones:
\begin{cond} \label{cond:rvrevS}
For any $a,b,c,d \in S$, $p(a) T(a,b,c;d) = p(d) T(d,c,b;a)$.
\end{cond}
\begin{condd} \label{cond:rvrevS*}
For any $a,b,c,d \in S^*$, $p(a) T(a,b,c;d) = p(d) T(d,c,b;a)$.
\end{condd}

That are linked by the following lemma:
\begin{lemma} \label{lem:equiv-6}
(\C~\ref{cond:prob} + \C~\ref{cond:vp} + \C~\ref{cond:rvrevS}) $\Leftrightarrow$ (\C~\ref{cond:Ts} + \C~\ref{cond:rvrevS}*)
\end{lemma}

\begin{proof}
Proof is similar to the one of Lemma~\ref{lem:equiv-5}.
\end{proof}

Hence, by Theorems~\ref{theo:gen0} and~\ref{theo:rev} and Lemma~\ref{lem:equiv-6}, we get
\begin{displaymath}
\dim \left( \{A \in \triang{S}{p} : \text{$A$ is $<r \circ v>$-reversible}\} \right) \leq \dim \{(T(a,b,c;d) : a,b,c,d \in S^*) : \C~\ref{cond:rvrevS}^*\}.
\end{displaymath}

Now, we compute the upper bound:
\begin{lemma} \label{lem:dim-6*}
For any finite set $S$:
\begin{equation}
\dim \{(T(a,b,c;d) : a,b,c,d \in S^*) : \C~\ref{cond:rvrevS}^*\} =  \frac{(n-1)^2 (n^2-2n+2)}{2}.
\end{equation}
\end{lemma}

\begin{proof}
Proof is similar to the one of Lemma~\ref{lem:dim-5*}, except that the variable space is not partitioned in the same way. The new partition (based on \C~\ref{cond:rvrevS}*) and its enumeration is given in the Table~\ref{table:part10}. Thus, the size of this partition is $\displaystyle \frac{(n-1)^2 (n^2-2n+2)}{2}$.

\begin{table}
\begin{center}
\begin{tabular}{|c|c|c|c|}
\hline
Subset type & Involved equations & Conditions on & Number of subsets \\
& & the arguments & of this type \\
\hline
$\{T(i,i,i;i)\}$ & $T(i,i,i;i)$ & $i\in S^*$ & $|S^*| = (n-1)$ \\
\hline
$\begin{matrix} \{T(i,i,i;j), \\ T(j,i,i;i)\} \end{matrix}$ & $\begin{matrix} \phantom{=} \, p(i) T(i,i,i;j) \\ = p(j) T(j,i,i;i) \end{matrix}$ & $\begin{matrix} i,j \in S^* \\ i \neq j \end{matrix}$ & $\begin{matrix} \displaystyle \binom{|S^*|}{1} \binom{|S^*|-1}{1} \\ \displaystyle = (n-1) (n-2) \end{matrix}$ \\
\hline
$\begin{matrix} \{T(i,i,j;i), \\ T(i,j,i;i)\} \end{matrix}$ & $\begin{matrix} \phantom{=}\, p(i) T(i,i,j;i) \\ = p(j) T(i,j,i;i) \end{matrix}$ & $\begin{matrix} i,j \in S^* \\ i \neq j \end{matrix}$ & $\begin{matrix} \displaystyle \binom{|S^*|}{1} \binom{|S^*|-1}{1} \\ \displaystyle = (n-1)(n-2) \end{matrix}$ \\
\hline
$\begin{matrix} \{T(i,j,i;j),\\ T(j,i,j;i)\} \end{matrix}$ & $\begin{matrix} \phantom{=}\, p(i) T(i,j,i;j) \\ = p(j) T(j,i,j;i) \end{matrix}$ & $\begin{matrix} i,j \in S^* \\ i < j \end{matrix}$ & $\displaystyle \binom{|S^*|}{2} = \frac{(n-1)(n-2)}{2}$ \\
\hline
$\begin{matrix} \{T(i,i,j;j),\\ T(j,j,i;i) \} \end{matrix}$ & $\begin{matrix} \phantom{=}\, p(i) T(i,i,j;j) \\ = p(j) T(j,j,i;i) \end{matrix}$ & $\begin{matrix} i,j \in S^* \\ i < j \end{matrix}$ & $\displaystyle \binom{|S^*|}{2} = \frac{(n-1)(n-2)}{2}$ \\
\hline
$\{T(i,j,j;i) \}$ & $T(i,j,j;i)$ & $\begin{matrix} i,j \in S^* \\ i \neq j \end{matrix}$ & $\begin{matrix} \displaystyle \binom{|S^*|}{1} \binom{|S^*|-1}{1} \\ \displaystyle = (n-1) (n-2) \end{matrix}$ \\
\hline
$\begin{matrix} \{T(i,i,j;k),\\ T(k,j,i;i)\} \end{matrix}$ & $\begin{matrix} \phantom{=}\, p(i) T(i,i,j;k) \\ = p(k) T(k,j,i;i) \end{matrix}$ & $\begin{matrix} i,j,k \in S^* \\ i \neq j \neq k \neq i \end{matrix}$ & $\begin{matrix} \displaystyle \binom{|S^*|}{1} \binom{|S^*|-1}{1} \binom{|S^*|-2}{1} \\  = (n-1) (n-2) (n-3) \end{matrix}$ \\
\hline
$\begin{matrix} \{T(i,j,i;k),\\ T(k,i,j;i)\} \end{matrix}$ & $\begin{matrix} \phantom{=}\, p(i) T(i,j,i;k) \\ = p(k) T(k,i,j;i) \end{matrix}$ & $\begin{matrix} i,j,k \in S^* \\ i \neq j \neq k \neq i \end{matrix}$ & $\begin{matrix} \displaystyle \binom{|S^*|}{1} \binom{|S^*|-1}{1} \binom{|S^*|-2}{1} \\  = (n-1) (n-2) (n-3) \end{matrix}$ \\
\hline
$\begin{matrix} \{T(i,j,k;i),\\ T(i,k,j;i)\} \end{matrix}$ & $\begin{matrix} \phantom{=}\, T(i,j,k;i) \\ = T(i,k,j;i) \end{matrix}$ & $\begin{matrix} i,j,k \in S^* \\ i \neq j,k \\ j< k\end{matrix}$ & $\begin{matrix} \displaystyle \binom{|S^*|}{1} \binom{|S^*|-1}{2} \\  \displaystyle = \frac{(n-1)(n-2)(n-3)}{2} \end{matrix}$ \\
\hline
$\begin{matrix} \{T(j,i,i;k),\\ T(k,i,i;j)\} \end{matrix}$ & $\begin{matrix} \phantom{=}\, p(j) T(j,i,i;k) \\ = p(k) T(k,i,i;j) \end{matrix}$ & $\begin{matrix} i,j,k \in S^* \\ i \neq j,k \\ j< k\end{matrix}$ & $\begin{matrix} \displaystyle \binom{|S^*|}{1} \binom{|S^*|-1}{2} \\  \displaystyle =  \frac{(n-1)(n-2)(n-3)}{2} \end{matrix}$ \\
\hline
$\begin{matrix} \{T(a,b,c;d),\\ T(d,c,b;a)\} \end{matrix}$ & $\begin{matrix} \phantom{=}\, p(a) T(a,b,c;d) \\ = p(d) T(d,c,b;a) \end{matrix}$ & $\begin{matrix} a,b,c,d \in S^* \\ a < d \\ a \neq b \neq c \neq a \\ d \neq b,c \end{matrix}$ & $\begin{matrix} \displaystyle \frac{1}{2} |S^*| (|S^*|-1) (|S^*|-2) (|S^*|-3) \\ \displaystyle =  \frac{(n-1)(n-2)(n-3)(n-4)}{2} \end{matrix}$ \\
\hline
\end{tabular}
\end{center}
\caption{Partition of $\{T(a,b,c;d) : a,b,c,d \in S^*\}$ according to~\C~\ref{cond:rvrevS}*. 
} \label{table:part10}
\end{table}

The end of the proof is like the ones of Lemma~\ref{lem:dim-mp} and~\ref{lem:dim-5*}. It consists in checking that there exists a neighbourhood of the point $(T(a,b,c;d) = p(d) : a,b,c,d \in S)$ with the good dimension such that any point of this neighbourhood satisfies the required conditions.
\end{proof}

\subsubsection{Proof of 11 of Theorem~\ref{theo:dim} ($D_4$-reversible)}
The proof of point 11 is similar to the two previous ones. We begin by introducing the two new following conditions.
\begin{cond} \label{cond:d4revS}
For any $a,b,c,d \in S$, $T(a,b,c;d) = T(c,b,a;d)$.
\end{cond}

\begin{condd} \label{cond:d4revS*}
For any $a,b,c,d \in S^*$, $T(a,b,c;d) = T(c,b,a;d)$.
\end{condd}

We have then the following relation.
\begin{lemma} \label{lem:equiv-7}
(\C~\ref{cond:prob} + \C~\ref{cond:rrevS} + \C~\ref{cond:d4revS}) $\Leftrightarrow$ (\C~\ref{cond:Ts} + \C~\ref{cond:rrevS}* + \C~\ref{cond:d4revS}*)
\end{lemma}

By Theorem~\ref{theo:gen0} and~\ref{theo:rev} and Lemma~\ref{lem:equiv-7}, we have
\begin{align*}
& \dim \left( \{A \in \triang{S}{p} : \text{$A$ is $D_4$-reversible}\} \right) \\
& \leq \dim \{(T(a,b,c;d) : a,b,c,d \in S^*) : \C~\ref{cond:rrevS}^*+ \C~\ref{cond:d4revS}^*\}.
\end{align*}

Now, we compute the dimension.
\begin{lemma}
\begin{equation}
\dim \{(T(a,b,c;d) : a,b,c,d \in S^*) : \C~\ref{cond:rrevS}^* + \C~\ref{cond:d4revS}^*\} =  \frac{(n-1)^2 (n^2-2n+2)}{2}.
\end{equation}
\end{lemma}

\begin{proof}
As before, the main argument is to find the partition of $T$ based on~\C~\ref{cond:rrevS}* and \C~\ref{cond:d4revS}*. This partition and its enumeration is given in Table~\ref{table:part11}. Thus, the size of this partition is $\displaystyle \frac{n (n-1) (n^2-n+2)}{8}$.

\begin{table}
\begin{center}
\begin{tabular}{|c|c|c|c|}
\hline
Subset type & Involved equations & Conditions on & Number of subsets \\
& & the arguments & of this type \\
\hline
$\{T(i,i,i;i)\}$ & $T(i,i,i;i)$ & $i\in S^*$ & $|S^*|$ \\
\hline
$\begin{matrix} \{T(i,i,i;j), \\ T(i,i,j;i), \\ T(i,j,i;i),\\ T(j,i,i;i)\} \end{matrix}$ & $\begin{matrix} \phantom{=}\, p(i) T(i,i,i;j) \\ = p(j) T(i,i,j;i) \\ = p(j) T(i,j,i;i) \\ = p(j) T(j,i,i;i) \end{matrix}$ & $\begin{matrix} i,j \in S^* \\ i \neq j \end{matrix}$ & $\begin{matrix} \displaystyle \binom{|S^*|}{1} \binom{|S^*|-1}{1} \\ \displaystyle = (n-1) (n-2) \end{matrix}$ \\
\hline
$\begin{matrix} \{T(i,j,i;j),\\ T(j,i,j;i)\} \end{matrix}$ & $\begin{matrix} \phantom{=}\, p(i) T(i,j,i;j) \\ = p(j) T(j,i,j;i) \end{matrix}$ & $\begin{matrix} i,j \in S^* \\ i < j \end{matrix}$ & $\displaystyle \binom{|S^*|}{2} = \frac{(n-1)(n-2)}{2}$ \\
\hline
$\begin{matrix} \{T(i,i,j;j),\\ T(i,j,j;i),\\ T(j,j,i;i),\\ T(j,i,i;j)\} \end{matrix}$ & $\begin{matrix} \phantom{=}\, p(i) T(i,i,j;j) \\ = p(j) T(i,j,j;i) \\ = p(j) T(j,j,i;i) \\ = p(i) T(j,i,i;j) \end{matrix}$ & $\begin{matrix} i,j \in S^* \\ i < j \end{matrix}$ & $\displaystyle \binom{|S^*|}{2} = \frac{(n-1)(n-2)}{2}$ \\
\hline
$\begin{matrix} \{T(i,i,j;k),\\ T(k,i,i;j),\\ T(j,k,i,i), \\ T(i,j,k;i),\\ T(j,i,i;k),\\ T(i,i,k;j),\\ T(i,k,j;i),\\ T(k,j,i;i) \} \end{matrix}$ & $\begin{matrix} \phantom{=}\, p(i) p(j) T(i,i,j;k) \\ = p(j) p(k) T(i,j,k;i) \\ = p(j) p(k) T(j,k,i;i) \\ = p(i) p(k) T(k,i,i;j) \\ = p(i) p(j) T(j,i,i;k) \\ = p(j) p(k) T(k,j,i;i) \\ = p(j) p(k) T(i,k,j;i) \\ = p(i) p(k) T(i,i,k;j) \end{matrix}$ & $\begin{matrix} i,j,k \in S^* \\ i \neq j,k \\ j<k \end{matrix}$ & $\begin{matrix} \displaystyle \binom{|S^*|}{1} \binom{|S^*|-1}{2} \\  = \displaystyle \frac{(n-1)(n-2)(n-3)}{2} \end{matrix}$ \\
\hline
$\begin{matrix} \{T(i,j,i;k),\\ T(k,i,j;i),\\ T(i,k,i;j),\\ T(j,i,k;i)\} \end{matrix}$ & $\begin{matrix} \phantom{=}\, p(i) p(k) T(i,k,i;j) \\ = p(j) p(k) T(k,i,j;i) \\ = p(i) p(j) T(i,j,i;k) \\ = p(j) p(k) T(j,i,k;i) \end{matrix}$ & $\begin{matrix} i,j,k \in S^* \\ i \neq j,k \\ j<k \end{matrix}$ & $\begin{matrix} \displaystyle \binom{|S^*|}{1} \binom{|S^*-1|}{2} \\ = \displaystyle \frac{(n-1)(n-2)(n-3)}{2} \end{matrix}$ \\
\hline
$\begin{matrix} \{T(a,b,c;d),\\ T(d,a,b;c),\\ T(c,d,a;b),\\ T(b,c,d;a), \\ T(c,b,a;d),\\ T(b,a,d;c),\\ T(a,d,c;b),\\ T(d,c,b;a) \} \end{matrix}$ & $\begin{matrix} \phantom{=}\, p(a) p(b) p(c) T(a,b,c;d) \\ = p(b) p(c) p(d) T(b,c,d;a) \\ = p(a) p(c) p(d) T(c,d,a;b) \\ = p(a) p(b) p(d) T(d,a,b;c) \\ = p(a) p(b) p(c) T(c,b,a;d) \\ = p(b) p(c) p(d) T(d,c,b;a) \\ = p(a) p(c) p(d) T(a,d,c;b) \\ = p(a) p(b) p(d) T(b,a,d;c) \end{matrix}$ & $\begin{matrix} a,b,c,d \in S^* \\ a < b,c,d \\ b<c,d \\ c \neq d \end{matrix}$ & $\begin{matrix} \displaystyle \frac{1}{8} |S^*| (|S^*|-1) (|S^*|-2) (|S^*|-3) \\ = \displaystyle \frac{(n-1)(n-2)(n-3)(n-4)}{8} \end{matrix}$ \\
\hline
\end{tabular}
\end{center}
\caption{Partition of $\{T(a,b,c;d) : a,b,c,d \in S^*\}$ according to~\C~\ref{cond:rrevS}* and \C~\ref{cond:d4revS}*.} \label{table:part11}
\end{table}

To prove equality between $\dim \left( \{A \in \triang{S}{p} : \text{$A$ is $D_4$-reversible}\} \right)$ and $\dim \{(T(a,b,c;d) : a,b,c,d \in S^*) : \C~\ref{cond:rrevS}^*+ \C~\ref{cond:d4revS}^*\}$, we use the same trick as developed in the end of the proof of Lemma~\ref{lem:dim-mp}.
\end{proof}


\newpage

\begin{thebibliography}{10}

\bibitem{AP15}
T.~Antunovi\'c and E.~B. Procaccia.
\newblock Stationary {E}den model on {C}ayley graphs.
\newblock {\em Ann. Appl. Probab.}, 27(1):517--549, 2017.

\bibitem{Baxter72}
R.~J. Baxter.
\newblock Partition function of the eight-vertex lattice model.
\newblock {\em Annals of Physics}, 70(1):193--228, 1972.

\bibitem{Baxter82}
R.~J. Baxter.
\newblock {\em Exactly solved models in statistical mechanics}.
\newblock London: Academic Press Inc., 1982.

\bibitem{belitsky}
V.~Belitsky and P.~A. Ferrari.
\newblock Invariant measures and convergence properties for cellular automaton
  184 and related processes.
\newblock {\em Journal of Statistical Physics}, 118(3-4):589--623, 2005.

\bibitem{belyaev}
Yu.K. Belyaev, Yu.I. Gromak, and V.A. Malyshev.
\newblock {Invariant random Boolean fields (in Russian).}
\newblock {\em Mat. Zametki}, 6:555--566, 1969.

\bibitem{BousquetMelou98}
M.~Bousquet-M{\'e}lou.
\newblock New enumerative results on two-dimensional directed animals.
\newblock {\em Discrete Mathematics}, 180(1-3):73--106, 1998.

\bibitem{bmm_aap}
A.~Bu{\v s}i{\'c}, J.~Mairesse, and I.~Marcovici.
\newblock Probabilistic cellular automata, invariant measures, and perfect
  sampling.
\newblock {\em Adv. in Appl. Probab.}, 45(4):960--980, 2013.

\bibitem{Casse17}
J.~Casse.
\newblock Correlation function of the 8-vertex model with free boundary
  conditions using triangular probabilistic cellular automata.
\newblock {\em arXiv preprint arXiv:1607.05030}, 2016.

\bibitem{Casse16}
J.~Casse.
\newblock Probabilistic cellular automata with general alphabets possessing a
  {M}arkov chain as an invariant distribution.
\newblock {\em Adv. in Appl. Probab.}, 48(2):369--391, 2016.

\bibitem{CM15}
J.~Casse and J.-F. Marckert.
\newblock Markovianity of the invariant distribution of probabilistic cellular
  automata on the line.
\newblock {\em Stochastic processes and their applications}, 125(9):3458--3483,
  2015.

\bibitem{Dhar82}
D.~Dhar.
\newblock Equivalence of the two-dimensional directed-site animal problem to
  {B}axter's hard-square lattice-gas model.
\newblock {\em Phys. Rev. Lett.}, 49(14):959--962, 1982.

\bibitem{dobrushin}
R.~L. Dobrushin, V.~I. Kryukov, and A.~L. Toom.
\newblock {\em Stochastic cellular systems: ergodicity, memory, morphogenesis}.
\newblock Nonlinear science. Manchester University Press, 1990.

\bibitem{DGHMT16}
H.~Duminil-Copin, M.~Gagnebin, M.~Harel, I.~Manolescu, and V.~Tassion.
\newblock The {B}ethe ansatz for the six-vertex and xxz models: an exposition.
\newblock {\em arXiv preprint arXiv:1611.09909}, 2016.

\bibitem{Durrett10}
R.~Durrett.
\newblock {\em Probability: theory and examples}.
\newblock Cambridge university press, 2010.

\bibitem{Eden61}
M.~Eden.
\newblock A two-dimensional growth process.
\newblock {\em Dynamics of fractal surfaces}, 4:223--239, 1961.

\bibitem{FW70}
C.~Fan and F.~Y. Wu.
\newblock General lattice model of phase transitions.
\newblock {\em Physical Review B}, 2(3):723, 1970.

\bibitem{Gerstenhaber61}
M.~Gerstenhaber.
\newblock On dominance and varieties of commuting matrices.
\newblock {\em Annals of mathematics}, pages 324--348, 1961.

\bibitem{GKLM}
S.~Goldstein, R.~Kuik, J.~L. Lebowitz, and C.~Maes.
\newblock From {PCA}s to equilibrium systems and back.
\newblock {\em Comm. Math. Phys.}, 125(1):71--79, 1989.

\bibitem{gray}
L.~Gray and D.~Griffeath.
\newblock The ergodic theory of traffic jams.
\newblock {\em J. Statist. Phys.}, 105(3-4):413--452, 2001.

\bibitem{Guralnick92}
R.~M. Guralnick.
\newblock A note on commuting pairs of matrices.
\newblock {\em Linear and Multilinear Algebra}, 31(1-4):71--75, 1992.

\bibitem{hmm}
A.~E. Holroyd, I.~Marcovici, and J.~B. Martin.
\newblock Percolation games, probabilistic cellular automata, and the hard-core
  model.
\newblock {\em arXiv:1210660}, 2015.

\bibitem{kozlov}
O.~Kozlov and N.~Vasilyev.
\newblock Reversible {M}arkov chains with local interaction.
\newblock In {\em Multicomponent random systems}, volume~6 of {\em Adv. Probab.
  Related Topics}, pages 451--469. Dekker, New York, 1980.

\bibitem{LBM07}
Y.~Le~Borgne and J.-F. Marckert.
\newblock Directed animals and gas models revisited.
\newblock {\em Journal of Combinatorics}, 14(4):R71, 2007.

\bibitem{LMS}
J.~L. Lebowitz, C.~Maes, and E.~R. Speer.
\newblock Statistical mechanics of probabilistic cellular automata.
\newblock {\em J. Statist. Phys.}, 59(1-2):117--170, 1990.

\bibitem{mm_tcs}
J.~Mairesse and I.~Marcovici.
\newblock Around probabilistic cellular automata.
\newblock {\em Theoret. Comput. Sci.}, 559:42--72, 2014.

\bibitem{mm_ihp}
J.~Mairesse and I.~Marcovici.
\newblock Probabilistic cellular automata and random fields with i.i.d.
  directions.
\newblock {\em Ann. Inst. Henri Poincar\'e Probab. Stat.}, 50(2):455--475,
  2014.

\bibitem{marco_cie}
I.~Marcovici.
\newblock Ergodicity of noisy cellular automata: The coupling method and
  beyond.
\newblock In A.~Beckmann, L.~Bienvenu, and N.~Jonoska, editors, {\em Pursuit of
  the Universal - 12th Conference on Computability in Europe, CiE 2016, Paris,
  France, June 27 - July 1, 2016, Proceedings}, volume 9709 of {\em Lecture
  Notes in Computer Science}, pages 153--163. Springer, 2016.

\bibitem{MT55}
T.~S. Motzkin and O.~Taussky.
\newblock Pairs of matrices with property {$L$}. {II}.
\newblock {\em Transactions of the American mathematical society},
  80(2):387--401, 1955.

\bibitem{Pauling35}
L.~Pauling.
\newblock The structure and entropy of ice and of other crystals with some
  randomness of atomic arrangement.
\newblock {\em Journal of the American Chemical Society}, 57(12):2680--2684,
  1935.

\bibitem{Sutherland70}
B.~Sutherland.
\newblock Two-dimensional hydrogen bonded crystals without the ice rule.
\newblock {\em Journal of Mathematical Physics}, 11(11):3183--3186, 1970.

\bibitem{vasilyev}
N.~B. Vasilyev.
\newblock Bernoulli and {M}arkov stationary measures in discrete local
  interactions.
\newblock In {\em Developments in statistics, {V}ol. 1}, pages 99--112.
  Academic Press, New York, 1978.

\end{thebibliography}
\end{document}